\documentclass[reqno,a4paper,11pt]{amsart}
\usepackage{amsmath,amsthm,amssymb,mathrsfs}
\usepackage[
centering,
left = 2.4cm,
right = 2.4cm,
top = 2.6 cm,
bottom = 2.2 cm
]{geometry}
\usepackage{graphicx,mathtools}

\usepackage[svgnames]{xcolor}

\usepackage{todonotes}
\makeatletter
\define@key{todonotes}{Mingwen}[]{%
	\setkeys{todonotes}{author=\textbf{Mingwen},inline,color=red!20}}%
\define@key{todonotes}{Xiang}[]{%
	\setkeys{todonotes}{author=\textbf{Xiang},inline,color=green!20}}%
\define@key{todonotes}{Yadong}[]{%
	\setkeys{todonotes}{author=\textbf{Yadong},inline,color=cyan!20}}%
\makeatother

\usepackage{tikz}

\usepackage{bm,stmaryrd}
\usepackage{enumitem}
\usepackage[toc,page]{appendix}
\usepackage[
pdfencoding=auto,
colorlinks = true,
citecolor = blue,
linkcolor = blue,
anchorcolor = blue,
urlcolor = blue
]{hyperref}

\theoremstyle{plain}
\newtheorem{theorem}{Theorem}[section]
\newtheorem{corollary}[theorem]{Corollary}
\newtheorem{definition}[theorem]{Definition}
\newtheorem{lemma}[theorem]{Lemma}

\newtheorem{assumption}[theorem]{Assumption}
\newtheorem{remark}[theorem]{Remark}

\numberwithin{equation}{section}

\newcommand{\bb}[1]{\mathbb{#1}}

\newcommand{\bbt}{\bb{T}}

\newcommand{\bu}{\mathbf{u}}
\newcommand{\bv}{\mathbf{v}}

\newcommand{\bg}{\mathbf{g}}

\renewcommand{\d}{\mathrm{d}}
\newcommand{\dx}{\d x}
\newcommand{\dxdt}{\,\d x \d t}

\newcommand{\dt}{\d t}

\newcommand{\dtau}{\,\d\tau}

\newcommand{\abs}[1]{\left\vert #1 \right \vert}
\newcommand{\absm}[1]{\vert #1 \vert}
\newcommand{\norm}[1]{\left\Vert #1 \right \Vert}

\newcommand{\inner}[2]{\left\langle #1 , #2 \right\rangle} 

\allowdisplaybreaks[4]

\title[Quasi-incompressible Navier--Stokes/Cahn--Hilliard]{Weak solutions and incompressible limit of a quasi-incompressible Navier--Stokes/Cahn--Hilliard model for viscous two-phase flows}

\author{Mingwen Fei}
\address{School of  Mathematics and Statistics, Anhui Normal University, Wuhu 241002, P. R. China}
\email{mwfei@ahnu.edu.cn}
\author{Xiang Fei}
\address{School of  Mathematics and Statistics, Anhui Normal University, Wuhu 241002, P. R. China}
\email{feixiang@ahnu.edu.cn}
\author{Daozhi Han}
\address{Department of Mathematics, The State University of New York at Buffalo, Buffalo, NY 14260 USA}
\email{daozhiha@buffalo.edu}
\author{Yadong Liu}
\address{School of Mathematical Sciences and Ministry of Education Key Laboratory of NSLSCS, Nanjing Normal University, Nanjing 210023, P. R. China}
\email{ydliu@njnu.edu.cn}

\date{\today}

\subjclass[2020]{Primary:
	35Q35; 
	Secondary:
	76T06, 
	76T99, 
	35D30, 
	35B25, 
	35Q30
}
\keywords{Navier--Stokes/Cahn--Hilliard, quasi-incompressible, two-phase flows, weak solutions, incompressible limit, relative entropy method}

\begin{document}
	\begin{abstract}
		We study a quasi-incompressible Navier--Stokes/Cahn--Hilliard coupled system which describes the motion of two macroscopically immiscible incompressible viscous fluids with partial mixing in a small interfacial region and long-range interactions. The case of unmatched densities with mass-averaged velocity is considered so that the velocity field is no longer divergence-free, and the pressure enters the equation of the chemical potential.
		We first prove the existence of global weak solutions to the model in a three-dimensional periodic domain, for which the implicit time discretization together with a fixed-point argument to the approximate system is employed. In particular, we obtain a new regularity estimate of the order parameter by exploiting the partial damping effect of the capillary force. Then utilizing the relative entropy method, we establish the incompressible limit---the quasi-incompressible two-phase model converges to model H as the density difference tends to zero. Crucial to the passage of the incompressible limit, due to the lack of regularity of the pressure, are some non-standard uniform-in-density difference controls of the pressure, which are derived from the structure of the momentum equations and the improved regularity of the order parameter.
	\end{abstract}
	
	\maketitle
	
	
	\section{Introduction}
	\label{sec:introduction}
	
	In this article, we study a  \textit{diffuse interface model} (also referred to as \textit{phase-field
		model}) for two incompressible, immiscible fluids of different densities and viscosities. Let $T > 0$, $Q_T \coloneqq \bbt^3 \times (0,T)$.
	The model consists of the quasi-incompressible Navier--Stokes equations coupled with the Cahn--Hilliard equation in $Q_T$:
	\begin{subequations}
		\label{model2}
		\begin{align}
			\label{model2-1}
			\partial_t(\rho  \mathbf{u}) + \mathrm{div}\left (\rho  \mathbf{u}\otimes
			\mathbf{u}\right)& =  \mathrm{div}\big(S(\phi,\mathbb{ D}\mathbf{u})\big)-(\nabla p
			+\phi \nabla \mu), \\
			\label{model2-2}
			\mathrm{div}  \mathbf{u} & = \alpha \Delta\mu_p, \\
			\label{model2-3}
			\partial_t\phi + \mathrm{div}(\phi  \mathbf{u})& =\Delta\mu_p, \\
			\label{model2-4}
			\mu& =F'(\phi)+\Lambda^{2s} \phi, \\
			\label{model2-5}
			\mu_p& =\mu+\alpha p, \\  
			\label{model2-6}\rho& =\rho(\phi)=\frac{\varepsilon}{2}\phi+\frac{\varepsilon}{2}+1,
		\end{align}
		with the initial conditions
		\begin{align}
			\label{model2-5-0}
			\mathbf{u}|_{t=0} =\mathbf{u}_{0},\ \phi|_{t=0}=\phi_{0}, \quad\text{in $\bbt^3$}.
		\end{align}
	\end{subequations}
	Here the unknowns $(\phi,\mathbf{u},p,\mu): Q_T \to \mathbb{R} \times \mathbb{R}^3 \times \mathbb{R} \times \mathbb{R}$ are the order parameter (volume fraction difference), the mass-averaged (barycentric) mixture velocity, the pressure and the chemical potential, respectively. The function $F$ denotes the homogeneous Helmholtz free energy of double-well type. 
	The constant $\varepsilon\coloneqq \frac{\rho_2}{\rho_1} - 1 $ measures the relative difference of the two densities $ \rho_1, \rho_2$. Without loss of generality, $\varepsilon$ is assumed to be negative, i.e., $\rho_2 < \rho_1$. Then
	\begin{align*}
		\alpha\coloneqq \frac{\rho_1 - \rho_2}{\rho_2 + \rho_1} = - \frac{\varepsilon}{2+\varepsilon} \in (0,1).
	\end{align*}
	The matrix
	\begin{align*}
		S(\phi,\mathbb{ D}\mathbf{u})\coloneqq2\eta(\phi) \mathbb{ D}\mathbf{u}-\frac{2}{3}\eta(\phi) (\mathrm{div}  \mathbf{u}) \mathbf{I}
	\end{align*}
	denotes the Newtonian stress tensor with the symmetric gradient $\mathbb{D}\mathbf{u} \coloneqq \frac12(\nabla  \mathbf{u}+(\nabla \mathbf{u})^\top)$, where
	\begin{align*}
		\eta(\phi)=\frac{\nu-1}{2}\phi+\frac{\nu+1}{2}
	\end{align*}
	is the viscosity of the mixture and
	$\nu>0$ stands for the constant viscosity ratio of the two fluids.
	
	In \eqref{model2-4}, for $s > 0$  the fractional Laplacian operator in $\bbt^3$ is defined by
	\begin{align*}
		\Lambda^{2s}\phi(x)=(-\Delta)^{s}\phi(x):=\sum_{k\in \mathbb{Z}^3} |k|^{2s}\widehat{\phi}(k)e^{ix\cdot k}
	\end{align*}
	with
	$$
	\widehat{\phi}(k)=\int_{\mathbb{T}^{3}}e^{-ix\cdot k}f(x)\d x.
	$$
	By direct computation one has 
	\begin{align*}
		\widehat{\Lambda^{2s}\phi}(k)=|k|^{2s}\widehat{\phi}(k),\ k\in \mathbb{Z}^3.
	\end{align*}
	
	We refer to \cite{GCLLJ,SYW,SRSvBvdZ2018,MKID} for the derivation of such models when $s=1$. In particular, if  $\rho_1 = \rho_2 $, the quasi-incompressible model \eqref{model2} reduces to the classical \textit{model H} \cite{GPV1996,HH} for two-phase flows of matched densities.

	The total energy $E$ of \eqref{model2} is defined as
	\begin{align}
		\label{Etot}
		E(\mathbf{u}, \phi)
		=\int_{\mathbb{T}^{3}}\frac{1}{2}\rho|\mathbf{u}|^2\, \mathrm{d} x+\int_{\mathbb{T}^{3}} \left( F(\phi)+\frac{1}{2}|\Lambda^{s} \phi|^2\right)\, \mathrm dx.
	\end{align}
	Then smooth solutions of \eqref{model2} fulfill the following energy dissipation law
	\begin{align}\label{ConEnL}
		\frac{\mathrm d}{\mathrm dt}E(\mathbf{u}, \phi)
		=
		-\int_{\mathbb{T}^{3}} \left( 2 \eta(\phi)
		\mathbb{D}\mathbf{u}:\mathbb{ D}\mathbf{u}-\frac{2}{3}\eta(\phi) (\mathrm{div} \mathbf{u})^2\right)\, \mathrm{d} x
		-\int_{\mathbb{T}^{3}} |\nabla \mu_p|^2\, \mathrm{d} x,
	\end{align}
	where one recalls the fact 
	$
	\mathbb{ D}\mathbf{u}:\mathbb{ D}\mathbf{u}-\frac{1}{3}(\mathrm{div} \mathbf{u})^2
	\geq 0.
	$

	Note that \eqref{model2-2} and \eqref{model2-6} imply
	\begin{align*}
		\partial_t\phi + \mathrm{div}(\phi  \mathbf{u})=\frac{1}{\alpha} \mathrm{div}\mathbf{u}.
	\end{align*}
	Recalling the definition of $\rho$ in \eqref{model2-6} and $\alpha=-\frac{\varepsilon}{2+\varepsilon}$, one recovers  the continuity equation (conservation of mass) \cite{AH,SRSvBvdZ2018}
	\begin{align*}
		\partial_t \rho+\mathrm{div} (\rho \mathbf{u})=0.
	\end{align*}

	\subsection{Literature review and current contribution}
	\label{sec:literature-review}
	
	Diffuse interface fluid models assume a thin diffusive layer of finite thickness separating two fluids inside which the two fluids mix due to chemical diffusion. Because they allow for smooth interface topological transitions (such as bubble merging or splitting), diffuse interface models have become an important tool in the study of interfacial dynamics of multiphase flows.  For incompressible two-phase flows of nearly matched densities, the model is known as Model H \cite{GPV1996, HH} taking the form of the incompressible Navier--Stokes/Cahn--Hilliard equations. There is a vast literature dedicated to the analysis and simulations of Model H which we refrain from reviewing. In the case of unmatched densities, Lowengrub and Truskinovsky \cite{LT1998} derived a
	quasi-incompressible Navier--Stokes/Cahn--Hilliard system, utilizing the concentration as the order parameter with the density present in the Cahn--Hilliard equation. In \cite{SYW} the volume fraction is adopted as the order parameter, which greatly simplified the Cahn--Hilliard equation. Recently, Shokrpour Roudbari et al. \cite{SRSvBvdZ2018} re-derived the quasi-incompressible  Navier--Stokes/Cahn--Hilliard model, i.e. model \eqref{model2} with $s=1$, based on the mixture theory and the Coleman-Noll procedure. In  \cite{MKID}, a unified derivation and comparison of the known diffuse interface models were explored from the physical point of view. It is also possible to derive solenoidal phase field fluid models by volume-averaging the velocity, cf. \cite{AGG2012,AGW2025,ADGC,Boyer2002,DSS2007} and \cite{ADG,AGG2024,FS1,GMW,Giorgini2021} for relevant models and analytical results. Finally for a diffuse interface model for compressible fluids, Abels and Feireisl \cite{AF2008} constructed weak solutions globally in time, and the low Mach number limit is justified in \cite{ALN2024} by the relative entropy method.

	In order to describe the long-range interactions among particles for the phase transition phenomenon, nonlocal versions of Cahn--Hilliard equation were employed extensively. One motivation for considering a nonlocal Cahn--Hilliard equation comes from its original physical formulation \cite{CH1958}. In that setting, the Laplace operator appearing in the chemical potential was initially replaced by a spatial convolution term specifically designed to capture long-range particle interactions.
	From this perspective, it seems to be more faithful to the underlying physical phenomena to employ the fractional Laplacian $(-\Delta)^s \phi$, which can be exactly represented as a convolution integral for smooth functions.
	For the Cahn--Hilliard equation with the fractional Laplacian $(-\Delta)^s \mu$ and  $(-\Delta)^{\sigma}\phi$, Akagi, Schimperna and Segatti \cite{ASS}
	showed the existence and uniqueness of weak solutions and long-time behavior of solutions in the case that  $0<s,\sigma<1$, and the singular limits as either $s$ or $\sigma$ tends to zero. Ainsworth and Mao \cite{AM} later investigated the global weak solutions of the fractional Cahn--Hilliard equation when $s=1$ and $\frac{1}{2}<\sigma\leq1$.
	Readers are also referred to other related modelling and analytical results (mostly with spatial convolutions) on nonlocal Cahn--Hilliard equation \cite{ABG,BH,EPPS2024,GGG,Giorgini2024,MR,Poiatti2025}, and nonlocal Navier--Stokes/Cahn--Hilliard system \cite{FS1,FS,FG1,FGG,FGK,HKP2024}.
	
	The major difficulty associated with the quasi-incompressible model \eqref{model2}, compared to Model H, lies in the fact that the velocity field is no longer divergence-free and the pressure enters the chemical potential. As observed in \cite{AH}, the low temporal regularity of the pressure prevents further regularity gains of the order parameter. In particular the usage of the logarithmic potential could not be justified. For weak solutions of the model by Lowengrub and Truskinovsky, Abels \cite{AH} employed the $p$-Laplacian with a slightly modified potential to ensure that the order parameter does not deviate from its physical bounds too much. Existence of local strong solution is established in \cite{AH1}. 
	
	It is noted that once the densities of two fluids are matched, the model reduces to the classical model H. This is also observed formally in \cite{AH}. However, the limit passage from quasi-incompressible models to Model H is once again out of reach due to the lack of regularities mentioned above. A related quasi-incompressible model derived from mixture theory was investigated by Feireisl, Lu and M\'alek \cite{FLM}, where they studied the incompressible limit with the same parameter as our ratio $\alpha = \frac{\rho_2 - \rho_1}{\rho_2 + \rho_1}$, see \cite[(2.13)]{FLM}. Moreover, Giorgini \cite{Giorgini2021} showed a stability result for the strong solutions to the Abels--Garcke--Gr\"un model derived in \cite{AGG2012} and model H \cite{HH} in terms of the density values, which provides a convergence result for the diffuse interface model with volume-averaged velocity. It is essential that in our case, pressure enters the chemical potential equation \eqref{model3-5}, demanding additional estimates for it. Finally, for diffuse interface models, the incompressible limit results are quite limited. Here we complement with two known results for two-phase flows by Feireisl, Petcu and Pra\v{z}\'ak \cite{FPP} and Abels, Liu and Ne\v{c}asov\'a \cite{ALN2024}, where the former one justified the incompressible limit of a compressible Navier--Stokes/Allen--Cahn model with low Mach number via the relative entropy method, and the latter one addressed a compressible Navier--Stokes/Cahn--Hilliard counterpart.
	
	The purpose of current contribution is to study the existence of weak solutions to the quasi-incompressible two-phase flow model \eqref{model2} in $\bbt^3$. In particular, we aim at generalizing the analysis obtained in \cite{AH} with the $p$-Laplacian to the current fractional Laplacian setting. Moreover, we are going to present a first convergence result of the incompressible limit for such quasi-incompressible models for two-phase flows with small density differences.
	
	\subsection{Main results and ideas}
	For the convenience of analysis, it is better to reformulate the system \eqref{model2} by suppressing the mean value of the pressure and chemical potential, see also \cite{AH}.
	Defining $\mu_p^{0}\coloneqq\mu_p-\overline{\mu_p}$ and $p_0\coloneqq p-\overline{p}$  where $\overline{f}\coloneqq\frac{1}{\absm{\bbt^3}}\int_{\bbt^3} f \,\mathrm{d}x$,
	we next focus on the following equivalent form of  \eqref{model2} in $Q_T$
	\begin{subequations}
		\label{model3}
		\begin{align}
			\label{model3-1}
			\partial_t(\rho  \mathbf{u}) + \mathrm{div}\left (\rho  \mathbf{u}\otimes
			\mathbf{u}\right) & = \mathrm{div}\big(S(\phi,\mathbb{ D}\mathbf{u})\big) -\zeta\rho\nabla p_{0}
			- \phi \nabla \mu_{p}^{0}, \\
			\label{model3-2}
			\partial_t \rho+\mathrm{div} (\rho \mathbf{u})& =0, \\
			\label{model3-3}
			\partial_t\phi + \mathrm{div}(\phi  \mathbf{u})& =\Delta\mu_p^{0}, \\
			\label{model3-4}
			\mu& =F'(\phi)+\Lambda^{2s} \phi , \\
			\label{model3-5}
			\mu_p^{0}+\overline{\mu}& =\mu+\alpha p_{0}, \\
			\rho & =\frac{\varepsilon}{2}\phi+\frac{\varepsilon}{2}+1, 
		\end{align}
		with \begin{align}
			\label{model3-5-0}
			\mathbf{u}|_{t=0} =\mathbf{u}_{0},\ \phi|_{t=0}=\phi_{0}, \quad\text{in $\bbt^3$},
		\end{align}
	\end{subequations}
	where $\zeta=\frac{2}{2+\varepsilon}=1+\alpha$. The energy identity is the same as \eqref{ConEnL}
	\begin{align}\label{ConEnL1}
		\frac{\mathrm d}{\mathrm dt}E(\mathbf{u}, \phi)
		=
		-\int_{\mathbb{T}^{3}} \left( 2 \eta(\phi)
		\mathbb{ D}(\mathbf{u}):\mathbb{ D}(\mathbf{u})-\frac{2}{3}\eta(\phi) (\mathrm{div} \mathbf{u})^2\right)\, \mathrm{d} x
		-\int_{\mathbb{T}^{3}} |\nabla \mu_p^{0}|^2\, \mathrm{d} x,
	\end{align}
	where the total energy $E$ is the same as \eqref{Etot}.
	
	Before stating the main result, let us present the assumption which will be employed through the context.
	\begin{assumption}\label{ass:main}
		We assume that $\eta\in C^{2}(\mathbb{R})$ satisfying 
		$0<\inf\eta\leq \eta\leq \sup \eta<\infty$ and $\eta'\in L^\infty(\mathbb{R})$. Moreover, $F(\phi)=\Phi(\phi)-\frac{\kappa}{2}\phi^{2}$ where $\Phi \in C^{3}(\mathbb{R}),\Phi''\geq 0$ and $\kappa$ is a positive constant.
	\end{assumption}

	Now we give the definition of the weak solutions. To proceed, we apply the Helmholtz decomposition to $\mathbf{u}$, i.e., we decompose $\mathbf{u}=\mathbb{P}\mathbf{u}+\nabla\big( \mathbb{G}(\mathbf{u})\big)$ where $\mathbb{G}$ is defined by
	\begin{align*}
		&\Delta \mathbb{G}(\mathbf{u})=\mathrm{div}\mathbf{u},\quad \text{in}\ \mathbb{T}^{3}
	\end{align*}
	and $\int_{\mathbb{T}^{3}} \mathbb{G}(\mathbf{u})  \d x=0$. Then we introduce a new pressure denoted by $p_{1}$
	\begin{align}
		p_{1}=\zeta p_{0}+\partial_{t}\mathbb{G}(\mathbf{u}).\label{new pressure-introduction}
	\end{align}
	\begin{definition}
		\label{def:weak-solutions}
		Let $0<T <\infty$ and $ (\bu_0,\phi_0) \in H^1(\bbt^3) \times H^s(\bbt^3) $ with $s>\frac{3}{2}$.
		In addition, let Assumption \ref{ass:main} hold true. We call the quadruple $(\bu,p_0,\phi,\mu_p^0)$ a weak solution to \eqref{model3} with initial data $(\bu_0,\phi_0)$, provided that
		\begin{enumerate}
			\item The quadruple $(\bu,p_0,\phi,\mu_p^0)$ satisfies
			\begin{align*}
				&\mathbf{u}\in BC_{\omega}(0,T; L^{2}(\mathbb{T}^{3}))\cap L^{2}(0,T; H^{1}(\mathbb{T}^{3})),
				\\&p_0\in L^2(0,T; H^{-s}(\mathbb{T}^{3})),
				\\&\phi\in BC_{\omega}(0,T; H^{s}(\mathbb{T}^{3})),
				\\& \mu_{p}^{0}\in L^{2}(0,T; H^{2}_{(0)}(\mathbb{T}^{3})).
			\end{align*}
			\item $p_1\in L^{2}(0,T;L^{r}(\mathbb{T}^{3}))(1<r<\frac{3}{2})$, and for $\boldsymbol{\varphi}\in C_{0}^{\infty}(0,T;H^1(\bbt^3) \cap L^{\infty}(\mathbb{T}^{3})\cap W^{1,r'}(\mathbb{T}^{3})) (\frac{1}{r}+\frac{1}{r'}=1)$, there holds
			\begin{align}\label{mweak1}
				&\nonumber(\mathbb{P}\mathbf{u},\partial_{t}\boldsymbol{\varphi})_{Q_{T}}+(\mathbf{u}\cdot\nabla\mathbf{u},\boldsymbol{\varphi})_{Q_{T}}-(\rho^{-1}S(\phi,\mathbb{ D}\mathbf{u}),\nabla\boldsymbol{\varphi})_{Q_{T}}-(S(\phi,\mathbb{ D}\mathbf{u}),\nabla\rho^{-1}\otimes\boldsymbol{\varphi})_{Q_{T}}
				\\&\qquad
				+(p_{1},\mathrm{div}\boldsymbol{\varphi})_{Q_{T}}-(\rho^{-1}\phi \nabla \mu_{p}^{0},\boldsymbol{\varphi})_{Q_{T}}
				=0.
			\end{align}
			\item For  $\psi\in C_{0}^{\infty}(0,T; H^{s}(\mathbb{T}^{3}))$,
			\begin{align}
				&\label{mweak2}(\rho,\partial_{t}\psi)_{Q_{T}}+(\rho\mathbf{u}, \nabla\psi)_{Q_{T}}=0,
				\\&\label{mweak3}(\phi,\partial_{t}\psi)_{Q_{T}}+(\phi\mathbf{u},\mathrm{div}\psi)_{Q_{T}}=( \nabla \mu_{p}^{0},\nabla\psi)_{Q_{T}},
				\\&\label{mweak4}(\mu_{p}^{0}-\alpha\zeta^{-1} p_{1}-F'(\phi)+\overline{\mu},\psi)_{Q_{T}}-(\zeta^{-1}\mathbb{G}(\mathbf{u}),\partial_{t}\psi)_{Q_{T}}=(\Lambda^{s}\phi,\Lambda^{s}\psi)_{Q_{T}},
			\end{align}
			where $\mathbb{G}(\mathbf{u})\in L^{2}(0,T;H^{1}(\mathbb{T}^{3}))$ and
			\begin{align*}
				(\overline{\mu},\psi)_{Q_{T}}=\frac{1}{\absm{\bbt^3}}\int_0^{T}\int_{\mathbb{T}^{3}}\big(F'(\phi)\psi -\Lambda^{s} \phi\Lambda^{s} \psi\big) dxdt.
			\end{align*}
			\item For a.e.~$t \in (0,T]$, the triple $(\bu,\phi,\mu_p^0)$ satisfies the energy inequality
			\begin{align}
				& E(\mathbf{u}(t), \phi(t))
				+\int_0^{t}\int_{\mathbb{T}^{3}} \left(2 \eta(\phi)
				\mathbb{D}(\mathbf{u}):\mathbb{D}(\mathbf{u})-\frac{2}{3}\eta(\phi) (\mathrm{div} \mathbf{u})^2\right) \,\dx\dtau\nonumber \\
				& \quad
				+\int_0^{t}\int_{\mathbb{T}^{3}}|\nabla \mu_{p}^{0}|^{2}\,\dx\dtau
				\leq E(\mathbf{u}_{0},\phi_{0}), \label{eqs:energy}
			\end{align}
			where the total energy $E(\mathbf{u}, \phi)$ is defined in \eqref{Etot}.
			\item For a.e.~$t \in (0,T]$,  $\int_{\mathbb{T}^{3}}\phi\, \mathrm{d} x=\int_{\mathbb{T}^{3}}\phi_0\, \mathrm{d} x$ 
			and the solution attains the initial value $\mathbf{u}_{0}$ and $\phi_{0}$ in  the following sense respectively,
		\begin{align}
			& \mathbf{u}(0,x)=\mathbf{u}_{0},\  a.e.,\ in \ \mathbb{T}^{3},\label{initial continuty1}
			\\&\phi(0,x)=\phi_{0},\  a.e.,\ in \ \mathbb{T}^{3}.\label{initial continuty2}
		\end{align}
	\end{enumerate}
\end{definition}

We are in the position to state our first main result on existence of weak solutions.
\begin{theorem}\label{thm:main}
	Let $ s>\frac{3}{2}$ and Assumption \ref{ass:main} hold,
	given $ (\bu_0,\phi_0) \in H^1(\bbt^3) \times H^s(\bbt^3) $. Then for any $0 < T < \infty$, there exists a weak solution $(\bu,p_0,\phi,\mu_p^0)$ of \eqref{model3} in the sense of Definition \ref{def:weak-solutions}. Moreover, for $0\leq\gamma\leq 1$ and some $\theta > 0$, it follows
	\begin{align}
		\label{eqs:phi-high}
		& \phi\in L^{2}(0,T; H^{s+\frac{\gamma}{2}}(\mathbb{T}^{3})), \quad \phi \in (-1-\theta, 1+\theta).
	\end{align}
\end{theorem}

Several remarks are in order.
\begin{remark}
	{It is noted that for the weak solutions, we derive an additional higher-order regularity \eqref{eqs:phi-high}, compared to standard Navier--Stokes/Cahn--Hilliard equations, cf. \cite{AH,ADG}. Here we found that in fact, the capillary force provides partial damping effects for the order parameter $\phi$. More precisely, in the momentum equation we have $\phi \nabla \mu \sim \phi \nabla \Delta \phi$, which is roughly third-order in $\phi$. Then with some suitable multiplier we can improve the regularity of $\phi$, cf. Lemma \ref{lem:higher eatimate}.} Furthermore, it is crucial to employ the higher-order regularity of $\phi$ in further analysis of the imcompressible limit due to the lack of regularity of pressure, cf. Remark \ref{rem:incompressible-limit}.
\end{remark}
\begin{remark}
	Typically, one shall derive a physically reasonable range for $ \phi $ in $ (-1,1) $ by emplying the singular potential, e.g., the Flory--Huggins potential (logarithmic type)
	\begin{align*}
		F(s) =
		\frac{\vartheta}{2}\Big[(1 + s)\log(1 + s) + (1 - s)\log(1 - s)\Big]
		- \frac{\vartheta_0}{2} s^2, \quad s \in [-1,1]
	\end{align*}
	for $ 0 < \vartheta < \vartheta_0 $, cf. \cite{AbelsARMA2009,ADG,AGG2024,Giorgini2021}.
	Analogous to \cite{AH} where the Lowengrub--Truskinovski model \cite{LT1998} was analyzed, in our case a singular potential is not applicable as well. Since the pressure $ p_0 $ enters \eqref{model2-5}, a priorily lack of regularity (even integrability) for $ p_0 $ leads to the crucial issue for the estimate of $ \Lambda^{2s} \phi + F'(\phi) $ in some Lebesgue space. Thanks to the higher-order fractional operator $ \Lambda^{2s} $ with $s > \frac{3}{2}$ introduced, we are able to extract the similar mechanism to keep the order parameter $ \phi \in [-1,1] $ or at least in a suitable neighborhood $ \phi \in (-1-\theta,1+\theta) $, $ \theta > 0 $, cf. Lemma \ref{lem2}. {It will be strongly desirable that the singular potential (logarithmic) can be considered under the mass-averaged mixture velocity.}
\end{remark}

Our second result is concerned with the incompressible limit of \eqref{model3}. Namely, as the parameter $\alpha$ tends to zero, the solution to the quasi-incompressible two-phase model approaches to the one of model H. The following is an abbreviated version.
For a precise description of the limit system and convergent rate, see Section \ref{sec:limit} and Theorem \ref{thm:limit}.
\begin{theorem}
	\label{thm:main2}
	Assume that the initial data is well-prepared in the sense of \eqref{wp-1}-\eqref{wp-2} and $s> \frac{3}{2}$, let $(\bu_\alpha, p_\alpha, \phi_\alpha, \mu_{p,\alpha}) $ be a weak solution (indicated by subscript $\alpha$) to the quasi-incompressible Navier--Stokes/Cahn--Hilliard model \eqref{model2} constructed in Theorem \ref{thm:main}. Then up to the strong solution life span $ T' > 0 $, the weak solution $ (\bu_\alpha, p_\alpha, \phi_\alpha, \mu_{p,\alpha}) $ coincides with the unique strong solution of model H \eqref{eqs:ModelH}, as the density difference vanishes, i.e, $\alpha\rightarrow 0$. Moreover, we have a certain convergent rate depending on $ \alpha $ and the initial data.
\end{theorem}
Here is a comment on the incompressible limit.

\begin{remark}
	\label{rem:incompressible-limit}
	For the incompressible limit, there are in principle three ways to rigorously justify it. First one concerns the weak solutions replies on the weak convergence method by uniform estimates directly. This does not work right now due to the lack of regularity of pressure. The second one is to derive the error estimates for the strong solution of original and limit systems, which may demand delicate energy estimates. This can be one of the further steps rendered from the authors previous work on strong well-posedness \cite{FFHL2024}. The last one is based on the so-called \textit{weak-strong principle}, meaning the weak solution coincides with the unique strong solution as long as the latter one exists. The general argument to justify it is based on the \textit{relative entropy method}, which is exactly the idea of current contribution. The main step is defining the \textit{relative (modulate) entropy (energy)} that measures the distance of weak and strong solutions, and then deriving certain estimates for the remainders and closing it with Gr\"onwall's inequality. Readers are referred to e.g. \cite{ALN2024,FJN,FLN,FPP} for the applications in fluid mechanics.
	
	The main difficulty is that one can not derive a uniform estimate independent of $\alpha$ for the pressure in some Lebesgue space $L^q(Q_T)$ with $q > 1$. In contrast to the incompressible limit of compressible fluids, the remainder terms involving the pressure cannot be directly controlled in the present setting. Fortunately, one shall derive new uniform-in-$\alpha$ controls for the pressure with the momentune equation and the improved regularity of $\phi$, cf. Lemmas \ref{lem:limit 1}, \ref{lem:limit 2}, which are critical to take the incompressible limit.
\end{remark}

\subsection{Outline}
The paper is organized as follows. In Section \ref{sec:preliminaries}, we give basic notations that will be used later.
In Section \ref{sec:approximation}, we approximate the model \eqref{model2} and show the existence of solutions to the approximation system via a fixed-point argument, while
in Section \ref{sec:mainthm}, existence of weak solutions to \eqref{model2} is shown by compactness arguments.
In Section \ref{sec:limit}, we investigate the limit from the quasi-compressible model \eqref{model2} to the incompressible model \eqref{eqs:ModelH} as $\alpha\rightarrow 0$ using the relative entropy method.

\section{Preliminaries}
\label{sec:preliminaries}

In this section, we introduce necessary notations, function spaces that will be used frequently. 
For $A,B\geq 0$, the notation $A\lesssim B$ $(A \gtrsim B)$ stand for inequalities of the form $A \leq CB$ $(A \geq CB)$ for $C>0$,  $A\approx B$ means $B\lesssim A\lesssim B$.
In situations where the precise value or dependence on other quantities of the constant
$C$ is not required.

For two vectors  $a,b\in\mathbb{R}^{3}$,  $a\otimes b\triangleq(a_{j}b_{j})_{i,j}^{3}\in\mathbb{R}^{3\times 3}$. For two $n\times n$ matrices $A=(a_{ij})_{i,j=1}^3, B=(b_{ij})_{i,j=1}^3$,  $A:B\triangleq\sum_{i,j=1}^{3}a_{i,j}b_{i,j}$ and
$\nabla\cdot A\triangleq(\sum_{i,j=1}^{d}\partial_{x_{j}}a_{i,j})_{i=1}^{3}$.  If $\mathbf{v}\in C^{1}(\Omega)^{3}$, then $\nabla \mathbf{v}=(\partial_{x_{j}}\mathbf{v}_{i})_{i,j=1}^{3}$.

Given a Banach space $X$, we denote its norm by $\|\cdot\|_{X}$. For a Banach space $X$ and for any $0<T\leq\infty$, we use standard notation $L^{p}(0,T;X)$ to denote the quasi-Banach space of Bochner measurable functions $f$ from $(0,T)$ to X endowed with the norm
\begin{equation*}
	\|f\|_{L^{p}(0,T;X)}\coloneqq
	\begin{cases}
		(\int_0^T \|f(\cdot,t)\|^{p}_{X}\,\dt)^{\frac{1}{p}},&\text{if $1\leq p<\infty$},\\
		\sup_{0\leq t \leq T}\|f(\cdot,t)\|_X,&\text{if $p=\infty$}.
	\end{cases}
\end{equation*}

The usual Sobolev spaces on a bounded domain $\Omega$ are denoted by $W^{k,p}(\Omega)$($k\in \mathbb{N}, 1\leq p\leq\infty$). The usual Besov spaces on a bounded domain $\Omega$ are denoted by $B_{p,q}^{s}(\Omega)$, where $s\in \mathbb{R}$, $1\leq p,q\leq\infty$. When $p=q=2$,  $B_{2,2}^{s}(\Omega)$ is the usual Sobolev space $H^{s}(\Omega)$. For convenience of later analysis, we recall (cf. \cite{ELC,SY})
\begin{align}
	& \label{BB1} B_{p,\infty}^{s+\varepsilon}(\Omega)\hookrightarrow B_{p,q_{1}}^{s}(\Omega)\hookrightarrow B_{p,q_{2}}^{s}(\Omega), \varepsilon>0, 1\leq q_{1}\leq q_{2}\leq\infty,
	\\& \label{BB2}B_{p_{1},q}^{s_{1}}(\Omega)\hookrightarrow B_{p_{2},q}^{s_{2}}(\Omega), s_{1} \geq s_{2}, s_{1}-\frac{d}{p_{1}}\geq s_{2}-\frac{d}{p_{2}},
	\\& \label{BB3}B_{p,1}^{\frac{d}{p}}(\Omega)\hookrightarrow C_{b}^{0}(\Omega):=C(\Omega)\cap L^{\infty}(\Omega),
	\\& \label{BB4}W^{k,p}(\Omega)\hookrightarrow C^{m,\alpha}(\Omega), k-\frac{d}{p}>m+\alpha(m\in\mathbb{N},\alpha\in (0,1)).
\end{align}


If $X$ is a Banach space and $X'$ is its dual, for $f\in X'$, $g\in X$, we denote the duality product as
\begin{align*}
	\langle f,g \rangle=\langle f,g \rangle_{X,X'}=f(g).
\end{align*}
Also we denote the inner product on a Hilbert space $H$ as $(.,.)_{H}$. $(.,.)_{Q_{T}}$ represents $(.,.)_{L^{2}(Q_{T})}$.

Let $I=[0,T]$ with $0<T<\infty$ and  $X$ be a Banach space.  $BC(I;X)$ represents the Banach space  of all functions $f:I\rightarrow X$ that is  bounded and continuous equipped with supremum norm and $BUC(I;X)$ is the space  of bounded and uniformly continuous function(cf. \cite{AH2}). We can also write $BC_{w}(I;X)$ as the topological spaces where the space of all bounded and weakly continuous functions $f:I\rightarrow X$. Finally, $f\in W^{k,p}(I; X)(1\leq p \leq\infty, k\in \mathbb{N}_{0})$, if and only if
$f,\cdot\cdot\cdot,\frac{d^{k}f}{dt^{k}}\in  L^{p}(I; X)$, where $\frac{d^{k}f}{dt^{k}}$  denotes the $k$-th $X$-valued distributional derivative of $f$.

\section{Approximate Problem and Implicit Time Discretization}
\label{sec:approximation}
The pinciple part of this subsection is to provide the proof of existence of a weak solution of an approximate problem  to \eqref{model3}.  We consider the following approximate problem in $Q_T$
\begin{subequations}
	\label{model4}
	\begin{align}
		\label{model4-1}
		\partial_t(\rho  \mathbf{u}) + \mathrm{div}\left (\rho  \mathbf{u}\otimes
		\mathbf{u}\right) 
		& = \mathrm{div}\big(S(\phi,\mathbb{ D}\mathbf{u})\big)
		- \zeta\rho\nabla p_{0}
		- \phi \nabla \mu_{p}^{0}
		+ \frac{\varepsilon\delta}{4\alpha}  p_{0} \mathbf{u}, 
		\\\label{model4-2}
		\partial_t \rho+\mathrm{div} (\rho \mathbf{u})& =\frac{\varepsilon\delta}{2\alpha} p_{0}, 
		\\
		\label{model4-3}
		\partial_t\phi + \mathrm{div}(\phi  \mathbf{u})& =\Delta\mu_p^{0}, \\
		\label{model4-4}
		\mu& =F'(\phi)+\Lambda^{2s} \phi, \\
		\label{model4-5}
		\mu_p^{0}+\overline{\mu}& =\mu+\alpha p_{0}, \\\rho& =\frac{\varepsilon}{2}\phi+\frac{\varepsilon}{2}+1,
	\end{align}
	with
	\begin{align}
		\label{model4-5-0}\mathbf{u}|_{t=0}=\mathbf{u}_{0},\ \phi|_{t=0}=\phi_{0}, \quad\text{in $\bbt^3$}.
	\end{align}
\end{subequations}

It follows from \eqref{model4-2},\eqref{model4-3}  and \eqref{model4-5} that
\begin{align}
	\label{model4-6}
	&\mathrm{div}  \mathbf{u} +\delta p_{0}= \alpha \Delta\mu_p^{0}, \quad\text{in $Q_{T}$}.
\end{align}
Moreover, it is noted that every smooth solution of \eqref{model4} satisfies
\begin{align}\label{ConEnL3}
	\nonumber\frac{\mathrm d}{\mathrm dt}E(\mathbf{u}, \phi)&=
	-\int_{\mathbb{T}^{3}}\big( (2 \eta(\phi)
	\mathbb{ D}(\mathbf{u}):\mathbb{ D}(\mathbf{u})-\frac{2}{3}\eta(\phi) (\mathrm{div} \mathbf{u})^2)\,\big)\, \mathrm{d} x \\
	& \quad
	-\int_{\mathbb{T}^{3}} |\nabla \mu_p^0|^2\, \mathrm{d} x
	-\delta\int_{\mathbb{T}^{3}} |p_{0}|^{2}  \, \mathrm{d} x.
\end{align}

Before stating the main theorem, we present the definition of weak solutions of \eqref{model4-1}--\eqref{model4-6}.
\begin{definition}\label{appro;def}
	Let $0<T < \infty$ and $ (\bu_0,\phi_0) \in H^1(\bbt^3) \times H^s(\bbt^3) $ with $s>\frac{3}{2}$.
	In addition, let Assumption \ref{ass:main} hold true. We call the quadruple $(\bu,p_0,\phi,\mu_p^0)$ a weak solution to \eqref{model4-1}--\eqref{model4-5} with initial data $(\bu_0,\phi_0)$, provided that	
	\begin{enumerate}
		\item The quadruple $(\bu,p_0,\phi,\mu_p^0)$ satisfies
		\begin{align*}
			&\mathbf{u}\in BC_{\omega}(0,T; L^{2}(\mathbb{T}^{3}))\cap L^{2}(0,T; H^{1}(\mathbb{T}^{3})),
			\\& p_{0}\in L^{2}(0,T; L^{2}_{(0)}(\mathbb{T}^{3})),
			\\&\phi\in BC_{\omega}(0,T; H^{s}(\mathbb{T}^{3})),
			\\& \mu_{p}^{0}\in L^{2}(0,T; H^{2}_{(0)}(\mathbb{T}^{3})).
		\end{align*}
		\item For all  $\boldsymbol{\varphi}\in C_{0}^{\infty}(0,T;H^1(\bbt^3) \cap L^{\infty}(\mathbb{T}^{3}))$,
		\begin{align}\label{weak1}
			&\nonumber(\mathbf{u},\partial_{t}\boldsymbol{\varphi})_{Q_{T}}+(\mathbf{u}\cdot\nabla\mathbf{u},\boldsymbol{\varphi})_{Q_{T}}-(\rho^{-1}S(\phi,\mathbb{ D}\mathbf{u}),\nabla\boldsymbol{\varphi})_{Q_{T}}-(S(\phi,\mathbb{ D}\mathbf{u}),\nabla\rho^{-1}\otimes\boldsymbol{\varphi})_{Q_{T}}
			\\&\qquad
			+\zeta(p_{0},\mathrm{div}\boldsymbol{\varphi})_{Q_{T}}-(\rho^{-1}\phi \nabla \mu_{p}^{0},\boldsymbol{\varphi})_{Q_{T}}
			-(\frac{\varepsilon\delta}{4\alpha}  p_{0} \mathbf{u},\boldsymbol{\varphi}\rho^{-1})_{Q_{T}}=0,
		\end{align}
		where and in what follows $(\cdot,\cdot)_{Q_{T}}$ is defined in section \ref{sec:preliminaries}.
		\item For  $\psi\in C_{0}^{\infty}(0,T; H^{s}(\mathbb{T}^{3}))$,
		\begin{align}
			&\label{weak2}
			(\rho,\partial_{t}\psi)_{Q_{T}}+(\rho\mathbf{u}, \nabla\psi)_{Q_{T}}+\frac{\varepsilon\delta}{2\alpha}( p_{0},\psi)_{Q_{T}}=0,
			\\&\label{weak3}(\phi,\partial_{t}\psi)_{Q_{T}}+(\phi\mathbf{u},\nabla\psi)_{Q_{T}}=( \nabla \mu_{p}^{0},\nabla\psi)_{Q_{T}},
			\\&\label{weak4}(\mu_{p}^{0}-\alpha p_{0}-F'(\phi)+\overline{\mu},\psi)_{Q_{T}}=(\Lambda^{s}\phi,\Lambda^{s}\psi)_{Q_{T}},
		\end{align}
		where
		\begin{align*}
			(\overline{\mu},\psi)_{Q_{T}}=\frac{1}{\absm{\bbt^3}}\int_0^{T}\int_{\mathbb{T}^{3}}\big(F'(\phi)\psi -\Lambda^{s} \phi\Lambda^{s} \psi\big) dxdt.
		\end{align*}
		\item For a.e.~$t \in (0,T]$, the triple $(\bu,\phi,\mu_p^0,p_{0})$ satisfies the energy inequality
		\begin{align}\label{ConEnL2}
			\nonumber E(\phi(t),\mathbf{u}(t))&+\int_0^{t}\int_{\mathbb{T}^{3}}S(\phi,\mathbb{ D}\mathbf{u}):\mathbb{ D}\mathbf{u}\,\mathrm d x\mathrm dt
			\\&
			+\int_0^{t}\int_{\mathbb{T}^{3}}|\nabla \mu_{p}^{0}|^{2}\,\mathrm d x\mathrm dt+\delta\int_0^{t}\int_{\mathbb{T}^{3}}|{p}_{0}|^{2}\,\mathrm d x\mathrm dt\leq E(\phi_0,\mathbf{u}_0),
		\end{align}
		where the total energy $E$ is defined as
		\begin{align*}
			E(\mathbf{u}, \phi)=
			\int_{\mathbb{T}^{3}}\frac{1}{2}\rho|\mathbf{u}|^2\, \mathrm{d} x+
			\int_{\mathbb{T}^{3}}\Big( F(\phi)+\frac{1}{2}|\Lambda^{s} \phi|^2\Big)\, \mathrm dx.
		\end{align*}
		\item For a.e.~$t \in (0,T]$,  $\int_{\mathbb{T}^{3}}\phi\, \mathrm{d} x=\int_{\mathbb{T}^{3}}\phi_0\, \mathrm{d} x$ 
		and the solution attains the initial value $\mathbf{u}_{0}$ and $\phi_{0}$ in  the following sense respectively,
		\begin{align}
			& \mathbf{u}(0,x)=\mathbf{u}_{0},\  a.e.,\ in \ \mathbb{T}^{3},\label{appinitial data-1}
			\\&\phi(0,x)=\phi_{0},\  a.e.,\ in \ \mathbb{T}^{3}.\label{appinitial data-2}
		\end{align}
	\end{enumerate}

\end{definition}
\begin{theorem}\label{main theorem2}
	Let Assumption \ref{ass:main} hold, and $R,\theta>0$ be as in Lemma \ref{lem2}. Given $\mathbf{u}_{0}\in L^{2}(\mathbb{T}^{3})$, $\mathbf{\phi}_{0}\in H^{s}(\mathbb{T}^{3})$ with $s>\frac{3}{2}$ with $E(\mathbf{u}_{0}, \phi_{0})<R$.
	Then for every $0<T<\infty$, problem \eqref{model4-1}--\eqref{model4-5} admits a weak solution $(\mathbf{u}, p_{0}, \mathbf{\phi},\mu_{p}^{0})$ corresponding to $\mathbf{u}_{0},\mathbf{\phi}_{0}$ in the sense
	of Definition \ref{appro;def}.
	%
	
\end{theorem}

In order to prove Theorem \ref{main theorem2}, we apply the implicit time discretization(cf. \cite{AH,ADG}). At first, for given $(\phi_{k}, \mathbf{u}_{k})$ and $\rho_{k}=\rho(\phi_{k})$, $k\in \mathbb{N}_{0}:=\mathbb{N}\cup\{0\}$, $(\mathbf{u}_{k+1},p_{0,k+1},\phi_{k+1},\mu_{p,k+1}^{0})$ can be obtained by solving the following system in $\mathbb{T}^{3}$
\begin{align}\label{im1}
	&\rho_{k}\frac{\mathbf{u}-\mathbf{u}_{k}}{h}+\rho\mathbf{u}\cdot\nabla\mathbf{u}-\mathrm{div}S(\phi_{k},\mathbb{ D}\mathbf{u})
	+\zeta\rho \nabla p_{0}+\phi \nabla \mu_{p}^{0}
	+\frac{\varepsilon\delta}{4\alpha}p_{0} \mathbf{u}=0, \\
	\label{im2}
	&\frac{\rho-\rho_{k}}{h}+\mathrm{div} (\rho \mathbf{u})-\frac{\varepsilon\delta}{2\alpha} p_{0}=0, 
	\\  \label{im2-1}&\frac{\phi-\phi_{k}}{h}+\mathrm{div}(\phi\mathbf{u})= \Delta\mu_{p}^{0}, \\
	\label{im3}
	&\mu_{p}^{0}-\alpha p_{0}-\Phi'(\phi)+\kappa\frac{\phi+\phi_{k}}{2}+\overline{\mu}=\Lambda^{2s}\phi.
\end{align}



\begin{lemma}\label{im-energy}
	Let $R,h,\theta>0$. Then for every $(\phi_{k}, \mathbf{u}_{k})\in L^{2}(\mathbb{T}^{3})\times H^{s}(\mathbb{T}^{3})$ with $E(\phi_{k}, \mathbf{u}_{k})<R$, there exists a unique solution
	$$(\mathbf{u},p_{0},\phi,\mu_{p}^{0})\in H^{1}(\mathbb{T}^{3})\times L^{2}_{(0)}(\mathbb{T}^{3})\times H^{s}(\mathbb{T}^{3})\times H^{2}_{(0)}(\mathbb{T}^{3})$$
	solving \eqref{im1}--\eqref{im3} in the following sense
	\begin{align}\label{im1-w}
		&\nonumber(\rho_{k}\frac{\mathbf{u}-\mathbf{u}_{k}}{h},\boldsymbol{\varphi})_{\mathbb{T}^{3}}+(\rho\mathbf{u}\cdot\nabla\mathbf{u},\boldsymbol{\varphi})_{\mathbb{T}^{3}}+(S(\phi_{k},\mathbb{ D}\mathbf{u}),\nabla\boldsymbol{\varphi})_{\mathbb{T}^{3}}
		\\&\quad-\zeta(\rho p_{0},\mathrm{div}\boldsymbol{\varphi})_{\mathbb{T}^{3}}+(\phi \nabla \mu_{p}^{0},\boldsymbol{\varphi})_{\mathbb{T}^{3}}-\zeta(\nabla\rho p_{0},\boldsymbol{\varphi})_{\mathbb{T}^{3}}
		+(\frac{\varepsilon\delta}{4\alpha}p_{0} \mathbf{u},\boldsymbol{\varphi})_{\mathbb{T}^{3}}=0,\\
		\label{im2-w}
		&\frac{\rho-\rho_{k}}{h}+\mathrm{div} (\rho \mathbf{u})-\frac{\varepsilon\delta}{2\alpha} p_{0}=0, \quad\text{in $\mathbb{T}^{3}$,}
		\\  \label{im2-1-w}&(\frac{\phi-\phi_{k}}{h},\chi)_{\mathbb{T}^{3}}-(\phi\mathbf{u},\nabla\chi)_{\mathbb{T}^{3}}= -(\nabla \mu_{p}^{0},\nabla\chi)_{\mathbb{T}^{3}},\\
		\label{im3-w}
		&(\mu_{p}^{0}-\alpha p_{0}-\phi^3+\frac{\phi+\phi_{k}}{2}+\overline{\mu},\psi)_{\mathbb{T}^{3}}=(\Lambda^{s}\phi,\Lambda^{s}\psi)_{\mathbb{T}^{3}},
	\end{align}
	where  $\boldsymbol{\varphi}\in H^{1}(\mathbb{T}^{3})$, $\chi\in H_{(0)}^{1}(\mathbb{T}^{3})$ and $\psi\in H^{s}(\mathbb{T}^{3})$. Moreover, the solution satisfies  the following energy inequality
	\begin{align}\label{imes}
		&\nonumber E(\phi, \mathbf{u}) + \int_{\mathbb{T}^{3}} \rho_{k}\frac{|\mathbf{u}-\mathbf{u}_{k}|^{2}}{2}\,\mathrm dx +h\int_{\mathbb{T}^{3}}S(\phi_{k},\mathbb{ D}\mathbf{u}):\mathbb{ D}\mathbf{u}\,\mathrm dx
		\\&\quad+h\int_{\mathbb{T}^{3}}|\nabla \mu_{p}^{0}|^{2}\mathrm dx+h\delta\int_{\mathbb{T}^{3}}|{p}_{0}|^{2}\mathrm dx\leq E(\phi_{k}, \mathbf{u}_{k}),
	\end{align}
	where $H_{(0)}^{2}(\mathbb{T}^{3})=H^{2}(\mathbb{T}^{3})\cap L^{2}_{(0)}(\mathbb{T}^{3})$ and
	\begin{align*}
		E(\phi_{k}, \mathbf{u}_{k})=
		\int_{\mathbb{T}^{3}}\frac{1}{2}\rho_k|\mathbf{u_k}|^2\, \mathrm{d} x+
		\int_{\mathbb{T}^{3}}\Big( F(\phi_k)+\frac{1}{2}|\Lambda^{s} \phi_k|^2\Big)\, \mathrm dx.
	\end{align*}

\end{lemma}
\begin{proof} We divide two steps to proceed.

	\underline{Step 1.} In this part we aim to prove the energy inequality \eqref{imes}.
	Firstly, we have
	\begin{align}\label{im51}
		& \int_{\mathbb{T}^{3}} \rho_{k}\frac{(\mathbf{u}-\mathbf{u}_{k})\cdot\mathbf{u}}{h}\,\mathrm dx=\int_{\mathbb{T}^{3}} \rho_{k}\frac{|\mathbf{u}-\mathbf{u}_{k}|^{2}}{2h}\,\mathrm dx
		+\int_{\mathbb{T}^{3}} \rho_{k}\frac{|\mathbf{u}|^{2}}{2h}\,\mathrm dx-\int_{\mathbb{T}^{3}} \rho_{k}\frac{|\mathbf{u}_{k}|^{2}}{2h}\,\mathrm dx
	\end{align}
	and
	\begin{align}\label{im52}
		& \int_{\mathbb{T}^{3}} \rho\mathbf{u}\cdot\nabla\mathbf{u}\cdot\mathbf{u}\,\mathrm dx=-\int_{\mathbb{T}^{3}} \mathrm{div}(\rho\mathbf{u})\frac{|\mathbf{u}|^{2}}{2}\,\mathrm dx
		=\int_{\mathbb{T}^{3}} ( \frac{\rho-\rho_{k}}{h}-\frac{\varepsilon\delta}{2\alpha} p_{0})\frac{|\mathbf{u}|^{2}}{2}\,\mathrm dx.
	\end{align}
	Employing \eqref{im51}, \eqref{im52} and the fact that $\zeta\rho=1-\alpha\phi$, we  test \eqref{im1} by $\mathbf{u}$ in to get
	\begin{align}\label{im5}
		&\nonumber \int_{\mathbb{T}^{3}} \rho\frac{\mathbf{|u|}^{2}}{2h}\,\mathrm dx + \int\limits_{\mathbb{T}^{3}} \rho_{k}\frac{|\mathbf{u}-\mathbf{u}_{k}|^{2}}{2h}\,\mathrm dx +\int_{\mathbb{T}^{3}}S(\phi_{k},\mathbb{ D}\mathbf{u}):\mathbb{ D}\mathbf{u}\,\mathrm dx
		\\&\qquad-\int_{\mathbb{T}^{3}} p_{0}\mathrm{div}\mathbf{u}\mathrm dx -\int_{\mathbb{T}^{3}} \mu_{p}^{0}\mathrm{div}(\phi\mathbf{u})\mathrm dx
		+\int_{\mathbb{T}^{3}}\alpha p_{0}\mathrm{div}(\phi\mathbf{u})\mathrm dx=0
	\end{align}
	which together with \eqref{model4-6} and \eqref{im2-1} implies that
	\begin{align}\label{im5-1}
		&\nonumber \int_{\mathbb{T}^{3}} \rho\frac{\mathbf{|u|}^{2}}{2h}\,\mathrm dx + \int\limits_{\mathbb{T}^{3}} \rho_{k}\frac{|\mathbf{u}-\mathbf{u}_{k}|^{2}}{2h}\,\mathrm dx +\int_{\mathbb{T}^{3}}S(\phi_{k},\mathbb{ D}\mathbf{u}):\mathbb{ D}\mathbf{u}\,\mathrm dx
		\\&\qquad+\delta\int_{\mathbb{T}^{3}} |p_{0}|^2\mathrm dx -\alpha\int_{\mathbb{T}^{3}} p_{0}\frac{\phi-\phi_{k}}{h}\mathrm dx-\int_{\mathbb{T}^{3}} \mu_{p}^{0}\mathrm{div}(\phi\mathbf{u})\mathrm dx=0.
	\end{align}

	One tests \eqref{im2-1} by $\mu_{p}^{0}$  and tests  \eqref{im3} by $\psi=\frac{\phi-\phi_{k}}{h}$ respectively, then gets
	\begin{align}
		&\label{im54}(\frac{\phi-\phi_{k}}{h},\mu_{p}^{0})+(\mathrm{div}(\phi\mathbf{u}),\mu_{p}^{0})= -(\nabla \mu_{p}^{0},\nabla\mu_{p}^{0}),
		\\\label{im7}
		&\bigg(\Lambda^{s}\phi,\Lambda^{s}\big(\frac{\phi-\phi_{k}}{h}\big)\bigg)=(\mu_{p}^{0}-\alpha p_{0}-\Phi'(\phi)+\kappa\frac{\phi+\phi_{k}}{2}+\overline{\mu},\frac{\phi-\phi_{k}}{h}).
	\end{align}
	Submitting the above results into \eqref{im5-1}, we can obtain that
	\begin{align*}
		&\nonumber \int_{\mathbb{T}^{3}} \rho\frac{|\mathbf{u}|^{2}}{2h}\,\mathrm dx + \int_{\mathbb{T}^{3}} \rho_{k}\frac{|\mathbf{u}-\mathbf{u}_{k}|^{2}}{2h}\,\mathrm dx +\int_{\mathbb{T}^{3}}S(\phi_{k},\mathbb{ D}\mathbf{u}):\mathbb{ D}\mathbf{u}\,\mathrm dx +\int_{\mathbb{T}^{3}}|\nabla \mu_{p}^{0}|^{2}\mathrm dx
		\\&+\delta\int_{\mathbb{T}^{3}}|p_{0}|^{2}\mathrm dx
		+\int_{\mathbb{T}^{3}}\bigg(\big(\Phi'(\phi)-\kappa\frac{\phi+\phi_{k}}{2}\big)\frac{\phi-\phi_{k}}{h}+\Lambda^{s}\phi\Lambda^{s}\big(\frac{\phi-\phi_{k}}{h}\big)\bigg)\mathrm dx
		=\int_{\mathbb{T}^{3}} \rho_{k}\frac{|\mathbf{u}_{k}|^{2}}{2h}\,\mathrm dx,
	\end{align*}
	where $\int_{\mathbb{T}^{3}}(\phi-\phi_{k})dx=0$ which comes from \eqref{im2-1}.
	
	Finally, due to the Young's inequality and $\Phi$ is convex, we can derive the energy estimate \eqref{imes}. Furthermore, by virtue of Lemma \ref{lem2}, we can prove that $1-\theta<\phi<1+\theta$.

	\underline{Step 2.}
	We intend to use homotopy argument based on the Leray-Schauder degree(cf. \cite{NL,TPT}) to  show the existence of weak solutions of  \eqref{im1}--\eqref{im3}. Firstly, we define $\omega=(\mathbf{u},p_{0},\phi,\mu_{p}^{0})$ and
	\begin{align*}
		& X= H^{1}(\mathbb{T}^{3})\times L_{(0)}^{2}(\mathbb{T}^{3})\times H^{s}(\mathbb{T}^{3})\times H_{(0)}^{2}(\mathbb{T}^{3}),
		\\&Y= H^{-1}(\mathbb{T}^{3})\times L_{(0)}^{2}(\mathbb{T}^{3})\times H_{(0)}^{-1}(\mathbb{T}^{3})\times H^{-s}(\mathbb{T}^{3}).
	\end{align*}
	We can rewrite the system \eqref{im1}--\eqref{im3} as follows
	\begin{equation*}
		\mathcal{L}_{k}\mathbf{\omega}=\mathcal{F}_{k}\mathbf{\omega},
	\end{equation*}
	where  the operator $L_{k}:X\rightarrow Y$  is defined by
	\begin{equation}\label{main model}
		\mathcal{L}_{k}\mathbf{\omega}=
		\left(
		\begin{array}{c}
			L(\mathbf{u},p_{0})
			\\
			\mathrm{div}(\phi\mathbf{u})-\frac{1}{\alpha}\mathrm{div}  \mathbf{u} -\frac{\delta}{\alpha} p_{0}
			\\
			-\Delta\mu_{p}^{0} \\
			\phi+\Lambda^{2s}\phi
		\end{array}
		\right),
	\end{equation}
	where $\langle L(\mathbf{u},p_{0}),\psi\rangle_{H^{-1}(\mathbb{T}^{3}),H^{1}(\mathbb{T}^{3})}=(S(\phi_{k},\mathbb{ D}\mathbf{u}),\mathbb{ D}\psi)_{\mathbb{T}^{3}}-(p_{0},\mathrm{div}\psi)_{\mathbb{T}^{3}}+(\alpha p_{0},\mathrm{div}(\phi\psi))_{\mathbb{T}^{3}}+(\rho_{k}\frac{\mathbf{u}}{h},\psi)_{\mathbb{T}^{3}}$ for all $\psi\in H^{1}(\mathbb{T}^{3})$. Moreover, we define the nonlinear part  $\mathcal{F}_{k}\mathbf{\omega}:X\rightarrow Y$ as
	\begin{equation*}
		\mathcal{F}_{k}\omega=
		\left
		(\begin{array}{c}
			F_{1}\\
			F_{2}\\
			F_{3}\\
			F_{4}\\
		\end{array}
		\right),
	\end{equation*}
	where
	\begin{align*}
		F_{1}&=-\phi \nabla \mu_{p}^{0}+\rho_{k}\frac{\mathbf{u}_{k}}{h}-\rho\mathbf{u}\cdot\nabla\mathbf{u}-\frac{\varepsilon\delta}{4\alpha}p_{0}\mathbf{u},\\
		F_{2}&=-\frac{\phi-\phi_{k}}{h},
		\\F_{3}&=-\mathrm{div}  (\phi\mathbf{u}) -\frac{\phi-\phi_{k}}{h},
		\\F_{4}&=\mu_{p}^{0}-\alpha p_{0}-\Phi'(\phi)+\kappa+\frac{\phi+\phi_{k}}{2}+\phi+\overline{\mu}.
	\end{align*}

	In what follows, we will first show the invertibility of  $\mathcal{L}_{k}:X\rightarrow Y$ and continuity of $\mathcal{L}_{k}^{-1}:Y\rightarrow X$. It is noted that
	\begin{align}\label{im8}
		&\nonumber\langle L(\mathbf{u},p_{0}),\mathbf{u}\rangle_{H^{-1},H^{1}}-\alpha(\mathrm{div}(\phi\mathbf{u}),p_{0})+(\mathrm{div}\mathbf{u},p_{0})+\delta( p_{0},p_{0})
		\\&\nonumber=(S(\phi_{k},\mathbb{ D}\mathbf{u}),\mathbb{ D}\mathbf{u})+\int_{\mathbb{T}^{3}}\rho_k\frac{|\mathbf{u}|^{2}}{h}\mathrm d x+\delta( p_{0},p_{0})
		\\&\nonumber\gtrsim  \int_{\mathbb{T}^{3}} \big( 2
		\mathbb{ D}(\mathbf{u}):\mathbb{ D}(\mathbf{u})-\frac{2}{3} (\mathrm{div} \mathbf{u})^2\big)\, \mathrm{d} x+\int_{\mathbb{T}^{3}}\frac{|\mathbf{u}|^{2}}{h}\mathrm d x+\delta( p_{0},p_{0})
		\\&\nonumber\gtrsim  \int_{\mathbb{T}^{3}} \big( 2
		\mathbb{ D}(\mathbf{u}):\mathbb{ D}(\mathbf{u})-\frac{2}{3} (\mathrm{div} \mathbf{u})^2\big)\, \mathrm{d} x+(\int_{\mathbb{T}^{3}}\frac{|\mathbf{u}|}{h}\mathrm d x)^{2}+\delta( p_{0},p_{0})
		\\&\nonumber\geq C\|\mathbf{u}\|_{H^{1}(\mathbb{T}^{3})}^{2}+\delta\| p_{0}\|_{L^{2}(\mathbb{T}^{3})}^{2},
	\end{align}
	where we have used the Korn's inequality \cite[Theorem 10.16]{FN}.
	Thanks to the lemma of Lax-Milgram, we can obtain that
	\begin{equation*}
		\left
		(\begin{array}{c}
			L(\mathbf{u},p_{0})\\
			\mathrm{div}(\phi\mathbf{u})-\frac{1}{\alpha}\mathrm{div}  \mathbf{u} -\frac{1}{\alpha}\delta p_{0}\\
		\end{array}
		\right)=\left
		(\begin{array}{c}
			f_{1}\\
			f_{2}\\
		\end{array}
		\right), \:\:\text{$f_{1}\in H^{-1}(\mathbb{T}^{3})$,  $f_{2}\in L_{(0)}^{2}(\mathbb{T}^{3})$}
	\end{equation*}
	has a unique solution $(\mathbf{u},p_{0})\in H^{1}(\mathbb{T}^{3})\times L_{(0)}^{2}(\mathbb{T}^{3})$.
	
	The invertibility of $-\Delta(\cdot):H_{(0)}^{1}(\mathbb{T}^{3})\rightarrow H_{(0)}^{-1}(\mathbb{T}^{3})$ can obtained analogously by the lemma of Lax-Milgram. The invertibility of $\mathbf{I}+\Lambda^{2s}:H^{s}(\mathbb{T}^{3})\rightarrow H^{-s}(\mathbb{T}^{3})$ can derived by Lemma \ref{lem5}. Collecting above estimates, we get $\mathcal{L}_{k}:X\rightarrow Y$ is invertible. Moreover, we can get $\mathcal{L}_{k}^{-1}:Y\rightarrow X$ is continuous.

	Next, we define $\mathcal{K}_{k}(\omega)=\mathcal{L}_{k}^{-1}\mathcal{F}_{k}(\omega)$ which a compact operator on $X$ since
	$$\mathcal{F}_{k}(\omega)\in L^{\frac{3}{2}}(\mathbb{T}^{3})\times H^{s}(\mathbb{T}^{3})\times L_{(0)}^{2}(\mathbb{T}^{3})\times L_{(0)}^{2}(\mathbb{T}^{3})\hookrightarrow\hookrightarrow Y$$
	for all $\omega\in X$.
	For some suitable open set $U$,  we can apply a homotopy argument(cf. \cite{NL,TPT}) in order to show that the Leray-Schauder degree of $\mathbf{I}-\mathcal{K}_{k}\mathbf{\omega}$ at $0$ is $1$. Substituting $\mathbf{u}_{k},\Phi'(\phi), \phi_{k}, \kappa,\delta,\rho,\overline{\mu}$ by
	\begin{align*}
		&\mathbf{u}_{k}^{\tau}=(1-\tau)\mathbf{u}_{k}, (\Phi'(\phi))_{\tau}=(1-\tau)\Phi'(\phi)\color{black}, \phi_{k}^{\tau}=(1-\tau)\phi_{k}, \kappa_{\tau}=(1-\tau)\kappa
		\\&\delta_{\tau}=(1-\tau)\delta,\rho_{\tau}=(1-\tau)\rho, \overline{\mu}_{\tau}=(1-\tau)\overline{\mu}
	\end{align*}
	for $\tau\in[0,1]$. Similarly, we define the corresponding solution operators as $\mathcal{K}_{k}^{\tau}, \mathcal{L}_{k}^{\tau}, \mathcal{F}_{k}^{\tau}$. Furthermore, we can obtain a series of compact operator $\mathcal{K}_{k}^{\tau}$ for $\tau \in[0,1]$. We can also get that $\mathcal{K}_{k}^{0}=\mathcal{K}_{k}$.
	And if the following holds
	\begin{align}
		&\mathbf{\omega}-\mathcal{K}_{k}^{\tau}\mathbf{\omega}\neq0, \text{for $\mathbf{\omega} \in \partial U$ when $\tau \in[0,1]$},\label{im9-con}
	\end{align}
	we can show
	\begin{align*}
		&\deg(\mathbf{I}-\mathcal{K}_{k}^{0}\mathbf{\omega},0,U)=\deg(\mathbf{I}-\mathcal{K}_{k}^{1}\mathbf{\omega},0,U).
	\end{align*}
	
	It follows from  \eqref{imes} that there  exists $ C(R,\theta)>0$ such that any solution of \eqref{im1}--\eqref{im3} satisfies  $\|\mathbf{\omega}\|_{X}<C(R,\theta)$  and $E(\phi, \mathbf{u})< R$, which  then implies that  \eqref{im9-con} can be fulfilled if $U=\{\mathbf{\omega}\in X:\|\mathbf{\omega}\|_{X}<C(R,\theta),E(\phi, \mathbf{u})<R\}$.
	
	Hence, in order to show that $\deg(\mathbf{I}-\mathcal{K}_{k}^{1}\mathbf{\omega},0,U)=1$, we should define a second homotopy as
	\begin{align*}
		&\mathcal{K}_{k}^{\tau}(\omega)\coloneqq (\mathcal{L}_{k}^{1})^{-1}(2-\tau)\mathcal{F}_{k}^{1}\omega, \quad \text{$1\leq\tau\leq 2$.}
	\end{align*}
	It can be checked that $\mathcal{K}_{k}^{\tau}(\omega)-\omega=0$ for $1\leq\tau\leq 2$ if and only if $\omega=(\mathbf{u},p_{0},\phi,\mu_{p}^{0})$ is a solution to the following system in $\mathbb{T}^{3}$
	\begin{align}\label{im9}
		&\lambda(\frac{\varepsilon}{2}+1)\frac{\mathbf{u}}{h}
		+\mathrm{div}(S(0,\mathbb{ D}\mathbf{u}))
		+\nabla p_{0}-\alpha\phi\nabla p_{0}+\lambda\phi\nabla \mu_{p}^{0}=0,
		\\&\label{im10}\mathrm{div}(\phi\mathbf{u})- \frac{1}{\alpha}\mathrm{div}\mathbf{u} +\lambda\frac{\phi}{h}=0,
		\\&\label{im11}\Delta\mu_{p}^{0}=\frac{\lambda\phi}{h}+\lambda\mathrm{div}(\phi\mathbf{u}),
		\\&\label{im12}\phi+\Lambda^{2s}\phi=\lambda(\mu_{p}^{0}-\alpha p_{0}+\phi),
	\end{align}
	where $\lambda=2-\tau$.
	
	We then test \eqref{im9} by $\mathbf{u}$, \eqref{im10} by $\alpha p_{0}$, \eqref{im11} by $\mu_{p}^{0}$ and \eqref{im12} by $\frac{\phi}{h}$, and arrive at
	\begin{align*}
		(S(0,\mathbb{ D}\mathbf{u}),\nabla\mathbf{u})_{\mathbb{T}^{3}}+\|\nabla \mu_{p}^{0}\|_{L^{2}(\mathbb{T}^{3})}^{2}+(1-\lambda)\int_{\mathbb{T}^{3}} \frac{\phi^{2}}{h}\,\mathrm dx
		+\lambda(\frac{\varepsilon}{2}+1)\int_{\mathbb{T}^{3}} \frac{\mathbf{u}^{2}}{h}\,\mathrm dx+\int_{\mathbb{T}^{3}} \frac{|\Lambda^{s}\phi|^{2}}{h}\,\mathrm dx=0,
	\end{align*}
	which implies that  the solution of $\mathcal{K}_{k}^{\tau}(\omega)=\omega$ with $1\leq\tau\leq2$ remains in $U$. Then we can get that
	\begin{align*}
		&\deg(\mathbf{I}-\mathcal{K}_{k}^{0},0,U)=\deg(\mathbf{I}-\mathcal{K}_{k}^{1},0,U)=\deg(\mathbf{I}-\mathcal{K}_{k}^{2},0,U)=\deg(\mathbf{I},0,U)=1.
	\end{align*}
	
	It is noted that  $\omega=(\mathbf{u},p_{0},\phi, \mu_{p}^{0})$ is a solution of \eqref{im1}--\eqref{im3} if and only if $\mathbf{\omega}=\mathcal{K}_{k}^{0}(\omega)=\mathcal{K}_{k}(\omega)$.
	Thus, we complete the proof of Lemma \ref{im-energy}.
\end{proof}

Secondly, to prove Theorem \ref{main theorem2}, we apply the implicit time discretization and pass to the limit. Given $N\in \mathbb{N}$ and let $\mathbf{u}_{k+1}, p_{0,k+1},\phi_{k+1},\mu_{p,k+1}^{0}$( $k\in \mathbb{N}_{0}\coloneqq \mathbb{N}\cup \{0\}$) be chosen as a solution of \eqref{im1}--\eqref{im3} with $h=\frac{1}{N}$ and $\mathbf{u}_{k},\phi_{k}$ as initial value. Furthermore, define $f^{N}(t):[-h,\infty)$ by $f^{N}(t)=f_{k}$ for $t\in [(k-1)h,kh)$ (setting $p_{0,0}=\mu_{p,0}^{0}=0$). In the following, we denote
\begin{align*}
	&(\triangle_{h}^{+}f)(t)=f(t+h)-f(t),\quad (\triangle_{h}^{-}f)(t)=f(t)-f(t-h),
	\\&g_{h}=g(t-h),\quad\partial_{t,h}^{\pm}f=\frac{1}{h}\triangle_{h}^{\pm}f.
\end{align*}
Choosing $\boldsymbol{\varphi}(x)=\frac{1}{\rho_{k}}\int_{kh}^{(k+1)h} \boldsymbol{\chi}\,\mathrm dt$ in \eqref{im1}, where $\boldsymbol{\chi}\in C_{0}^{\infty}(0,T;H^{1}(\mathbb{T}^{3})\cap L^{\infty}(\mathbb{T}^{3}))$, and summing over all $k\in \mathbb{N}_{0}$, give
\begin{align}\label{weak11}
	&\nonumber(\partial_{t,h}^{-}\mathbf{u}^{N}+\rho^{N}(\rho_{h}^{N})^{-1}\mathbf{u}^{N}\cdot\nabla \mathbf{u}^{N},\boldsymbol{\chi})_{Q_{T}}+((\rho_{h}^{N})^{-1}S(\phi_{h}^{N},\mathbb{ D}\mathbf{u}^{N}),\nabla\boldsymbol{\chi})_{Q_{T}}
	\\&\nonumber\quad+(\nabla(\rho_{h}^{N})^{-1}S(\phi_{h}^{N},\mathbb{ D}\mathbf{u}^{N}),\boldsymbol{\chi})_{Q_{T}}
	-\zeta((\rho_{h}^{N})^{-1}\rho^{N}p_{0}^{N},\mathrm{div}\boldsymbol{\chi})_{Q_{T}}+((\rho_{h}^{N})^{-1}\phi^{N}\nabla\mu_{p}^{0,N},\boldsymbol{\chi})_{Q_{T}}
	\\&\quad-\zeta(\nabla((\rho_{h}^{N})^{-1}\rho^{N}) p_{0}^{N},\boldsymbol{\chi})_{Q_{T}}
	+(\frac{\varepsilon\delta}{4\alpha}(\rho_{h}^{N})^{-1}\mathbf{u}^{N}p_{0}^{N},\boldsymbol{\chi})_{Q_{T}}=0
\end{align}
which and
\begin{align*}
	&(\partial_{t,h}^{-}\mathbf{u}^{N},\boldsymbol{\chi})_{Q_{T}}=-(\mathbf{u}^{N},\partial_{t,h}^{+}\boldsymbol{\chi})_{Q_{T}}
\end{align*}
deduce  that
\begin{align}\label{weak111}
	&\nonumber-(\mathbf{u}^{N},\partial_{t,h}^{+}\boldsymbol{\chi})_{Q_{T}}+(\rho^{N}(\rho_{h}^{N})^{-1}\mathbf{u}^{N}\cdot\nabla \mathbf{u}^{N},\boldsymbol{\chi})_{Q_{T}}+((\rho_{h}^{N})^{-1}S(\phi_{h}^{N},\mathbb{ D}\mathbf{u}^{N}),\nabla\boldsymbol{\chi})_{Q_{T}}
	\\&\nonumber\quad+(\nabla(\rho_{h}^{N})^{-1}S(\phi_{h}^{N},\mathbb{ D}\mathbf{u}^{N}),\boldsymbol{\chi})_{Q_{T}}
	-\zeta((\rho_{h}^{N})^{-1}\rho^{N}p_{0}^{N},\mathrm{div}\boldsymbol{\chi})_{Q_{T}}+((\rho_{h}^{N})^{-1}\phi^{N}\nabla\mu_{p}^{0,N},\boldsymbol{\chi})_{Q_{T}}
	\\&\quad-\zeta(\nabla((\rho_{h}^{N})^{-1}\rho^{N}) p_{0}^{N},\boldsymbol{\chi})_{Q_{T}}
	+(\frac{\varepsilon\delta}{4\alpha}(\rho_{h}^{N})^{-1}\mathbf{u}^{N}p_{0}^{N},\boldsymbol{\chi})_{Q_{T}}=0.
\end{align}

Similarly, one has
\begin{align}
	\label{weak12-1}
	&-(\rho^{N},\partial_{t,h}^{+}\psi)_{Q_{T}}-(\rho^{N}\mathbf{u}^{N}, \nabla\psi)_{Q_{T}}-\frac{\varepsilon\delta}{2\alpha}(p_{0}^N,\psi)_{Q_{T}}=0,
	\\&\label{weak12-2}
	(\phi^{N},\partial_{t,h}^{+}\psi)_{Q_{T}}+(\phi^{N}\mathbf{u}^{N}, \nabla\psi)_{Q_{T}}-(\nabla\mu_{p}^{0,N}, \nabla\psi)_{Q_{T}}=0
\end{align}
for $\psi\in C_{0}^{\infty}(0,T;H_{(0)}^{1}(\mathbb{T}^{3}))$ and
\begin{align}\label{weak13}
	&(\Lambda^{s}\phi^{N},\Lambda^{s}\psi)_{Q_{T}}-\big(\mu_{p}^{0,N}+\overline{\mu}-\Phi'(\phi^{N})+\frac{\kappa\phi^{N}}{2}+\frac{\kappa\phi_h^{N}}{2}-\alpha p_0^{N}, \psi\big)_{Q_{T}}=0
\end{align}
for $\psi\in C_{0}^{\infty}(0,T;H^{s}(\mathbb{T}^{3}))$.

Defining $D_{N}(t)$ by
\begin{align*}
	&D_{N}(t)=\int_{\mathbb{T}^{3}}S(\phi_{h}^{N},\mathbb{ D}\mathbf{u}^{N}):\mathbb{ D}\mathbf{u}^{N}\,\mathrm dx
	+\int_{\mathbb{T}^{3}}|\nabla \mu_{p}^{0,N}|^{2}\mathrm dx+\delta\int_{\mathbb{T}^{3}}|{p}_0^{N}|^{2}\mathrm dx.
\end{align*}
Let $E_{N}(t)$ be piecewise linear interpolation of $E(\phi_{k},\mathbf{u}_{k})$ at $t=kh\triangleq t_k,k\in\mathbb{N}_{0}$. Then \eqref{imes} imply that
\begin{align*}
	&-\frac{d}{dt}E_{N}(t)=\frac{E(\phi_{k},\mathbf{u}_{k})-E(\phi_{k+1},\mathbf{u}_{k+1})}{h}\geq D_{N}(t)
\end{align*}
for all $t\in(t_{k},t_{k+1})$. Hence, integrating by parts, we have
\begin{align}\label{im11-1}
	&\int_0^{T}E_{N}(t)\varphi' dt+E(\phi_{0},\mathbf{u}_{0})\geq \int_0^{T}D_{N}(t)\varphi(t) dt
\end{align}
for all $\varphi\in W^{1,1}(0,\infty)$ with $\varphi(T)=0$. By Lemma \ref{lem4}, we then obtain
\begin{align}
	&E(\phi^{N}(t),\mathbf{u}^{N}(t))+\int_0^{t}\int_{\mathbb{T}^{3}}S(\phi_{h}^{N},\mathbb{ D}\mathbf{u}^{N}):\mathbb{ D}\mathbf{u}^{N}\,\dxdt\nonumber
	\\&\quad+\int_0^{t}\int_{\mathbb{T}^{3}}|\nabla \mu_{p}^{0,N}|^{2}\dxdt+\delta\int_0^{t}\int_{\mathbb{T}^{3}}|{p}_{0}^{N}|^{2}\dxdt\nonumber \\&\leq E(\phi^{N}(0),\mathbf{u}^{N}(0))\leq E(\phi_0,\mathbf{u}_0).\label{N-energy}
\end{align}
Employing the bounds above, we have up to a subsequence
\begin{align}\label{weak convergence 1}
	(\mathbf{u}^{N},p_{0}^{N},\mu_{p}^{0,N})
	& \rightharpoonup_{N\rightarrow\infty}(\mathbf{u},p_{0},\mu_{p}^{0})\ &&\text{in $L^{2}(0,T;H^{1}(\mathbb{T}^{3})\times L^{2}(\mathbb{T}^{3})\times H^{1}(\mathbb{T}^{3}))$} ,
	\\ (\phi^{N},\mathbf{u}^{N})&\rightharpoonup^{\ast}_{N\rightarrow\infty}(\phi,\mathbf{u}) \  && \text{in  $L^{\infty}(0,T;H^{s}(\mathbb{T}^{3})\times L^{2}(\mathbb{T}^{3}))$}.
\end{align}

We also need the strong convergence of $(\phi^{N},\mathbf{u}^{N})$ to pass to the limit in the nonlinear parts. We define
$\widetilde{\rho}^{N}=\frac{1}{h}\chi[0,h]\ast_{t} \rho^{N}$ and $\widetilde{\mathbf{u}}^{N}=\frac{1}{h}\chi[0,h]\ast_{t} \mathbf{u}^{N}$ where $\chi[0,h]$ is a characteristic function on $[0,h]$. The convolution is only related to $t$.

We then get that $\partial_{t}\widetilde{\mathbf{u}}^{N}=\partial_{t,h}^{-}\mathbf{u}^{N}$ and
\begin{align}\label{im12u}
	&\|\widetilde{\mathbf{u}}^{N}-\mathbf{u}^{N}\|_{H^{-s}(\mathbb{T}^{3})}\lesssim h\|\partial_{t}\widetilde{\mathbf{u}}^{N}\|_{H^{-s}(\mathbb{T}^{3})},
	\\&\label{im12rho}\|\widetilde{\rho}^{N}-\rho^{N}\|_{H^{-s}(\mathbb{T}^{3})}\lesssim h\|\partial_{t}\widetilde{\rho}^{N}\|_{H^{-s}(\mathbb{T}^{3})}.
\end{align}
It follows from
\begin{align*}
	\rho^{N}(\rho_{h}^{N})^{-1}\mathbf{u}^{N}\cdot\nabla \mathbf{u}^{N}& \in L^{2}(0,T;L^{1}(\mathbb{T}^{3})),
	\\  (\rho_{h}^{N})^{-1} S(\phi_{h}^{N},\mathbb{ D}\mathbf{u}^{N})&\in L^{2}(0,T;L^{2}(\mathbb{T}^{3})),
	\\\nabla (\rho_{h}^{N})^{-1}\cdot S(\phi_{h}^{N},\mathbb{ D}\mathbf{u}^{N})&\in L^{2}(0,T;L^{1}(\mathbb{T}^{3})),
	\\(\rho_{h}^{N})^{-1}\rho^{N}p_{0}^{N}&\in L^{2}(0,T;L^{2}(\mathbb{T}^{3})),
	\\(\rho_{h}^{N})^{-1}\phi^{N}\nabla\mu_{p}^{0,N}&\in L^{2}(0,T;L^{2}(\mathbb{T}^{3})),
	\\ \nabla((\rho_{h}^{N})^{-1} \rho^{N})p_{0}^{N}&\in L^{2}(0,T;L^{\frac{6}{5}}(\mathbb{T}^{3})),
	\\(\rho_{h}^{N})^{-1}\mathbf{u}^{N}p_{0}^{N}&\in L^{2}(0,T;L^{1}(\mathbb{T}^{3}))
\end{align*}
and \eqref{weak11} that there holds $\partial_{t}\widetilde{\mathbf{u}}^{N}$ is uniformly bounded in $L^{2}(0,T; H^{-s}(\mathbb{T}^{3}))$ with $s > \frac{3}{2}$. By virtue of the Aubin--Lions compactness lemma, we have up to a subsequence
\begin{align*}
	& \widetilde{\mathbf{u}}^{N}\rightarrow_{N\rightarrow\infty}\mathbf{u} \ \  \text{in  $L^{2}(0,T; H^{q}(\mathbb{T}^{3}))$},\ 0\leq q<1,
\end{align*}
where we have used $\widetilde{\mathbf{u}}^{N}$ is uniformly bounded in $L^{2}(0,T; H^{1}(\mathbb{T}^{3}))$ and $H^{1}(\mathbb{T}^{3})\hookrightarrow\hookrightarrow H^{s}(\mathbb{T}^{3})(0\leq s<1)$. Moreover, by \eqref{im12u}, we get
\begin{align*}
	& \widetilde{\mathbf{u}}^{N}-\mathbf{u}^{N}\rightarrow_{N\rightarrow\infty}0 \ \  \text{in  $L^{2}(0,T; H^{-s}(\mathbb{T}^{3}))$}.
\end{align*}
Due to $\widetilde{\mathbf{u}}^{N},\mathbf{u}^{N}$ are uniformly bounded in $L^{\infty}(0,T;  L^{2}(\mathbb{T}^{3}))$, it follows
\begin{align*}
	& \widetilde{\mathbf{u}}^{N}-\mathbf{u}^{N}\rightarrow_{N\rightarrow\infty}0 \ \  \text{in  $L^{p}(0,T; H^{-s}(\mathbb{T}^{3}))$}
\end{align*}
for any $1<p<\infty$. Furthermore, since  $\widetilde{\mathbf{u}}^{N},\mathbf{u}^{N}$ are uniformly bounded in $L^{2}(0,T;  H^{1}(\mathbb{T}^{3}))$, we get
\begin{align}\label{Strong convergence 1}
	& \mathbf{u}^{N}\rightarrow_{N\rightarrow\infty}\mathbf{u}\ \  \text{in  $L^{2}(0,T; H^{q}(\mathbb{T}^{3}))$},\ 0\leq q<1,
\end{align}
where we have used $\|\bu^{N}-u\|_{L^{2}(0,T;H^{q}(\mathbb{T}^{3}))}\leq \|\bu^{N}-u\|_{L^{2}(0,T;H^{-s}(\mathbb{T}^{3}))}^{\theta}\|\bu^{N}-u\|_{L^{2}(0,T;H^{1}(\mathbb{T}^{3}))}^{1-\theta}(\theta=\frac{1-q}{s+1})$. Finally, thanks to Lemma \ref{lem1}, the regularity $\widetilde{\mathbf{u}}^{N}\in L^{\infty}(0,T;  L^{2}(\mathbb{T}^{3}))$, $\widetilde{\mathbf{u}}^{N}$ converges weakly in $H^{1}(0,T; H^{-s}(\mathbb{T}^{3}))\hookrightarrow BUC(0,T; H^{-s}(\mathbb{T}^{3}))$ for $s>\frac{3}{2}$ and $\widetilde{\mathbf{u}}^{N}|_{t=0}=\mathbf{u}_{0}$, we can obtain $\mathbf{u}\in BC_{\omega}(0,T;  L^{2}(\mathbb{T}^{3}))$ and \eqref{appinitial data-1}.

With the help of \eqref{N-energy}, we obtain $\partial_{t}\widetilde{\rho}^{N}$ is uniformly bounded in $L^{2}(0,T;H^{-1}(\mathbb{T}^{3}))$ and $\widetilde{\rho}^{N}$ is uniformly bounded in $L^{\infty}(0,T;  H^{s}(\mathbb{T}^{3}))$. Then, due to the Aubin-Lions compactness lemma, we can find up to a subsequence that
\begin{align}\label{im13}
	& \widetilde{\rho}^{N}\rightarrow _{N\rightarrow\infty}  \widetilde{\rho}, \ \  \text{in  $L^{2}(Q_{T})$ }
\end{align}
and  $\widetilde{\rho}\in L^{\infty}(Q_{T})$ for all $0<T<\infty$.  Moreover, by \eqref{im12rho}, we have
\begin{align*}
	& \widetilde{\rho}^{N}-\rho^{N}\rightarrow _{N\rightarrow\infty} 0  , \ \  \text{in  $L^{2}(0,T;H^{-1}(\mathbb{T}^{3}))$.}
\end{align*}
Since $\widetilde{\rho}^{N}, \rho^{N}$ are both uniformly bounded in $L^{\infty}(0,T;  H^{s}(\mathbb{T}^{3}))$ and \eqref{im13}, we have
\begin{align*}
	& \rho^{N}\rightarrow _{N\rightarrow\infty}  \widetilde{\rho}, \ \  \text{in  $L^{2}(Q_{T})$ }
\end{align*}
which and  $\rho^{N}=\frac{\varepsilon}{2}\phi^{N}+1+\frac{\varepsilon}{2}$ imply
\begin{align*}
	& \phi^{N}\rightarrow _{N\rightarrow\infty}  \phi\triangleq \frac{2}{\varepsilon}\widetilde{\rho}-1-\frac{2}{\varepsilon}, \ \  \text{in  $L^{2}(Q_{T})$,}
\end{align*}
that is,  $\widetilde{\rho}=\rho(\phi)$.

Moreover, Since $\widetilde{\rho}^{N}\in H^{1}(0,T;  H^{-1}(\mathbb{T}^{3}))\hookrightarrow BUC(0,T; H^{-1}(\mathbb{T}^{3}))$ and
$\widetilde{\rho}^{N}$ is bounded in $L^{\infty}(0,T;  H^{s}(\mathbb{T}^{3}))$, we get 
$\rho \in BC_{\omega}(0,T; H^{s}(\mathbb{T}^{3}))$(cf. Lemma \ref{lem1}). Hence, $\phi \in BC_{\omega}(0,T;  H^{s}(\mathbb{T}^{3}))$ and \eqref{appinitial data-2} holds.

To prove Theorem \ref{main theorem2}, the remaining part is  to verify \eqref{weak1}-\eqref{ConEnL2} in what follows. By passing to the limit in \eqref{weak12-1} and \eqref{weak12-2} respectively we can immediately deduce \eqref{weak2} and \eqref{weak3}.

It follows from $\mu_{p}^{0,N}+\overline{\mu^{N}}-\Phi'(\phi^{N})+\frac{\kappa\phi^{N}}{2}+\frac{\kappa\phi_h^{N}}{2}-\alpha p_{0,N}$  converges weakly to $\mu_{p}^{0}+\overline{\mu}-F'(\phi)-\alpha p_{0}$ in $L^{2}(Q_{T})$ that
\begin{align*}
	(\Lambda^{2s}\phi^{N},\phi^{N})_{X_{T}', X_{T}}&=(\mu_{p}^{0,N}+\overline{\mu^{N}}-\Phi'(\phi^{N})+\frac{\kappa\phi^{N}}{2}
	+\frac{\kappa\phi_h^{N}}{2}-\alpha p_{0,N}, \phi^{N})_{Q_{T}}\\&\rightarrow _{N\rightarrow\infty}
	(\mu_{p}^{0}+\overline{\mu}-F'(\phi)-\alpha p_{0}, \phi)_{Q_{T}},
\end{align*}
where $X_{T}=L^{2}(0,T;H^{s}(\mathbb{T}^{3}))$. By Lemma \ref{lem5}  one has  $\Lambda^{2s}\phi=\mu_{p}^{0}+\overline{\mu}-F'(\phi)-\alpha p_{0}$ and then
\begin{align}\label{Strong convergence}
	&\int_{0}^T\int_{\mathbb{T}^{3}}|\Lambda^{s} \phi^{N}|^{2} \dxdt \rightarrow _{N\rightarrow\infty}\int_{0}^T\int_{\mathbb{T}^{3}}|\Lambda^{s} \phi|^{2} \dxdt
\end{align}
hence $\phi^{N}$ converges to $\phi$ strongly in $L^{2}(0,T;  H^{s}(\mathbb{T}^{3}))$. Therefore \eqref{weak4} holds by passing to the limit $N\rightarrow\infty$ in \eqref{weak13}.

Thanks to $\nabla (\rho_{h}^{N})^{-1}\cdot S(\phi_{h}^{N},\mathbb{ D}\mathbf{u}^{N}) \rightharpoonup _{N\rightarrow\infty} \nabla \rho^{-1}\cdot S(\phi,\mathbb{ D}\mathbf{u})$ in $L^{1}(0,T;  L^{1}(\mathbb{T}^{3}))$, we can get  \eqref{weak1} by passing to limit $N\rightarrow\infty$ in \eqref{weak111}.

Finally, up to a subsequence, there holds
\begin{align*}
	& E(\mathbf{u}^{N}(x,t),\phi^{N}(x,t))\rightarrow _{N\rightarrow\infty}E(\mathbf{u}(x,t),\phi(x,t)), \quad\text{for all $0<t< \infty$}.
\end{align*} which comes from the facts that $\mathbf{u}^{N}$ converges strongly to $\mathbf{u}$ in $L^{2}(0,T;  L^{2}(\mathbb{T}^{3}))$,  $\phi^{N}$ converges to $\phi$ almost everywhere and $\phi^{N}$ converges strongly to $\phi$ in $L^{2}(0,T;  H^{s}(\mathbb{T}^{3}))$. Then it easy to get $\int_{\mathbb{T}^{3}}\phi\, \mathrm{d} x=\int_{\mathbb{T}^{3}}\phi_0\, \mathrm{d} x$ for a.e.~$t \in (0,T]$. 

Noting that $\phi^{N}$ converges  to $\phi$ almost everywhere, we obtain
\begin{align*}
	&\nonumber \liminf_{N\to\infty}\int_{0}^T D_{N}(t)\varphi(t) \,\dt
	\geq \int_{0}^T D(t)\varphi(t)  \,\dt \ \text{for $\varphi\in W^{1,1}(0,T)$ and $\varphi\geq0$ }
\end{align*}
with
\begin{align*}
	D(t)=\int_{\mathbb{T}^{3}}S(\phi,\mathbb{ D}\mathbf{u}):\mathbb{ D}\mathbf{u}\,\mathrm dx
	+\int_{\mathbb{T}^{3}}|\nabla \mu_{p}^{0}|^{2}\,\mathrm dx+\delta\int_{\mathbb{T}^{3}}|{p}^{0}|^{2}\,\mathrm dx.
\end{align*}
Hence, we can pass to the limit in \eqref{im11-1} to get  that
\begin{align}\label{im13-1}
	&\int_0^{T}E(t)\varphi'(t) \,\dt+E(\phi_{0},\mathbf{u}_{0})\geq \int_0^{T}D_{N}(t)\varphi(t) \,\dt
	\ \text{for $\varphi\in W^{1,1}(0,T)$ and $\varphi\geq0$ }
\end{align}
which and  lemma \ref{lem4} imply  \eqref{ConEnL2} holds.

Consequently we complete the proof of Theorem \ref{main theorem2}.

\section{Existence of Weak Solutions: Proof of Theorem \ref{thm:main}}
\label{sec:mainthm}
In this section, $R$ in Lemma \ref{lem2} will be chosen such that $E(\mathbf{u}_{0}, \phi_{0})<R$ as in Theorem \ref{main theorem2}. We also take $(\mathbf{u}_{\delta},p_{0,\delta}, \phi_{\delta},\mu_{p,\delta}^{0})\equiv (\mathbf{u},p_{0}, \phi,\mu_{p}^{0})$ as the solution of \eqref{model4-1}--\eqref{model4-5} in Theorem \ref{main theorem2}.
%

To proceed, we apply the Helmholtz decomposition to $\mathbf{u}_{\delta}$, i.e., we decompose
\begin{align}\mathbf{u}_{\delta}=\mathbb{P}\mathbf{u}_{\delta}+\nabla\big( \mathbb{G}(\mathbf{u}_{\delta})\big),\label{Helmholtz decomposition}\end{align}
where $\mathbb{G}$ is defined by
\begin{align*}
	&\Delta \mathbb{G}(\mathbf{u}_{\delta})=\mathrm{div}\mathbf{u}_{\delta},\quad \text{in}\ \mathbb{T}^{3}
\end{align*}
and $\int_{\mathbb{T}^{3}} \mathbb{G}(\mathbf{u}_{\delta})  \d x=0$. Then we introduce a new pressure denoted by $p_{1,\delta}$
\begin{align}
	p_{1,\delta}=\zeta p_{0,\delta}+\partial_{t}\mathbb{G}(\mathbf{u}_{\delta}).\label{new pressure}
\end{align}
In view of \eqref{weak1}, one then has
\begin{align}\label{newweak1}
	&\nonumber(\mathbb{P}\mathbf{u}_{\delta},\partial_{t}\boldsymbol{\varphi})_{Q_{T}}+(\mathbf{u}_{\delta}\cdot\nabla\mathbf{u}_{\delta},\boldsymbol{\varphi})_{Q_{T}}-(\rho_{\delta}^{-1}S(\phi_{\delta},\mathbb{ D}\mathbf{u}_{\delta}),\nabla\boldsymbol{\varphi})_{Q_{T}}-(S(\phi_{\delta},\mathbb{ D}\mathbf{u}_{\delta}),\nabla\rho_{\delta}^{-1}\otimes\boldsymbol{\varphi})_{Q_{T}}
	\\&\qquad
	+(p_{1,\delta},\mathrm{div}\boldsymbol{\varphi})_{Q_{T}}-(\rho_{\delta}^{-1}\phi_{\delta} \nabla \mu_{p,\delta}^{0},\boldsymbol{\varphi})_{Q_{T}}
	-(\frac{\varepsilon\delta}{4\alpha}  p_{0,\delta} \mathbf{u}_{\delta},\boldsymbol{\varphi}\rho_{\delta}^{-1})_{Q_{T}}=0.
\end{align}

In the rest of this section, we show a series of uniform estimates with respect to  $\delta>0$.

\begin{lemma}\label{lem5.1}
	For $1<r<\frac{3}{2}$  there exists a constant $C$ independent of $\delta > 0$ such that
	\begin{align*}
		\|p_{1,\delta}\|_{L^{2}(0,T;L^{r}(\mathbb{T}^{3}))}\leq C.
	\end{align*}
\end{lemma}
\begin{proof}
	By choosing $\boldsymbol{\varphi}=\eta(t)\nabla\psi$ with $\psi\in W_{(0)}^{2,r'}(\mathbb{T}^{3})(\frac{1}{r}+\frac{1}{r'}=1)$  and $\eta(t)\in C_{0}^{\infty}(0,T)$  in \eqref{newweak1}, we have
	\begin{align*}
		&\int_0^{T}\int_{\mathbb{T}^{3}} p_{1,\delta}\Delta\psi \eta(t) \dx\dt
		\\&=\int_0^{T}\int_{\mathbb{T}^{3}}\eta(t)(\mathbf{u}_{\delta}\cdot\nabla
		\mathbf{u}_{\delta})\cdot\nabla\psi  \dx\dt+\int_0^{T}\int_{\mathbb{T}^{3}}\eta(t)\rho_{\delta}^{-1}\big(S(\phi_{\delta},\mathbb{ D}\mathbf{u}_{\delta}):\nabla^{2}\psi\big)  \dx\dt
		\\&\quad+\int_0^{T}\int_{\mathbb{T}^{3}}\eta(t)\big(\rho_{\delta}^{-1}\phi_{\delta} \nabla \mu_{p,\delta}^{0}+\nabla\rho_{\delta}^{-1}\cdot\big(S(\phi_{\delta},\mathbb{ D}\mathbf{u}_{\delta})\big)+\frac{\varepsilon\delta}{4\alpha\rho_{\delta}}  p_{0,\delta} \mathbf{u}_{\delta}\big)\cdot\nabla\psi  \dx\dt
	\end{align*}
	which implies for almost every $0<t<T$
	\begin{align*}
		&\int_{\mathbb{T}^{3}} p_{1,\delta}\Delta\psi  \dx
		\\&=\int_{\mathbb{T}^{3}}(\mathbf{u}_{\delta}\cdot\nabla
		\mathbf{u}_{\delta})\cdot\nabla\psi  \dx+\int_{\mathbb{T}^{3}}\rho_{\delta}^{-1}S(\phi_{\delta},\mathbb{ D}\mathbf{u}_{\delta}):\nabla^{2}\psi  \dx
		\\&\quad+\int_{\mathbb{T}^{3}}\big(\rho_{\delta}^{-1}\phi_{\delta} \nabla \mu_{p,\delta}^{0}+\nabla\rho_{\delta}^{-1}\cdot\big(S(\phi_{\delta},\mathbb{ D}\mathbf{u}_{\delta})\big)+\frac{\varepsilon\delta}{4\alpha \rho_{\delta}}  p_{0,\delta} \mathbf{u}_{\delta}\big)\cdot\nabla\psi  \dx.
	\end{align*}
	
	Noting that
	\begin{align*}
		&\int_{\mathbb{T}^{3}}(\mathbf{u}_{\delta}\cdot\nabla
		\mathbf{u}_{\delta})\cdot\nabla\psi  \dx+\int_{\mathbb{T}^{3}}\rho_{\delta}^{-1}S(\phi_{\delta},\mathbb{ D}\mathbf{u}_{\delta}):\nabla^{2}\psi  \dx
		\\&\quad+\int_{\mathbb{T}^{3}}\big(\rho_{\delta}^{-1}\phi_{\delta} \nabla \mu_{p,\delta}^{0}+\nabla\rho_{\delta}^{-1}\cdot\big(S(\phi_{\delta},\mathbb{ D}\mathbf{u}_{\delta})\big)+\frac{\varepsilon\delta}{4\alpha \rho_{\delta}}  p_{0,\delta} \mathbf{u}_{\delta}\big)\cdot\nabla\psi  \dx\\ &\leq C\bigg(
		\|\mathbf{u}_{\delta}\cdot\nabla
		\mathbf{u}_{\delta}\|_{L^{1}(\mathbb{T}^{3})}+\|\rho_{\delta}^{-1}S(\phi_{\delta},\mathbb{ D}\mathbf{u}_{\delta})\|_{L^{r}(\mathbb{T}^{3})}
		\\&\quad+\|\rho_{\delta}^{-1}\phi_{\delta} \nabla \mu_{p,\delta}^{0}+\nabla\rho_{\delta}^{-1}\cdot S(\phi_{\delta},\mathbb{ D}\mathbf{u}_{\delta})+\frac{\varepsilon\delta}{4\alpha\rho_{\delta}}  p_{0,\delta} \mathbf{u}_{\delta}
		\|_{L^{1}(\mathbb{T}^{3})}\bigg)\|\psi\|_{W^{2,r'}(\mathbb{T}^{3})}
		\\&\leq C(\theta,R)\bigg(\|\mathbf{u}_{\delta}\|_{H^{1}(\mathbb{T}^{3})}+\|\mu_{p,\delta}^{0}\|_{H^{1}(\mathbb{T}^{3})}+\| \delta p_{0,\delta}\|_{L^{2}(\mathbb{T}^{3})})\bigg)\|\psi\|_{W^{2,r'}(\mathbb{T}^{3})},
	\end{align*}
	where we have used $W^{1,r'}(\mathbb{T}^{3})\hookrightarrow L^{\infty}(\mathbb{T}^{3})(r'>3)$.
	
	Accordingly,  $p_{1,\delta}$ is the very weak  solution of the Laplace with right-hand side $F_{\delta}(t)$(cf. Lemma \ref{lem3}) and
	\begin{align*}
		\|F_{\delta}(t)\|_{(W_{(0)}^{2,r'}(\mathbb{T}^{3}))'}\leq C(\theta,R)(\|\mathbf{u}_{\delta}\|_{H^{1}(\mathbb{T}^{3})}+\|\mu_{p,\delta}^{0}\|_{H^{1}(\mathbb{T}^{3})}+\| \delta p_{0,\delta}\|_{L^{2}(\mathbb{T}^{3})}).
	\end{align*}
	According to Lemma \ref{lem3}, one derives
	\begin{align*}
		&\|p_{1,\delta}\|_{L^{2}(0,T;L^{r}(\mathbb{T}^{3}))}\leq C(\theta,R)\|F_{\delta}(t)\|_{L^{2}(0,T;(W_{(0)}^{2,r'}(\mathbb{T}^{3}))')}\leq C(\theta,R).
	\end{align*}
	Thus, we complete the proof.
\end{proof}

			
			For the following argument, we show the following result.
			\begin{lemma}\label{diverge free}
				For $1<r<\frac{3}{2}$, the following holds
				\begin{align*}
					\|\partial_{t}\mathbb{P}(\mathbf{u}_{\delta})\|_{L^{2}(0,T;W^{-1,r}(\mathbb{T}^{3}))}\leq C,
				\end{align*}
				where $C$ is independent of $\delta$.
			\end{lemma}
			\begin{proof}
				It follows from \eqref{ConEnL2}, \eqref{newweak1}, the H\"{o}lder's inequality and Sobolev embedding that there holds
				\begin{align}\label{d5.1}
					\nonumber(\mathbb{P}(\mathbf{u}_{\delta}),\partial_{t}\boldsymbol{\varphi})_{Q_{T}} &=(\mathbf{u}_{\delta}\cdot\nabla
					\mathbf{u}_{\delta},\boldsymbol{\varphi})_{Q_{T}}+( S(\phi_{\delta},\mathbb{ D}\mathbf{u}_{\delta}),\nabla\frac{1}{\rho_{\delta}}\otimes\boldsymbol{\varphi})_{Q_{T}}+( \frac{1}{\rho_{\delta}}S(\phi_{\delta},\mathbb{ D}\mathbf{u}_{\delta}),\nabla\boldsymbol{\varphi})_{Q_{T}}
					\nonumber \\&\quad+(\nabla p_{1,\delta},\boldsymbol{\varphi})_{Q_{T}}
					+(\frac{1}{\rho_{\delta}}\phi_{\delta} \nabla \mu_{p,\delta}^{0},\boldsymbol{\varphi})_{Q_{T}}
					+(\frac{\varepsilon\delta}{4\alpha\rho_{\delta}}p_{0,\delta} \mathbf{u}_{\delta},\boldsymbol{\varphi})_{Q_{T}}
					\nonumber \\  \nonumber&\lesssim\|\mathbf{u}_{\delta}\|_{L^{\infty}(0,T;L^{2}(\mathbb{T}^{3}))}\|\nabla\mathbf{u}_{\delta}\|_{L^{2}(0,T;L^{2}(\mathbb{T}^{3}))}\|\boldsymbol{\varphi}\|_{L^{2}(0,T;L^{\infty}(\mathbb{T}^{3}))}
					\\& \nonumber\quad+\|\phi_{\delta}\|_{L^{\infty}(0,T;L^{\infty}(\mathbb{T}^{3}))}\|\mathbb{ D}\mathbf{u}_{\delta}\|_{L^{2}(0,T;L^{2}(\mathbb{T}^{3}))}\|\nabla\rho_{\delta}^{-1}\|_{L^{\infty}(0,T;L^{3}(\mathbb{T}^{3}))}\|\boldsymbol{\varphi}\|_{L^{2}(0,T;L^{6}(\mathbb{T}^{3}))}
					\\& \nonumber\quad+\|\phi_{\delta}\|_{L^{\infty}(0,T;L^{\infty}(\mathbb{T}^{3}))}\|\mathbb{ D}\mathbf{u}_{\delta}\|_{L^{2}(0,T;L^{2}(\mathbb{T}^{3}))}\|\nabla\boldsymbol{\varphi}\|_{L^{2}(0,T;L^{2}(\mathbb{T}^{3}))}
					\\&  \nonumber\quad+\|p_{1,\delta}\|_{L^{2}(0,T;L^{r}(\mathbb{T}^{3}))}\|\nabla\boldsymbol{\varphi}\|_{L^{2}(0,T;L^{r'}(\mathbb{T}^{3}))}
					\\&  \nonumber\quad+\|\phi_{\delta}\rho_{\delta}^{-1}\|_{L^{\infty}(0,T;L^{\infty}(\mathbb{T}^{3}))}\|\nabla \mu_{p,\delta}^{0}\|_{L^{2}(0,T;L^{2}(\mathbb{T}^{3}))}\|\boldsymbol{\varphi}\|_{L^{2}(0,T;L^{2}(\mathbb{T}^{3}))}
					\\& \nonumber\quad+\|\rho_{\delta}^{-1}\|_{L^{\infty}(0,T;L^{\infty}(\mathbb{T}^{3}))}\|\delta p_{0,\delta}\|_{L^{2}(0,T;L^{2}(\mathbb{T}^{3}))}\|\boldsymbol{\varphi}\|_{L^{2}(0,T;L^{2}(\mathbb{T}^{3}))}
					\\&\lesssim \|\boldsymbol{\varphi}\|_{L^{2}(0,T;W^{1,r'}(\mathbb{T}^{3}))},
				\end{align}
				where $\boldsymbol{\varphi} \in C_{0}^{\infty}(0,T;W^{1,r'}(\mathbb{T}^{3}))$.
				This yields the desired conclusion.
			\end{proof}
			
			Moreover, we will verify the following strong convergence results.
			\begin{lemma}\label{lem5.3}
				There exists a sequence, still denoted by $\mathbf{u}_{\delta}$, such that
				\begin{align*}
					(\mathbf{u}_{\delta}, \mathbb{P}(\mathbf{u}_{\delta}), \mathbb{G}(\mathbf{u}_{\delta}))\rightarrow_{\delta\rightarrow0}(\mathbf{u},\mathbb{G}(\mathbf{u})) \quad\text{in $L^{2}(0,T;L^{2}(\mathbb{T}^{3})\times L^{2}(\mathbb{T}^{3})\times H^{1}(\mathbb{T}^{3}))$}
				\end{align*}
				for all $0<T<\infty$.
			\end{lemma}
			\begin{proof}
				By substituting \eqref{new pressure} into \eqref{weak4} we arrive at
				\begin{align*}
					(\alpha\partial_{t} \mathbb{G}(\mathbf{u}_{\delta}),\psi)_{Q_{T}}=(\alpha p_{1,\delta}-\zeta\mu_{p,\delta}^{0}-\zeta\overline{\mu}_{\delta}+\zeta F'(\phi_{\delta}),\psi)_{Q_{T}}
					+\zeta(\Lambda^{s}\phi,\Lambda^{s}\psi)_{Q_{T}}
				\end{align*}
				which together with  \eqref{ConEnL2} and Lemma \ref{lem5.1} yields
				\begin{align*}
					\|\partial_{t}\mathbb{G}(\mathbf{u}_{\delta})\|_{L^{2}(0,T;H^{-s}(\mathbb{T}^{3}))}\leq C.
				\end{align*}
				With the help of Lemma \ref{diverge free}, one obtains
				\begin{align}
					\label{eqs:dt-u-delta}
					\partial_{t}\mathbf{u}_{\delta}\ \text{is uniformly bounded in}\ L^{2}(0,T;H^{-1-s}(\mathbb{T}^{3})).
				\end{align}
				
				By virtue of the Aubin-Lions lemma, there exists a sequence, still denoted by $\mathbf{u}_{\delta}$ such that   $\mathbf{u}_{\delta}\rightarrow_{\delta\rightarrow0}\mathbf{u}$ in $L^{2}(Q_{T})$. Furthermore, in view of $\|\mathbb{G}(\mathbf{u}_{\delta})-\mathbb{G}(\mathbf{u})\|_{L^{2}(0,T;H^{1}(\mathbb{T}^{3}))}\leq C \|\mathbf{u}_{\delta}-\mathbf{u}\|_{L^{2}(0,T,L^{2}(\mathbb{T}^{3}))}$, one immediately gets $\mathbb{G}(\mathbf{u}_{\delta})\rightarrow_{\delta\rightarrow0}\mathbb{G}(\mathbf{u})$ in $L^{2}(0,T;H^{1}(\mathbb{T}^{3}))$ and $\mathbb{P}(\mathbf{u}_{\delta})\rightarrow_{\delta\rightarrow0}\mathbb{P}(\mathbf{u})$ in $L^{2}(0,T;L^{2}(\mathbb{T}^{3}))$.
				
				Therefore we conclude  the desired strong convergence.
			\end{proof}
			
			For $\phi_{\delta}$ we have  the following pointwise convergence.
			\begin{lemma}\label{lem5.5}
				There exists a sequence, still denoted by $\phi_{\delta}$ such that $\phi_{\delta}\rightarrow_{\delta\rightarrow0}\phi$ almost everywhere in $Q_T$ for all $0<T<\infty$.
			\end{lemma}
			\begin{proof}
				In view of \eqref{model4-3}, \eqref{model4-6} and \eqref{ConEnL2}, we get $\partial_{t}\phi_{\delta}$ is uniformly bounded in  $L^{2}(0,T;L^{2}(\mathbb{T}^{3}))$ and $\phi_{\delta}$ is uniformly bounded in $L^{\infty}(0,T;H^{s}(\mathbb{T}^{3}))$. Hence, by the Aubin-Lions lemma($L^{\infty}$ version), one concludes $\phi_{\delta}\rightarrow_{\delta\rightarrow0}\phi$ almost everywhere in $Q_T$ for a suitable sequence for all $0<T<\infty$.
			\end{proof}
			
			In the following, we obtain higher order spatial regularity for the order parameter $\phi_\delta$ via the momentum equation.
			\begin{lemma}\label{lem:higher eatimate}
				
				Let $s>\frac{3}{2}$, then it follows
				\begin{align*}
					\|\phi_\delta\|_{L^{2}(0,T;H^{s+\frac{\gamma}{2}}(\mathbb{T}^{3}))}
					\leq C
				\end{align*}
				for  $0\leq\gamma\leq 1$, where $C>0$ depends on the initial data and $ T> 0$, but is independent of $\delta$ and $\alpha$.
			\end{lemma}
			\begin{proof}
				
				It follows from \eqref{ConEnL2} that
				\begin{align*}
					&\mathbf{u}_\delta\cdot\nabla\mathbf{u}_\delta, p_{0,\delta} \mathbf{u}_\delta\in L^\frac{4}{3}(0,T;L^{\frac{6}{5}}(\mathbb{T}^3)),
					\\& \rho_\delta^{-1}S(\phi_{\delta},\mathbb{ D}\mathbf{u}_\delta),\rho_\delta^{-1}\phi_\delta \nabla \mu_{p,\delta}^{0}\in L^2(0,T;L^{2}(\mathbb{T}^3)),
					\\& \nabla\rho_\delta^{-1}\cdot S(\phi_\delta,\mathbb{ D}\mathbf{u}_\delta)\in L^2(0,T;L^{\frac{6}{5}}(\mathbb{T}^3)).
				\end{align*}
				Therefore  \eqref{weak1} holds for  $\boldsymbol{\varphi}\in  H^1(0,T;H^{-1}(\mathbb{T}^3))\cap L^4(0,T;L^{6}(\mathbb{T}^3))\cap L^2(0,T;H^{1}(\mathbb{T}^3))$.
				
				Due to the Sobolev embedding and elliptic regularity theory  we have
				\begin{align*}
					&\|\nabla\Delta^{-1}\Lambda^{\gamma}(\psi\phi_{\delta})\|_{L^4(0,T;L^{6}(\mathbb{T}^3))} \leq
					C\|\Lambda^{\gamma}(\psi\phi_{\delta})\|_{L^{\infty}(0,T;L^{2}(\mathbb{T}^{3}))}
					\leq C\|\psi\phi_{\delta}\|_{L^{\infty}(0,T;H^{\gamma}(\mathbb{T}^{3}))},
					\\&
					\|\nabla\Delta^{-1}\Lambda^{\gamma}(\psi\phi_{\delta})\|_{L^{2}(0,T;H^{1}(\mathbb{T}^{3}))} \leq
					C\|\Lambda^{\gamma}(\psi\phi_{\delta})\|_{L^{2}(0,T;L^{2}(\mathbb{T}^{3}))}
					\leq C\|\psi\phi_{\delta}\|_{L^{2}(0,T;H^{\gamma}(\mathbb{T}^{3}))},
					\\&
					\|\nabla\Delta^{-1}\Lambda^{\gamma}(\psi\phi_{\delta})\|_{H^{1}(0,T;H^{-1}(\mathbb{T}^{3}))}\leq
					C\|\Delta^{-1}\Lambda^{\gamma}(\psi\phi_{\delta})\|_{H^{1}(0,T;L^{2}(\mathbb{T}^{3}))}
					\leq C\|\psi\phi_{\delta}\|_{H^{1}(0,T;L^{2}(\mathbb{T}^{3}))}
				\end{align*}
				for $0\leq\gamma\leq1$ and $\psi\in C_{0}^{\infty}(0,T)$.

				Consequently we arrive at
				\begin{align}\label{weak1111}
					&\nonumber(\mathrm{div}\mathbf{u}_{\delta},\partial_{t}(\Delta^{-1}\Lambda^{\gamma}(\psi\phi_{\delta})))_{Q_{T}}-(\mathbf{u}_{\delta}\cdot\nabla \mathbf{u}_{\delta},\nabla\Delta^{-1}\Lambda^{\gamma}(\psi\phi_{\delta}))_{Q_{T}}
					\\&\nonumber\quad+(\nabla(\rho_{\delta})^{-1}S(\phi_{\delta},\mathbb{ D}\mathbf{u}_{\delta}),\nabla\Delta^{-1}\Lambda^{\gamma}(\psi\phi_{\delta}))_{Q_{T}}+((\rho_{\delta}^{-1}S(\phi_{\delta},\mathbb{ D}\mathbf{u}_{\delta}),\nabla(\nabla\Delta^{-1}\Lambda^{\gamma}(\psi\phi_{\delta})))_{Q_{T}}
					\\&\nonumber\quad -\zeta(p_{0,\delta},\mathrm{div}(\nabla\Delta^{-1}\Lambda^{\gamma}(\psi\phi_{\delta})))_{Q_{T}}
					+((\rho_{\delta})^{-1}\phi_{\delta}\nabla\mu_{p,\delta}^{0},\nabla\Delta^{-1}\Lambda^{\gamma}(\psi\phi_{\delta}))_{Q_{T}}
					\\&\quad
					+(\frac{\varepsilon\delta}{4\alpha}(\rho_{\delta})^{-1}\mathbf{u}_{\delta}p_{0,\delta},\nabla\Delta^{-1}\Lambda^{\gamma}(\psi\phi_{\delta}))_{Q_{T}}=0.
				\end{align}
				
				We  submit  $p_{0,\delta}=\frac{1}{\alpha}(\mu_{p,\delta}^{0}-\Lambda^{2s}\phi_{\delta}-F'(\phi_{\delta})+\bar{\mu}_{\delta})$ into \eqref{weak1111} and then get
				\begin{align}
					& (\zeta\alpha^{-1}\Lambda^{s+\frac{\gamma}{2}}\phi_{\delta},\Lambda^{s+\frac{\gamma}{2}}(\psi\phi_{\delta}))_{Q_{T}}
					\nonumber \\& =
					(\mathbf{u}_{\delta},\nabla\Delta^{-1}\Lambda^{\gamma}\partial_t (\psi\phi_{\delta}))_{Q^{T}}
					+(\mathbf{u}_{\delta}\cdot\nabla
					\mathbf{u}_{\delta},\nabla\Delta^{-1}\Lambda^{\gamma}(\psi\phi_{\delta}))_{Q_{T}}
					\nonumber \\&\quad+((\rho_{\delta}^{-1}S(\phi_{\delta},\mathbb{ D}\mathbf{u}_{\delta}),\nabla(\nabla\Delta^{-1}\Lambda^{\gamma}(\psi\phi_{\delta})))_{Q_{T}}
					\nonumber \\&\quad+(\nabla(\rho_{\delta})^{-1}S(\phi_{\delta},\mathbb{ D}\mathbf{u}_{\delta}),\nabla\Delta^{-1}\Lambda^{\gamma}(\psi\phi_{\delta}))_{Q_{T}}
					\nonumber\\&\quad +(\phi_{\delta}\rho_{\delta}^{-1} \nabla \mu_{p,\delta},\nabla\Delta^{-1}\Lambda^{\gamma}(\psi\phi_{\delta}))_{Q_{T}}
					\nonumber\\& \quad+\bigg(\frac{\varepsilon\delta}{4\alpha\rho_{\delta}}  p_{0,\delta} \mathbf{u}_{\delta},\nabla\Delta^{-1}\Lambda^{\gamma}(\psi\phi_{\delta})\bigg)_{Q_{T}}
					\nonumber \\&\quad-\frac{\zeta}{\alpha}(\nabla(\mu_{p,\delta}^{0}-F'(\phi_{\delta})),\nabla\Delta^{-1}\Lambda^{\gamma}(\psi\phi_{\delta}))_{Q_{T}}
					\nonumber \\&\eqqcolon\sum_{i=1}^{7}J_{i}.\label{lemh55-5}
				\end{align}

				Furthermore, in view of \eqref{model4-3} and \eqref{ConEnL2} we can find
				\begin{align}\|\partial_{t}\phi_{\delta}\|_{L^{2}(0,T;L^{2}(\mathbb{T}^{3}))}&=\|\mathrm{div}(\phi_{\delta}  \mathbf{u}_{\delta})-\Delta\mu_{p,\delta}^{0}\|_{L^{2}(0,T;L^{2}(\mathbb{T}^{3}))}
					\nonumber\\&=\|\phi_{\delta}\mathrm{div}\mathbf{u}_{\delta}+\nabla\phi_{\delta}\cdot  \mathbf{u}_{\delta}-\alpha^{-1}(\mathrm{div}  \mathbf{u}_{\delta} +\delta p_{0,\delta})\|_{L^{2}(0,T;L^{2}(\mathbb{T}^{3}))}
					\nonumber\\&\lesssim \|\phi_{\delta}\|_{L^{\infty}(Q_T))}
					\|\mathrm{div}\mathbf{u}_{\delta}\|_{L^{2}(Q_T)}+\|\nabla\phi_{\delta}\|_{L^{\infty}(0,T;L^{3}(\mathbb{T}^{3}))}
					\|\mathbf{u}_{\delta}\|_{L^{2}(0,T;L^{6}(\mathbb{T}^{3}))}
					\nonumber\\& \quad+\alpha^{-1}\|\mathrm{div}  \mathbf{u}_{\delta}\|_{L^{2}(0,T;L^{2}(\mathbb{T}^{3}))}
					+\alpha^{-1}\|\delta p_{0,\delta}\|_{L^{2}(0,T;L^{2}(\mathbb{T}^{3}))}
					\nonumber\\&\lesssim \alpha^{-1}                  \label{uniform phi-alpha}
				\end{align}  which and \eqref{ConEnL2} imply
				\begin{align}
					\label{lemh55-7}
					\nonumber  |J_{1}|&
					\lesssim \|\mathbf{u}_{\delta}\|_{L^{\infty}(0,T;L^{2}(\mathbb{T}^{3}))}\|\nabla\Delta^{-1}\Lambda^{\gamma}\partial_t (\psi\phi_{\delta})\|_{L^{1}(0,T;L^{2}(\mathbb{T}^{3}))}
					\\& \nonumber\lesssim (\|\psi\|_{L^{\infty}(0,T)}\|\partial_t \phi_{\delta}\|_{L^{2}(0,T;L^{2}(\mathbb{T}^{3}))}+\|\partial_t\psi\|_{L^{1}(0,T)}\|\phi_{\delta}\|_{L^{\infty}(0,T;L^{2}(\mathbb{T}^{3}))})
					\\&\lesssim (\frac{1}{\alpha}\|\psi\|_{L^{\infty}(0,T)}+\|\partial_t\psi\|_{L^{1}(0,T)})
				\end{align}
				for $\gamma\leq 1$.
				Similarly, using the fact that $H^{s-\gamma+1}\hookrightarrow L^{\infty}$ for $\gamma\leq 1<s-\frac{1}{2}$, we derive
				\begin{align}
					\label{lemh55-8}
					\nonumber  |J_{2}|
					& \lesssim \|\mathbf{u}_{\delta}\|_{L^{\infty}(0,T;L^{2}(\mathbb{T}^{3}))}\|\nabla\mathbf{u}_{\delta}\|_{L^{2}(0,T;L^{2}(\mathbb{T}^{3}))}\|\nabla\Delta^{-1}\Lambda^{\gamma} (\psi\phi_{\delta})\|_{L^{2}(0,T;L^{\infty}(\mathbb{T}^{3}))}
					\\&\nonumber\lesssim \|\mathbf{u}_{\delta}\|_{L^{\infty}(0,T;L^{2}(\mathbb{T}^{3}))}\|\psi\|_{L^{2}(0,T)}\|\mathbf{u}_{\delta}\|_{L^{2}(0,T;H^{1}(\mathbb{T}^{3}))}\|\nabla\Delta^{-1}\Lambda^{\gamma} (\psi\phi_{\delta})\|_{L^{2}(0,T;H^{s-\gamma+1}(\mathbb{T}^{3}))}
					\\&\nonumber\lesssim \|\mathbf{u}_{\delta}\|_{L^{\infty}(0,T;L^{2}(\mathbb{T}^{3}))}\|\psi\|_{L^{\infty}(0,T)}\|\mathbf{u}_{\delta}\|_{L^{2}(0,T;H^{1}(\mathbb{T}^{3}))}\| \phi_{\delta}\|_{L^{\infty}(0,T;H^{s}(\mathbb{T}^{3}))}
					\\&\lesssim \|\psi\|_{L^{\infty}(0,T)}.
				\end{align}

				For $J_{3}$ and $J_{4}$,  it follows from $H^{s}(\mathbb{T}^{3})\hookrightarrow W^{1,3}(\mathbb{T}^{3})(s>\frac{3}{2})$, $H^{1}(\mathbb{T}^{3})\hookrightarrow L^{6}(\mathbb{T}^{3})$ and H\"{o}lder inequality  that
				\begin{align}
					\label{lemh55-9}
					\nonumber |J_{3}|+|J_{4}|& \lesssim |(\nabla\rho_{\delta}^{-1}S(\phi_{\delta},\mathbb{ D}\mathbf{u}_{\delta}),\nabla\Delta^{-1}\Lambda^{\gamma}(\psi\phi_{\delta}))_{Q_{T}}|+|(\rho_{\delta}^{-1}S(\phi_{\delta},\mathbb{ D}\mathbf{u}_{\delta}),\nabla^{2}\Delta^{-1}\Lambda^{\gamma}(\psi\phi_{\delta}))_{Q_{T}}|
					\\&\nonumber\lesssim \|\psi\|_{L^{2}(0,T)}\|\nabla\phi_{\delta}\|_{L^{\infty}(0,T;L^{3}(\mathbb{T}^{3}))}\|\mathbb{ D}\mathbf{u}_{\delta}\|_{L^{2}(0,T;L^{2}(\mathbb{T}^{3}))}\|\nabla\Delta^{-1}\Lambda^{\gamma} \phi_{\delta}\|_{L^{\infty}(0,T;L^{6}(\mathbb{T}^{3}))}
					\\&\nonumber \quad+\|\psi\|_{L^{2}(0,T)}\|\mathbb{ D}\mathbf{u}_{\delta}\|_{L^{2}(0,T;L^{2}(\mathbb{T}^{3}))}\|\nabla^{2}\Delta^{-1}\Lambda^{\gamma} \phi_{\delta}\|_{L^{\infty}(0,T;L^{2}(\mathbb{T}^{3}))}
					\\&\nonumber\lesssim \|\psi\|_{L^{\infty}(0,T)}\|\phi_{\delta}\|_{L^{\infty}(0,T;H^{s}(\mathbb{T}^{3}))}\|\mathbf{u}_{\delta}\|_{L^{2}(0,T;H^{1}(\mathbb{T}^{3}))}\| \phi_{\delta}\|_{L^{\infty}(0,T;H^{\gamma}(\mathbb{T}^{3}))}
					\\&\nonumber \quad+\|\psi\|_{L^{\infty}(0,T)}\|\mathbf{u}_{\delta}\|_{L^{2}(0,T;H^{1}(\mathbb{T}^{3}))}\|\phi_{\delta}\|_{L^{\infty}(0,T;H^{\gamma}(\mathbb{T}^{3}))}
					\\&\lesssim\|\psi\|_{L^{\infty}(0,T)}.
				\end{align}
				
				By using the similar argument we can easily deduce
				\begin{align}
					\label{lemh55-10}
					|J_{5}|\lesssim \|\psi\|_{L^{\infty}(0,T)},\ |J_{6}|\lesssim \|\psi\|_{L^{\infty}(0,T)},\ |J_{7}|\lesssim \frac{1}{\alpha}\|\psi\|_{L^{\infty}(0,T)}.
				\end{align}

				Combining \eqref{lemh55-7}--\eqref{lemh55-10} with \eqref{lemh55-5}, we show
				\begin{align}
					\label{lemh55-121}
					&
					(\Lambda^{s+\frac{\gamma}{2}}\phi_\delta,\psi\Lambda^{s+\frac{\gamma}{2}}\phi_\delta)_{Q_{T}}
					\lesssim \alpha\|\partial_t\psi\|_{L^{1}(0,T)}+\|\psi\|_{L^{\infty}(0,T)}.
				\end{align}
				Choosing  $\psi=\psi_{m}$ where
				\begin{equation}
					\label{cut-off function}
					\begin{aligned}
						& \psi_{m}\in C_{0}^{\infty}(0,T),  0\leq\psi_{m}\leq1,
						\\& \psi_{m}(t)=1, t\in\left[\frac{1}{m},T-\frac{1}{m}\right], |\psi'|\leq 2m,\ m\in \mathbb{N},
					\end{aligned}
				\end{equation}
				and then sending  $ m \to \infty $,  one derives $\|\phi_{\delta}\|_{L^{2}(0,T;H^{s+\frac{\gamma}{2}}(\mathbb{T}^{3}))}\leq C$. The proof is completed.
			\end{proof}

			As a corollary, we have
			\begin{corollary}\label{lem5.7}
				There is a subsequence, still denoted by $(\phi_{\delta})_{0<\delta\leq1}$ such that $\phi_{\delta}\rightarrow\phi$ in $L^{2}(0,T;H^{s}(\mathbb{T}^{3}))$, as $\delta\rightarrow0$.
			\end{corollary}
			\begin{proof}
				
				Owing to the facts that $\phi_{\delta}$ is uniform bounded in $L^{2}(0,T;H^{s+\frac{\gamma}{2}}(\mathbb{T}^{3}))$,
				$H^{s+\frac{\gamma}{2}}(\mathbb{T}^{3})\hookrightarrow\hookrightarrow H^{s}(\mathbb{T}^{3})$(cf.P330 in \cite{TM}) and $\partial_{t}\phi_{\delta}$ is uniform bounded in $L^{2}(Q_{T})$(cf.\eqref{uniform phi-alpha}), there is a subsequence, still denoted by $(\phi_{\delta})$, such that
				$\phi_{\delta}\rightarrow\phi$  in $L^{2}(0,T;H^{s}(\mathbb{T}^{3}))$ as $\delta\rightarrow 0$ by Aubin-Lions lemma.
			\end{proof}

\begin{lemma}\label{lem5.2}
	There is a constant $C$ independent of $\delta > 0$ such that
	\begin{align*}
		\|\overline{\mu}_{\delta}\|_{L^{\infty}(0,T)}\leq C.
	\end{align*}
\end{lemma}
\begin{proof}
	By choosing $\psi=\eta(t)\in C_{0}^{\infty}(0,T)$ in \eqref{weak4}, we get
	\begin{align*}
		\overline{\mu}_{\delta}(t)=\frac{1}{|\mathbb{T}^{3}|}\int_{\mathbb{T}^{3}}\big(-\mu_{p,\delta}^{0}+\alpha p_{0,\delta}+F'(\phi_{\delta})\big)dx=\frac{1}{|\mathbb{T}^{3}|}\int_{\mathbb{T}^{3}}F'(\phi_{\delta})dx,\ \ a .e. \ \text{in}\ [0,T].
	\end{align*}
	Since  $\|\phi_{\delta}\|_{L^{\infty}(Q_T)}\leq C$, we infer that
	\begin{align*}
		\|\overline{\mu}_{\delta}\|_{L^{\infty}(0,T)}\leq C,
	\end{align*}
	where $C$ does not depend on $\delta$.
\end{proof}

								In the final part of this section we aim to show the proof of  Theorem \ref{thm:main}.
								
								{\bf Proof of Theorem \ref{thm:main}:}
								It follows from  \eqref{model4-5}, \eqref{ConEnL2} and Corollary \ref{lem5.7} that there exists a subsequence $(\delta_n)$ such that $\phi_{\delta_n}\rightarrow\phi$, $\rho_{\delta_n}\rightarrow\rho$ in $L^{2}(0,T;H^{s}(\mathbb{T}^{3}))$ and  $\mathbf{u}_{\delta_n}\rightharpoonup \mathbf{u}$ in $L^{2}(0,T;H^{1}(\mathbb{T}^{3}))$, we conclude that
								$\rho_{\delta_n}\mathbf{u}_{\delta_n}\rightharpoonup \rho\mathbf{u}$ in $L^{1}(0,T;L^{2}(\mathbb{T}^{3}))$ as $n\rightarrow\infty$. Combined with the facts that $\delta_n p_{0,\delta_n}\rightharpoonup 0$ in $L^{2}(Q_{T})$, $\mu_{p,\delta_n}^{0}\rightharpoonup \mu_{p}^{0}$ in $L^{2}(0,T;H^{1}(\mathbb{T}^{3}))$, one can easily get
								\eqref{mweak2}--\eqref{mweak3}  by passing to the limit $n\rightarrow\infty$ in \eqref{weak2}--\eqref{weak3}. Moreover, it is noted that for a.e.~$t \in [0,T]$,  $\int_{\mathbb{T}^{3}}\phi_{\delta}\, \mathrm{d} x=\int_{\mathbb{T}^{3}}\phi_0\, \mathrm{d} x$, then we can derive $\int_{\mathbb{T}^{3}}\phi\, \mathrm{d} x=\int_{\mathbb{T}^{3}}\phi_0\, \mathrm{d} x$ for a.e.~$t \in [0,T]$.

							Submitting \eqref{new pressure} into
							\eqref{weak4} implies
							\begin{align}
								(\zeta\Lambda^{s}\phi_{\delta},\Lambda^{s}\psi)_{Q_{T}}&=(\alpha p_{1,\delta}-\zeta\mu_{p,\delta}^{0}-\zeta\overline{\mu}_{\delta}(t)+\zeta F'(\phi_{\delta}),\psi)_{Q_{T}}
								\nonumber\\&\quad+(\alpha\mathbb{G}(\mathbf{u}_{\delta}),\partial_{t}\psi)_{Q_{T}}.\label{new initial convergence}
							\end{align}
							Thanks  to \eqref{ConEnL2}, Lemma \ref{lem5.1}, Lemma \ref{lem5.3}, Corollary \ref{lem5.7} and Lemma \ref{lem5.2}, there is a subsequence $(\delta_n)$ such that $p_{1,\delta_n}\rightharpoonup p_{1}$ in $L^{2}(0,T;L^{r}(\mathbb{T}^{3}))$($1<r<\frac{3}{2}$),  $F'(\phi_{\delta_n})\rightarrow F'(\phi)$  in $L^{2}(0,T;L^{\infty}(\mathbb{T}^{3}))$, $(\overline{\mu}_{\delta_n},\psi)_{Q_{T}}\rightarrow(\overline{\mu},\psi)_{Q_{T}}$, $\mu_{p,\delta_n}^{0}\rightharpoonup \mu_{p}^{0}$ in $L^{2}(0,T;H^{1}(\mathbb{T}^{3}))$ and $\mathbb{G}(\mathbf{u}_{\delta_n})\rightarrow\mathbb{G}(\mathbf{u})$ in $L^{2}(0,T;H^{1}(\mathbb{T}^{3}))$ as $n\rightarrow \infty$,
							we can show \eqref{mweak4} by passing to the limit $n\rightarrow \infty$  in \eqref{new initial convergence}.
							
							By Lemma \ref{lem5.1}, Lemma \ref{lem5.3}, Lemma \ref{lem5.5} and Lemma \ref{lem5.7}, \eqref{mweak1} and \eqref{eqs:energy} can be  easily verified by passing to the limit $n\rightarrow \infty$  in \eqref{newweak1} and \eqref{ConEnL2} respectively. Taking the advantage of \eqref{model4-6} 
							and \eqref{ConEnL2}, it follows that $\mu_{p}^{0}\in L^{2}(0,T;H^{2})$. 
							
							With the help of \eqref{eqs:dt-u-delta} one has $\partial_{t}\mathbf{u}\in L^{2}(0,T;H^{-1-s}(\mathbb{T}^{3}))$, then $\mathbf{u}\in BC_{\omega}(0,T; L^{2}(\mathbb{T}^{3}))$ holds  from Lemma \ref{lem1} in which $X=H^{-s-1}(\mathbb{T}^{3})$ and $Y=L^2(\mathbb{T}^{3})$. Similarly, $\phi\in BC_{\omega}(0,T; H^{s}(\mathbb{T}^{3}))$ holds  by using Lemma \ref{lem1} in which $X=L^2(\mathbb{T}^{3})$ and $Y=H^s(\mathbb{T}^{3})$.
							
							For any  $\varphi\in C_{0}^{\infty}(\mathbb{T}^{3})$, we write
							\begin{align}
								\int_{\mathbb{T}^{3}}(\mathbf{u}(0)-\mathbf{u}_{0})\varphi\,\dx&=\int_{\mathbb{T}^{3}}(\mathbf{u}(0)-\mathbf{u}(t,x))\varphi\,\dx\nonumber
								\\&\quad+\int_{\mathbb{T}^{3}}(\mathbf{u}(t,x)-\mathbf{u}_{\delta}(t,x))\varphi\,\dx\nonumber
								\\&\quad+\int_{\mathbb{T}^{3}}(\mathbf{u}_{\delta}(t,x)-\mathbf{u}_{0})\varphi\,\dx.
								\label{weak continuity0}
							\end{align}
							Due to $\mathbf{u}\in BC_{w}(0,T;L^{2}(\mathbb{T}^{3}))$, $\mathbf{u}_{\delta}\in BC_{w}(0,T;L^{2}(\mathbb{T}^{3}))$, \eqref{appinitial data-1} and  Lemma \ref{lem5.3} one has
							\begin{align*}
								&\int_{\mathbb{T}^{3}}(\mathbf{u}(0)-\mathbf{u}_{0})\varphi\,\dx=0,\ \forall \varphi\in C_{0}^{\infty}(\mathbb{T}^{3}),
							\end{align*}
							which leads to \eqref{initial continuty1}. Similarly, one can get \eqref{initial continuty2}.
							
							Therefore we complete the proof of Theorem \ref{thm:main}.

							\section{Incompressibility Limit}
							\label{sec:limit}
							\subsection{Formal argument and main result}
							
							In this section, we investigate the limit passage of the system of \eqref{model2} as $\alpha\rightarrow0$. Let $(\mathbf{u}_{\alpha},\phi_{\alpha}, \mu_{p,\alpha},p_{\alpha})$ be the solution of \eqref{model2}. For convenience, we rewrite the system in $Q_{T}$ as follows
							\begin{subequations}
								\label{mode54}
								\begin{align}
									\label{mode54-1}
									\partial_t(\rho_{\alpha}  \mathbf{u}_{\alpha}) + \mathrm{div}\left (\rho_{\alpha}  \mathbf{u}_{\alpha}\otimes
									\mathbf{u}_{\alpha}\right) 
									-\mathrm{div} \big(S(\phi_{\alpha},\mathbb{ D}\mathbf{u}_{\alpha})\big)
									& = - \nabla p_{\alpha}
									- \phi_{\alpha} \nabla \mu_{p,\alpha}
									+ \alpha\phi_{\alpha}\nabla p_{\alpha}, 
									\\
									\label{mode54-6}
									\mathrm{div}  \mathbf{u}_{\alpha} &= \alpha \Delta\mu_{p,\alpha}, \\
									\label{mode54-3}
									\partial_t\phi_{\alpha} + \mathrm{div}(\phi_{\alpha}  \mathbf{u}_{\alpha})&=\Delta\mu_{p,\alpha}, \\
									\label{mode54-4}
									\mu_{\alpha} &=F'(\phi_{\alpha})+\Lambda^{2s} \phi_{\alpha}, \\
									\label{mode54-5}
									\mu_{p,\alpha}&=\mu_{\alpha}+\alpha p_{\alpha},\\
									\rho_{\alpha}&=\frac{\varepsilon}{2}\phi_{\alpha}+\frac{\varepsilon}{2}+1,
								\end{align}
								and initial data
								\begin{align}
									\label{model54-7}\mathbf{u}_{\alpha}|_{t=0}=\mathbf{u}^{\alpha}_{0},\ \phi_{\alpha}|_{t=0}=\phi^{\alpha}_{0}, \quad\text{in $\bbt^3$}.
								\end{align}
							\end{subequations}
							Moreover, the continuity equation in $Q_{T}$ reads
							\begin{align}
								\label{mode54-2}
								\partial_t \rho_{\alpha}+\mathrm{div} (\rho_{\alpha} \mathbf{u}_{\alpha})=0.
							\end{align}

							The energy inequality now reads
							\begin{align}\label{que}
								&\nonumber E(\phi_{\alpha}(t),\mathbf{u}_{\alpha}(t))+\int_0^{t}\int_{\mathbb{T}^{3}} \left( 2 \eta(\phi_{\alpha})
								\mathbb{ D}(\mathbf{u}_{\alpha}):\mathbb{ D}(\mathbf{u}_{\alpha})-\frac{2}{3}\eta(\phi_{\alpha}) (\mathrm{div} \mathbf{u}_{\alpha})^2\right)\, \dx\dtau
								\\&\qquad
								+\int_0^{t}\int_{\mathbb{T}^{3}}|\nabla \mu_{p,\alpha}^{0}|^{2}\, \dx\dtau
								\leq E(\phi_{0}^{\alpha},\mathbf{u}_{0}^{\alpha})
							\end{align}
							for a.e.~$t \in (0,T)$.

							Sending $\alpha\rightarrow0$, one formally gets $\mathrm{div} \mathbf{u}_{\alpha}|_{\alpha \to 0}=0$ and a modified version of the celebrated \textit{model H} \cite{HH}, with the solution denoted by $(\mathbf{u},\phi, \mu_{p},p)$ satisfying
							\begin{subequations}
								\label{eqs:ModelH}
								\begin{align}
									\partial_t  \mathbf{u} +\mathbf{u}\cdot\nabla \mathbf{u}-\mathrm{div}(2\eta(\phi)\mathbb{ D}\mathbf{u}) &=-\nabla p-\phi \nabla \mu, \\
									\mathrm{div} \mathbf{u}&= 0  \label{imcompressible},\\
									\partial_t\phi + \mathbf{u}\cdot\nabla \phi &=\Delta\mu \label{no density},\\
									\mu& =F'(\phi)+\Lambda^{2s} \phi.
								\end{align}
							\end{subequations}
							
							The system \eqref{eqs:ModelH} with $s=1$ has been investigated extensively over the past decades. It is known that the strong solution of 3D Navier--Stokes/Cahn--Hilliard equations exists only locally in time \cite{AbelsARMA2009,GMT}. Here we assume the life span of the strong solution is $(0,T')$ for some $T' > 0$. Moreover, for simplicity, we assume
							\begin{align}
								\label{eqs:reg-assump}
								\mathbf{u}\in H^{1}(0,T';H^{s+\frac{1}{2}}(\mathbb{T}^{3})), p\in L^{2}(0,T';H^{1}(\mathbb{T}^{3})), \mu\in H^{1}(Q_{T'}),  \phi\in H^2(0,T';H^{s}(\mathbb{T}^{3})).
							\end{align}
							Note that here we are not devoted to finding the optimal regularity assumption of the strong solution. Instead, we would like to perform the relative entropy argument with some suitably nice strong solution.
							
							\begin{theorem}
								\label{thm:limit}
								Let $(\mathbf{u}_{\alpha},\phi_{\alpha}, \mu_{p,\alpha},p_{\alpha})$ be a weak solution to \eqref{mode54}, i.e. \eqref{model2}, constructed in Theorem \ref{thm:main}. Assume that the initial data $(\bu_0^\alpha,\phi_0^\alpha) \in H^1(\bbt^3) \times H^{s}(\bbt^3)$ is well-prepared, i.e., there hold
								$\int_{\bbt^3} \phi_0^\alpha \,\dx = \int_{\bbt^3} \phi_0 \,\dx = const.$ and as $\alpha\rightarrow 0$ 
								\begin{align}
									\bu_0^\alpha & \to \bv_0 \text{ in } L^2(\bbt^3),\label{wp-1} \\
									\phi_0^\alpha & \to \psi_0 \text{ in } H^s(\bbt^3).\label{wp-2}
								\end{align}
								Let $(\mathbf{u},\phi, \mu,p)$ be the strong solution of \eqref{eqs:ModelH} defined on $[0,T']$ subjected to the initial data $(\mathbf{v}_{0},\phi^{0})$ fulfilling the regularity \eqref{eqs:reg-assump}. Then it follows
								\begin{align}
									&\nonumber \sup_{0\leq \tau \leq T'}E(\phi_{\alpha},\mathbf{u}_{\alpha}|\phi,\mathbf{u})(\tau)
									+  \frac{ C_1}{2} \| \mathbf{u}_{\alpha}-\mathbf{u}\|_{L^2(0,\tau;H^{1}(\mathbb{T}^{3}))}^{2}
									\\&
									\nonumber\qquad +\frac{1}{2}\int_0^{T'}\int_{\mathbb{T}^{3}}|\nabla \mu_{p,\alpha}-\nabla \mu|^{2}\,\dx\dt
									\\&\quad
									\leq C(T',D)\bigg(\alpha+ E(\phi_{\alpha},\mathbf{u}_{\alpha}|\phi,\mathbf{u})(0)\bigg),
									\label{th5.1}
								\end{align}
								where $C(T',D)$ is a positive constant depending on $T'$ and $D$ which is defined by
								\begin{align}
									D&\coloneqq E(\phi_{0}^{\alpha},\mathbf{u}_{0}^{\alpha})+\| \mathbf{u}\|_{H^{1}(0,T';H^{s+\frac{1}{2}}(\mathbb{T}^{3}))}+\| p\|_{L^{2}(0,T';H^{1}(\mathbb{T}^{3}))}\nonumber
									\\&\quad+\|\mu\|_{H^{1}(Q_{T'})}+\|\phi\|_{H^2(0,T';H^{s}(\mathbb{T}^{3}))}.\label{def D}
								\end{align}
								Moreover, the relative energy is defined by
								\begin{align}
									\nonumber E(\phi_{\alpha},\mathbf{u}_{\alpha}|\phi,\mathbf{u})(\tau)&\coloneqq\int_{\mathbb{T}^{3}}\frac{1}{2}\rho_{\alpha}|\mathbf{u}-\mathbf{u}_{\alpha}|^{2}(\tau)\,\dx
									+\int_{\mathbb{T}^{3}}\frac{1}{2}|\Lambda^{s}\phi-\Lambda^{s}\phi_{\alpha}|^{2}(\tau)\,\dx
									\\&\label{54-17}
									\quad +\int_{\mathbb{T}^{3}}(\Phi(\phi_{\alpha})-\Phi'(\phi)(\phi_{\alpha}-\phi)-\Phi(\phi))(\tau)\,\dx,
								\end{align}
								where $\Phi$ is the convex part of $F$ by Assumption \ref{ass:main}.
							\end{theorem}
							\begin{remark}
								As an application of the relative entropy method, a significant advantage is that the relative energy \eqref{54-17} provides the control of the certain distance of the weak and strong solutions by the coercive property of the functional we constructed. In particular, we have
								\begin{align*}
									E(\phi_{\alpha},\mathbf{u}_{\alpha}|\phi,\mathbf{u}) \gtrsim \norm{\bu-\bu_\alpha}_{L^2(\bbt^3)}^2
									+ \norm{\phi-\phi_\alpha}_{H^s(\bbt^3)}^2,
								\end{align*}
								due to $\Phi'' \geq 0$.
								As a consequence, it implies the convergence with the rate $\alpha$  when the initial data is well-prepared.
							\end{remark}

							\subsection{Further uniform-in-\texorpdfstring{$\alpha$}{α} controls}
							The following two lemmas are essential when passing to the limit of $ \alpha \to 0 $, as we do not have any uniform bounds of the pressure individually.
							\begin{lemma}\label{lem:limit 1}
								For $s>\frac{3}{2}$ and any  $f\in H^{1}(0,T';H^{s}(\mathbb{T}^{3}))$, there is a constant $C=C(f) > 0$  independent of  $\alpha$ such that
								\begin{align}
									\int_0^{T'}\int_{\mathbb{T}^{3}}p_{\alpha}f\,\dx\dt\leq C.\label{new pp est}
								\end{align}
							\end{lemma}
							\begin{proof}
								Given a  weak solution $ (\bu_{\alpha},\phi_{\alpha},\mu_{p,\alpha},p_{\alpha})$ and utilizing \eqref{mode54-2} we have
								\begin{align}
									\label{lem55-1}
									& \partial_t  \mathbf{u}_{\alpha} +   \mathbf{u}_{\alpha}\cdot\nabla
									\mathbf{u}_{\alpha}-  \frac{1}{\rho_{\alpha}}\mathrm{div}\big(S(\phi_{\alpha},\mathbb{ D}\mathbf{u}_{\alpha})\big)
									+\zeta\nabla p_{\alpha}+\frac{\phi_{\alpha}}{\rho_{\alpha}} \nabla \mu_{p,\alpha}=0, \quad\text{in $Q_{T}$}.
								\end{align}
								Analogous to the proof of Lemma \ref{lem:higher eatimate}, we  take $\psi\nabla\Delta^{-1}f$ with $\psi\in C_{0}^{\infty}(0,T')$ as the test function in \eqref{lem55-1} and integrate by parts, then  get
							\begin{align}
								\label{lem55-3}
								\nonumber\zeta (p_{\alpha},\psi f)_{Q_{T'}}& =-(\mathbf{u}_{\alpha},\psi'\nabla\Delta^{-1}f)_{Q_{T'}}-(  \mathbf{u}_{\alpha},\psi\partial_t\nabla\Delta^{-1}f)_{Q_{T'}}+   (\mathbf{u}_{\alpha}\cdot\nabla
								\mathbf{u}_{\alpha},\psi\nabla\Delta^{-1}f)_{Q_{T'}}
								\\&\nonumber\quad+ (\nabla\rho_{\alpha}^{-1}\big(S(\phi_{\alpha},\mathbb{ D}\mathbf{u}_{\alpha})\big),\psi\nabla\Delta^{-1}f)_{Q_{T'}}+  (\rho_{\alpha}^{-1}\big(S(\phi_{\alpha},\mathbb{ D}\mathbf{u}_{\alpha})\big),\psi\nabla^{2}\Delta^{-1}f)_{Q_{T'}}
								\\&\nonumber\quad+(\phi_{\alpha}\rho_{\alpha}^{-1} \nabla \mu_{p,\alpha},\psi\nabla\Delta^{-1}f)_{Q_{T'}}
								\\&\eqqcolon \sum_{i=1}^{6}I_{i}.
							\end{align}
						
						In what following we will estimate $I_{i}(i=1,...,6)$ one by one with the help of \eqref{eqs:energy}.
						
						By H\"{o}lder inequality, there holds
						\begin{align}
							\label{lem55-5}
							\nonumber |I_{1}|& \leq\|\partial_{t}\psi\|_{L^{1}(0,T')}\|\mathbf{u}_{\alpha}\|_{L^{\infty}(0,T';L^{2}(\mathbb{T}^{3}))}\|\nabla\Delta^{-1}f\|_{L^{\infty}(0,T';L^{2}(\mathbb{T}^{3}))}
							\\& \lesssim  \|\partial_{t}\psi\|_{L^{1}(0,T')}\|\mathbf{u}_{\alpha}\|_{L^{\infty}(0,T';L^{2}(\mathbb{T}^{3}))}\|f\|_{L^{\infty}(0,T';L^{2}(\mathbb{T}^{3}))}
							\leq C\|\partial_{t}\psi\|_{L^{1}(0,T')}.
						\end{align}
						Similarly, one has
						\begin{align}
							\label{lem55-6}
							|I_{2}|\lesssim \|\psi\|_{L^{\infty}(0,T')}\|\mathbf{u}_{\alpha}\|_{L^{\infty}(0,T';L^{2}(\mathbb{T}^{3}))}\|\partial_{t}f\|_{L^{\infty}(0,T';L^{2}(\mathbb{T}^{3}))}
							\leq C\|\psi\|_{L^{\infty}(0,T')}
						\end{align}
						and
						\begin{align}
							\label{lem55-7}
							|I_{3}| &\lesssim \|\psi\|_{L^{\infty}(0,T')}\|\mathbf{u}_{\alpha}\|_{L^{\infty}(0,T';L^{2}(\mathbb{T}^{3}))}\|
							\mathbf{u}_{\alpha}\|_{L^{2}(0,T';H^{1}(\mathbb{T}^{3}))}\|f\|_{L^{2}(0,T';L^{\infty}(\mathbb{T}^{3}))}
							\nonumber\\&\leq C\|\psi\|_{L^{\infty}(0,T')}.
						\end{align}

						It follows from $H^{s}(\mathbb{T}^{3})\hookrightarrow W^{1,3}(\mathbb{T}^{3})$  for $s>\frac{3}{2}$ that
						\begin{align}
							\label{lem55-8}
							\nonumber  |I_{4}|+|I_{5}|& \lesssim\|\nabla\rho_{\alpha}^{-1}\|_{L^{\infty}(0,T';L^{3}(\mathbb{T}^{3}))}\|\mathbb{ D}\mathbf{u}_{\alpha}\|_{L^{2}(0,T';L^{2}(\mathbb{T}^{3}))}\|\nabla\Delta^{-1}f\|_{L^{2}(0,T';L^{6}(\mathbb{T}^{3}))}\|\psi\|_{L^{\infty}(0,T')}
							\\&\nonumber \quad +\|\mathbb{ D}\mathbf{u}_{\alpha}\|_{L^{2}(0,T';L^{2}(\mathbb{T}^{3}))}\|\nabla^{2}\Delta^{-1}f\|_{L^{2}(0,T';L^{2}(\mathbb{T}^{3}))}\|\psi\|_{L^{\infty}(0,T')}
							\\&\lesssim \|\psi\|_{L^{\infty}(0,T')}.
						\end{align}
						Finally, we have
						\begin{align}\label{lem55-10}
							\nonumber|I_{6}| &\lesssim \|\phi_{\alpha}\rho_{\alpha}^{-1} \|_{L^{\infty}(0,T';L^{\infty}(\mathbb{T}^{3}))}\|\nabla \mu_{p,\alpha}\|_{L^{2}(0,T';L^{2}(\mathbb{T}^{3}))}\|\nabla\Delta^{-1}f\|_{L^{2}(0,T';L^{2}(\mathbb{T}^{3}))}\|\psi\|_{L^{\infty}(0,T')}
							\\&\lesssim \|\psi\|_{L^{\infty}(0,T')}.
						\end{align}
						
						Combing \eqref{lem55-5}-\eqref{lem55-10} with \eqref{lem55-3} and noting that $\zeta=1+\alpha$, we arrive at
						\begin{align}
							\int_0^{T'}\int_{\mathbb{T}^{3}}p_{\alpha}\psi f\,\dx\dt \leq C\big(\|\psi'\|_{L^{1}(0,T')}+\|\psi\|_{L^{\infty}(0,T')}\big).\label{new pp est-0}
						\end{align}
						Choosing  $\psi=\psi_{m}$ in \eqref{new pp est-0} where
						\begin{equation*}
							\begin{aligned}
								& \psi_{m}\in C_{0}^{\infty}(0,T'),  0\leq\psi_{m}\leq1,
								\\& \psi_{m}(t)=1, t\in\left[\frac{1}{m},T'-\frac{1}{m}\right], |\psi'|\leq 2m,\ m\in \mathbb{N}.
							\end{aligned}
						\end{equation*}
						By \eqref{mode54-4}-\eqref{mode54-5} and Lemma \ref{lem:higher eatimate}, we conclude
						\begin{align}
							\label{eqs:alpha-p-psi-f}
							&\nonumber\lim_{m\to\infty}\int_0^{T'}\int_{\mathbb{T}^{3}}\alpha \psi_{m} p_{\alpha}f\,\dx\dt
							\\&\nonumber=\lim_{m\to\infty}\bigg(\int_0^{T'}\int_{\mathbb{T}^{3}}\psi_{m}(\mu_{p,\alpha}-F'(\phi_{\alpha}))f\,\dx\dt
							-\int_0^{T'}\int_{\mathbb{T}^{3}}\psi_{m}\Lambda^{s}\phi_{\alpha}\Lambda^{s}f\,\dx\dt\bigg)
							\\&\nonumber=\bigg(\int_0^{T'}\int_{\mathbb{T}^{3}}(\mu_{p,\alpha}-F'(\phi_{\alpha}))f\,\dx\dt
							-\int_0^{T'}\int_{\mathbb{T}^{3}}\Lambda^{s}\phi_{\alpha}\Lambda^{s}f\,\dx\dt\bigg)
							\\&=\int_0^{T'}\int_{\mathbb{T}^{3}}\alpha p_{\alpha}f\,\dx\dt,
						\end{align}
						where we have used  $\big(\mu_{p,\alpha}-F'(\phi_{\alpha})\big)f\in L^{1}(Q_{T'})$, $(\Lambda^{s}\phi_{\alpha})\Lambda^{s}f\in L^{1}(Q_{T'})$ and  Lebesgue's dominated convergence theorem.
						
						By sending $m\to\infty$ in \eqref{new pp est-0}($\psi=\psi_{m}$), one can obtain the desired inequality \eqref{new pp est}.
					\end{proof}

					\begin{lemma}\label{lem:limit 2}
						For any $s>\frac{3}{2}$ and $\bg\in H^{1}(0,T';H^{s}(\mathbb{T}^{3}))$ with $\mathrm{div} \bg=0$,  there is a constant $C=C(\bg)>0$  independent of  $\alpha$ such that
						\begin{align}
							\int_0^{T'}\int_{\mathbb{T}^{3}}p_{\alpha}\nabla\phi_{\alpha}\cdot \bg\,\dx\dt\label{new pp est-00}
							\leq C.
						\end{align}
					\end{lemma}
					\begin{proof}
						Multiplying \eqref{lem55-1} by $\psi\varpi\bg$  and integrating over $Q_{T'}$, it follows that
						\begin{align}
							\nonumber\zeta(p_{\alpha},\psi\nabla\phi_{\alpha} \cdot \bg)_{Q_{T'}}
							&=(\partial_t  \mathbf{u}_{\alpha}, \psi\varpi\bg)_{Q_{T'}} +   (\mathbf{u}_{\alpha}\cdot\nabla
							\mathbf{u}_{\alpha},\psi\varpi\bg)_{Q_{T'}}\nonumber\\&\quad
							+  (\nabla\rho_{\alpha}^{-1}S(\phi_{\alpha},\mathbb{ D}\mathbf{u}_{\alpha}), \psi\varpi \bg)_{Q_{T'}}+  (\rho_{\alpha}^{-1}S(\phi_{\alpha},\mathbb{ D}\mathbf{u}_{\alpha}), \psi\nabla(\varpi \bg))_{Q_{T'}}\nonumber
							\\&\quad+(\phi_{\alpha}\rho_{\alpha}^{-1} \nabla \mu_{p,\alpha},\psi\varpi\bg)_{Q_{T'}}\nonumber
							\\&\eqqcolon\sum_{i=1}^{5}K_{i},\label{new pp est-000}
						\end{align}
						where $\varpi=\phi_{\alpha} -\frac{1}{\absm{\bbt^3}}\int_{\mathbb{T}^{3}} \phi_{\alpha}\,dx$.
						

					In order to give the proof of this lemma, it suffice to give the estimation of $K_{i}(1\leq i \leq 5)$.

					According to \eqref{eqs:energy} and \eqref{mode54-3}, $\partial_{t}\phi_{\alpha}$ is  bounded in $L^{2}(0,T';H^{-1}(\mathbb{T}^{3}))$  uniformly in $\alpha$.
					Then by the H\"{o}lder inequality and \eqref{eqs:energy} one has
					\begin{align}\label{K_{1}}
						\nonumber |K_{1}|&\leq|(\mathbf{u}_{\alpha}, \partial_t\psi \varpi \bg)_{Q_{T'}}|+ |(\mathbf{u}_{\alpha}, \psi\partial_t\phi_{\alpha} \bg)_{Q_{T'}}|+|(\mathbf{u}_{\alpha}, \psi\varpi\partial_t\bg)_{Q_{T'}}|
						\\&\nonumber\leq |(\mathbf{u}_{\alpha}, \partial_t\psi \varpi\bg)_{Q_{T'}}|+ |(\nabla\mathbf{u}_{\alpha}, \psi \nabla\Delta^{-1}(\partial_t\phi_{\alpha})\otimes\bg)_{Q_{T'}}|+|\big(\mathbf{u}_{\alpha}, \psi\big(\nabla\Delta^{-1}(\partial_t\phi_{\alpha})\cdot\nabla\big) \bg\big)_{Q_{T'}}|
						\\&\nonumber\quad+|(\mathbf{u}_{\alpha}, \psi\varpi\partial_t\bg)_{Q_{T'}}|
						\\&\nonumber\leq \|\mathbf{u}_{\alpha}\|_{L^{\infty}(0,T';L^{2}(\mathbb{T}^{3}))}\|\partial_t\psi\|_{L^{1}(0,T')}\|\varpi\|_{L^{\infty}(0,T';L^{2}(\mathbb{T}^{3}))}\|\bg\|_{L^{\infty}(0,T';L^{\infty}(\mathbb{T}^{3}))}
						\\&\nonumber\quad+\|\nabla\mathbf{u}_{\alpha}\|_{L^{2}(0,T';L^{2}(\mathbb{T}^{3}))}\|\psi\|_{L^{\infty}(0,T')}\|\nabla\Delta^{-1}(\partial_{t}\phi_{\alpha})\|_{L^{2}(0,T';L^{2}(\mathbb{T}^{3}))}\|\bg\|_{L^{\infty}(0,T';L^{\infty}(\mathbb{T}^{3}))}
						\\&\nonumber\quad+\|\mathbf{u}_{\alpha}\|_{L^{2}(0,T';L^{6}(\mathbb{T}^{3}))}\|\psi\|_{L^{\infty}(0,T')}\|\nabla\Delta^{-1}(\partial_{t}\phi_{\alpha})\|_{L^{2}(0,T';L^{2}(\mathbb{T}^{3}))}\|\nabla \bg\|_{L^{\infty}(0,T';L^{3}(\mathbb{T}^{3}))}
						\\&\nonumber\quad+\|\mathbf{u}_{\alpha}\|_{L^{\infty}(0,T';L^{2}(\mathbb{T}^{3}))}\|\psi\|_{L^{\infty}(0,T')}\|\varpi\|_{L^{2}(0,T';L^{2}(\mathbb{T}^{3}))}\|\partial_{t}\bg\|_{L^{2}(0,T';L^{\infty}(\mathbb{T}^{3}))}
						\\&\lesssim \|\partial_t\psi\|_{L^{1}(0,T')}+\|\psi\|_{L^{\infty}(0,T')}.
					\end{align}
					
					Similarly, we can get
					\begin{align}\label{K_{2}}
						\nonumber|K_{2}|&\lesssim \|\mathbf{u}_{\alpha} \|_{L^{\infty}(0,T';L^{2}(\mathbb{T}^{3}))}\|\nabla\mathbf{u}_{\alpha} \|_{L^{2}(0,T';L^{2}(\mathbb{T}^{3}))}\|\varpi\|_{L^{\infty}(0,T';L^{\infty}(\mathbb{T}^{3}))}\|\bg\|_{L^{2}(0,T';L^{\infty}(\mathbb{T}^{3}))}\|\psi\|_{L^{\infty}(0,T')}
						\\&\lesssim \|\psi\|_{L^{\infty}(0,T')}
					\end{align}
					and
					\begin{align}\label{K_{3}}
						\nonumber|K_{3}|+|K_{4}|&\leq|(\nabla(\rho_{\alpha}^{-1})S(\phi_{\alpha},\mathbb{ D}\mathbf{u}_{\alpha}), \psi\varpi\bg)_{Q_{T'}}|+|(\rho_{\alpha}^{-1}S(\phi_{\alpha},\mathbb{ D}\mathbf{u}_{\alpha}), \psi\nabla\phi_{\alpha}\otimes\bg)_{Q_{T'}}|
						\\&\nonumber\quad+|(\rho_{\alpha}^{-1}S(\phi_{\alpha},\mathbb{ D}\mathbf{u}_{\alpha}), \psi\varpi\nabla\bg)_{Q_{T'}}|
						\\&\lesssim\nonumber\|\nabla\phi_{\alpha} \|_{L^{\infty}(0,T';L^{3}(\mathbb{T}^{3}))}\|\mathbb{ D}\mathbf{u}_{\alpha} \|_{L^{2}(0,T';L^{2}(\mathbb{T}^{3}))}\|\varpi\|_{L^{\infty}(0,T';L^{6}(\mathbb{T}^{3}))}\|\bg\|_{L^{2}(0,T';L^{\infty}(\mathbb{T}^{3}))}\|\psi\|_{L^{\infty}(0,T')}
						\\& \nonumber\quad+\|\mathbb{ D}\mathbf{u}_{\alpha}\|_{L^{2}(0,T';L^{2}(\mathbb{T}^{3}))}\|\nabla\phi_{\alpha}\|_{L^{\infty}(0,T';L^{2}(\mathbb{T}^{3}))}\|\bg\|_{L^{2}(0,T';L^{\infty}(\mathbb{T}^{3}))}\|\psi\|_{L^{\infty}(0,T')}
						\\& \nonumber\quad + \|\mathbb{ D}\mathbf{u}_{\alpha}\|_{L^{2}(0,T';L^{2}(\mathbb{T}^{3}))}\|\varpi\|_{L^{\infty}(0,T';L^{\infty}(\mathbb{T}^{3}))}\|\nabla \bg\|_{L^{2}(0,T';L^{2}(\mathbb{T}^{3}))}\|\psi\|_{L^{\infty}(0,T')}
						\\&\lesssim \|\psi\|_{L^{\infty}(0,T')}
					\end{align}
					and
					\begin{align}\label{K_{4}}
						\nonumber|K_{5}| &\lesssim \|\nabla\mu_{p,\alpha} \|_{L^{2}(0,T';L^{2}(\mathbb{T}^{3}))}\|\varpi\|_{L^{\infty}(0,T';L^{\infty}(\mathbb{T}^{3}))}\|\bg\|_{L^{2}(0,T';L^{2}(\mathbb{T}^{3}))}\|\psi\|_{L^{\infty}(0,T')}
						\\& \lesssim \|\psi\|_{L^{\infty}(0,T')}.
					\end{align}

					Submitting \eqref{K_{1}}-\eqref{K_{4}} into \eqref{new pp est-000} implies
					\begin{align}
						\zeta(p_{\alpha},\psi\nabla\phi_{\alpha} \cdot \bg)_{Q_{T'}}\lesssim \|\partial_t\psi\|_{L^{1}(0,T')}+\|\psi\|_{L^{\infty}(0,T')}.\label{new pp est-0000}
					\end{align}
					To proceed, thanks to \eqref{eqs:phi-high} we firstly note that
					\begin{align}
						\nonumber\|\Lambda^{s-\frac{1}{2}}(\nabla \phi_\alpha \cdot \mathbf{g})\|_{L^{2}(0,T';L^{2}(\mathbb{T}^{3}))}&\leq \|\Lambda^{s+\frac{1}{2}} \phi_\alpha\|_{L^{2}(0,T';L^{2}(\mathbb{T}^{3}))}\|\mathbf{g}\|_{L^{2}(0,T';L^{2}(\mathbb{T}^{3}))}
						\\&\nonumber\quad+\|\Lambda^{s-\frac{1}{2}}\mathbf{g}\|_{L^{\infty}(0,T';L^{2}(\mathbb{T}^{3}))}\|\nabla \phi_\alpha\|_{L^{1}(0,T';L^{2}(\mathbb{T}^{3}))}
						\\&\nonumber\leq \|\phi_\alpha\|_{L^{2}(0,T';H^{s+\frac{1}{2}}(\mathbb{T}^{3}))}\|\mathbf{g}\|_{L^{2}(0,T';L^{2}(\mathbb{T}^{3}))}
						\\&+
						\|\phi_\alpha\|_{L^{1}(0,T';H^{1}(\mathbb{T}^{3}))}\|\mathbf{g}\|_{L^{\infty}(0,T';H^{s-\frac{1}{2}}(\mathbb{T}^{3}))}
					\end{align}
					which immediately implies $\Lambda^{s-\frac{1}{2}}(\nabla \phi_\alpha \cdot \mathbf{g})\in L^{2}(0,T';L^{2}(\mathbb{T}^{3}))$ and then  $\Lambda^{s+\frac{1}{2}}\phi_{\alpha}\Lambda^{s-\frac{1}{2}}(\nabla\phi_{\alpha}\cdot\bg)\in L^{1}(Q_{T'})$ for all $\bg\in H^{1}(0,T';H^{s}(\mathbb{T}^{3}))$.

					Choosing  $\psi=\psi_{m}$ in \eqref{new pp est-0000} where
					\begin{equation*}
						\begin{aligned}
							& \psi_{m}\in C_{0}^{\infty}(0,T'),  0\leq\psi_{m}\leq1,
							\\& \psi_{m}(t)=1, t\in\left[\frac{1}{m},T'-\frac{1}{m}\right], |\psi'|\leq 2m,\ m\in \mathbb{N}.
						\end{aligned}
					\end{equation*}
					
					By using  $(\mu_{p,\alpha}-F'(\phi_{\alpha})\nabla\phi_{\alpha}\cdot \bg\in L^{1}(Q_{T'})$, $\Lambda^{s+\frac{1}{2}}\phi_{\alpha}\Lambda^{s-\frac{1}{2}}(\nabla\phi_{\alpha}\cdot\bg)\in L^{1}(Q_{T'})$ and the Lebesgue's dominated convergence theorem,
					one derives
					\begin{align}\label{Leb}
						\nonumber&\lim_{m\to\infty}\int_0^{T'}\int_{\mathbb{T}^{3}} \alpha p_{\alpha}\nabla\phi_{\alpha}\cdot \bg\psi_{m}\,\dx\dt&
						\nonumber\\&=\lim_{m\to\infty}\big(\int_0^{T'}\int_{\mathbb{T}^{3}}\psi_{m}(\mu_{p,\alpha}-F'(\phi_{\alpha})\nabla\phi_{\alpha}\cdot \bg\,\dx\dt
						\nonumber\\&\quad-\int_0^{T'}\int_{\mathbb{T}^{3}}\psi_{m}\Lambda^{s+\frac{\gamma}{2}}\phi_{\alpha}\Lambda^{s-\frac{\gamma}{2}}(\nabla\phi_{\alpha}\cdot\bg)\,\dx\dt \big),
						\nonumber\\&=\int_0^{T'}\int_{\mathbb{T}^{3}}(\mu_{p,\alpha}-F'(\phi_{\alpha})\nabla\phi_{\alpha}\cdot \bg\,\dx\dt,
						-\int_0^{T'}\int_{\mathbb{T}^{3}}\Lambda^{s+\frac{1}{2}}\phi_{\alpha}\Lambda^{s-\frac{1}{2}}(\nabla\phi_{\alpha}\cdot\bg)\,\dx\dt,
						\nonumber\\&=\int_0^{T'}\int_{\mathbb{T}^{3}} \alpha p_{\alpha}\nabla\phi_{\alpha}\cdot \bg\,\dx\dt.
					\end{align}

					Finally one can obtain \eqref{new pp est-00}  and  finish the proof by sending $m\to\infty$ in \eqref{new pp est-0000}($\psi=\psi_{m}$).
				\end{proof}

					For the terms $(\rho_{\alpha}\mathbf{u}_{\alpha},\nabla\bg)_{Q_{T}}$ with some $ \bg $ appearing in the sequel argument, we would employ the continuity equation \eqref{mode54-2} to close the estimates.
					\begin{lemma}\label{lem:limit 3}
						For any  $g\in H^{1}(0,T';H^{1})$, there is a constant $C=C(g) > 0$ independent of $\alpha$ such that
						\begin{align}
							(\rho_{\alpha}\mathbf{u}_{\alpha},\nabla g)_{Q_{T'}}\leq C\alpha.\label{new velocity est-00}
						\end{align}
					\end{lemma}
					\begin{proof}
						Testing \eqref{mode54-2} by $\psi g$ where $\psi\in C_{0}^{\infty}(0,T)$, one obtains
						\begin{align}\label{lem:limit31}
							(\rho_{\alpha}\mathbf{u}_{\alpha},\psi\nabla g)_{Q_{T'}}&=(\partial_{t}(\rho_{\alpha}-1),\psi g)_{Q_{T'}}
							\nonumber\\&=-((\rho_{\alpha}-1),\partial_{t}\psi g)_{Q_{T'}}-((\rho_{\alpha}-1),\psi\partial_{t}g)_{Q_{T'}} \nonumber\\&\lesssim \|(\rho_{\alpha}-1)\|_{L^{\infty}(0,T;L^{\infty}(\mathbb{T}^{3}))}\|\partial_{t}\psi\|_{L^{1}(0,T)}\|\bg\|_{L^{\infty}(0,T;L^{1}(\mathbb{T}^{3}))}
							\nonumber\\&\quad+\|(\rho_{\alpha}-1)\|_{L^{\infty}(0,T;L^{\infty}(\mathbb{T}^{3}))}\|\psi\|_{L^{\infty}(0,T)}\|\partial_{t}\bg\|_{L^{1}(0,T;L^{1}(\mathbb{T}^{3}))}
							\nonumber\\&\lesssim \alpha\big(\|\partial_{t}\psi\|_{L^{1}(0,T)}+\|\psi\|_{L^{\infty}(0,T)}\big),
						\end{align}
						where we have used
						\begin{align*}
							\rho_{\alpha}-1=\frac{\varepsilon}{2}\phi_{\alpha}+\frac{\varepsilon}{2}=-\frac{\alpha}{1+\alpha}(\phi_{\alpha}+1).
						\end{align*}

						By choosing $\psi_m$ defined in Lemma \ref{lem:higher eatimate} (or Lemma \ref{lem:limit 1}, Lemma \ref{lem:limit 2}) and repeating the similar strategy as Lemma \ref{lem:higher eatimate}, we can conclude the proof.
					\end{proof}
					
					\subsection{Proof of Theorem \ref{thm:limit}}
					We begin to give the detailed proof of Theorem \ref{thm:limit} based on  the relative entropy method.
					
					\textbf{\underline{Step 1}: relative energy inequality.} In this part, we aim to derive the relative energy inequality.
					
					To get the first term in \eqref{54-17}, we test \eqref{mode54-1} with $\mathbf{u}$ in $ L^2 $ and integrate the resultant over $(0,\tau)(\tau\leq T')$ to get
					\begin{align}
						& - \bigg[\int_{\mathbb{T}^{3}}\rho_{\alpha}  \mathbf{u}_{\alpha}\cdot \mathbf{u}\,\dx\bigg]_{t=0}^{t=\tau} +  \int_0^\tau\int_{\mathbb{T}^{3}}\rho_{\alpha}  \mathbf{u}_{\alpha}\cdot \partial_{t} \mathbf{u}\,\dx\dt\nonumber \\&\quad+ \int_0^\tau \int_{\mathbb{T}^{3}} \left (\rho_{\alpha}  \mathbf{u}_{\alpha}\otimes
						\mathbf{u}_{\alpha}: \nabla \mathbf{u}\right)\,\dx\dt
						- \int_0^\tau \int_{\mathbb{T}^{3}} \big(S(\phi_{\alpha},\mathbb{ D}\mathbf{u}_{\alpha}\big) : \nabla \mathbf{u}\,\dx\dt
						\nonumber\\
						&
						\label{mode54-11} \quad -  \int_0^\tau\int_{\mathbb{T}^{3}}(\nabla p_{\alpha}+\phi_{\alpha} \nabla \mu_{p,\alpha}-\alpha\phi_{\alpha}\nabla p_{\alpha})\cdot \mathbf{u}\,\dx\dt=0.
					\end{align}
					Next testing the continuity equation \eqref{mode54-2} with $\frac{|\mathbf{u}|^{2}}{2}$ yields
					\begin{align}
						\label{mode54-12}
						& \bigg[\int_{\mathbb{T}^{3}}\frac{1}{2}\rho_{\alpha} |\mathbf{u}|^{2}\,\dx\bigg]_{t=0}^{t=\tau}
						-\int_0^\tau\int_{\mathbb{T}^{3}}(\rho_{\alpha} \mathbf{u} \cdot \partial _{t}\mathbf{u})\,\dx\dt-\int_0^\tau\int_{\mathbb{T}^{3}}\rho_{\alpha} (\mathbf{u}_{\alpha}\cdot \nabla) \mathbf{u} \cdot \mathbf{u}\,\dx\dt=0.
					\end{align}
					Combining the energy inequality \eqref{que}, \eqref{mode54-11} and  \eqref{mode54-12},  it follows that
				\begin{align}
					&\bigg[\int_{\mathbb{T}^{3}}F(\phi_{\alpha})\,\dx\bigg]_{t=0}^{t=\tau}
					+ \bigg[\int_{\mathbb{T}^{3}}\frac{1}{2}\rho_{\alpha} |\mathbf{u}_{\alpha}-\mathbf{u}|^{2}\,\dx\bigg]_{t=0}^{t=\tau}
					+
					\bigg[\frac{1}{2}\int_{\mathbb{T}^{3}} |\Lambda^{s}\phi_{\alpha}|^{2}\,\dx\bigg]_{t=0}^{t=\tau}
					\nonumber \\&\quad
					+\int_0^{\tau}\int_{\mathbb{T}^{3}} \big( 2 \eta(\phi_{\alpha})
					\mathbb{ D}(\mathbf{u}_{\alpha}):\mathbb{ D}(\mathbf{u}_{\alpha})-\frac{2}{3}\eta(\phi_{\alpha}) (\mathrm{div} \mathbf{u}_{\alpha})^2\big)\dxdt \nonumber \\&\quad
					+\int_0^{\tau}\int\limits_{\mathbb{T}^{3}}|\nabla \mu_{p,\alpha}|^{2}\dxdt
					+  \int_0^\tau\int_{\mathbb{T}^{3}}\rho_{\alpha}  (\mathbf{u}_{\alpha}-\mathbf{u})\cdot\big (\partial_{t} \mathbf{u}+\nabla\mathbf{u}\cdot\mathbf{u}\big)\,\dx\dt
					\nonumber \\&\quad-  \int_0^\tau\int_{\mathbb{T}^{3}}\big(S(\phi_{\alpha},\mathbb{ D}\mathbf{u}_{\alpha}\big) : \nabla \mathbf{u}\,\dx\dt
					- \nonumber  \int_0^\tau\int_{\mathbb{T}^{3}}(\nabla p_{\alpha}+\phi_{\alpha} \nabla \mu_{p,\alpha}-\alpha\phi_{\alpha}\nabla p_{\alpha})\cdot \mathbf{u}\,\dx\dt
					\\&
					\quad +\int_0^\tau\int_{\mathbb{T}^{3}}(\rho_{\alpha} (\mathbf{u}_{\alpha}-\mathbf{u})\otimes(\mathbf{u}_{\alpha}-\mathbf{u}):\nabla \mathbf{u} )\,\dx\dt
					\leq 0.
					\label{mode54-122}
				\end{align}
				
				To get the last two terms in \eqref{54-17}, multiplying \eqref{mode54-3} by $-\mu$ and integrating over $(0,\tau)\times\mathbb{T}^{3}$, one obtains
				\begin{align}
					\label{mode54-13}
					& -\bigg[\int_{\mathbb{T}^{3}}\phi_{\alpha}   F'(\phi)\,\dx\bigg]_{t=0}^{t=\tau}
					\nonumber +\int_0^\tau\int_{\mathbb{T}^{3}}\phi_{\alpha} \partial _{t}\phi F''(\phi)\,\dx\dt-\int_0^\tau\int_{\mathbb{T}^{3}}\partial _{t} \Lambda^{s}\phi_{\alpha} \Lambda^{s}\phi\,\dx\dt
					\\&\quad-\int_0^\tau\int_{\mathbb{T}^{3}}\mathrm{div}(\phi_{\alpha}\mathbf{u}_{\alpha}) \mu\,\dx\dt-\int_0^\tau\int_{\mathbb{T}^{3}}(\nabla\mu_{p,\alpha}\cdot \nabla \mu)\,\dx\dt=0.
				\end{align}
				Next by testing \eqref{no density} by $-\mu_{p,\alpha}$ and $\mu$ respectively, we find
				\begin{align}
					\label{mode54-131}
					& -\int_0^\tau\int_{\mathbb{T}^{3}} \partial_{t}\phi F'(\phi_{\alpha})\,\dx\dt
					\nonumber -\int_0^\tau\int_{\mathbb{T}^{3}} \alpha\partial _{t}\phi p_{\alpha}\,\dx\dt-\int_0^\tau\int_{\mathbb{T}^{3}}\partial _{t} \Lambda^{s}\phi \Lambda^{s}\phi_{\alpha}\,\dx\dt
					\\&\quad-\int_0^\tau\int_{\mathbb{T}^{3}}\mathrm{div}(\phi\mathbf{u}) \mu_{p,\alpha}\,\dx\dt-\int_0^\tau\int_{\mathbb{T}^{3}}(\nabla\mu_{p,\alpha}\cdot \nabla \mu)\,\dx\dt=0
				\end{align}
				and
				\begin{align}
					\label{mode54-14}
					& \bigg[\int_{\mathbb{T}^{3}}\frac{1}{2}|\Lambda^{s}\phi|^{2}\,\dx\bigg]_{t=0}^{t=\tau}+\int_0^\tau\int_{\mathbb{T}^{3}} \partial_{t}\phi F'(\phi)\,\dx\dt
					\nonumber
					\\&\quad+\int_0^\tau\int_{\mathbb{T}^{3}}\mathrm{div}(\phi\mathbf{u}) \mu\,\dx\dt
					+\int_0^\tau\int_{\mathbb{T}^{3}}|\nabla\mu|^2\,\dx\dt
					=0.
				\end{align}
				
				Collecting  \eqref{mode54-13}--\eqref{mode54-14} gives rise to
				\begin{align}
					\label{mode54-15}
					& \nonumber \bigg[\frac{1}{2}\int_{\mathbb{T}^{3}} |\Lambda^{s}\phi_{\alpha} - \Lambda^{s}\phi|^{2}\,\mathrm dx\bigg]_{t=0}^{t=\tau}+
					\bigg[\int_{\mathbb{T}^{3}}\big(F(\phi_{\alpha})- F'(\phi)(\phi_{\alpha}-\phi)-F(\phi)\big)\,\dx\bigg]_{t=0}^{t=\tau}
					\\&\quad\nonumber+\int_0^{\tau}\int_{\mathbb{T}^{3}}|\nabla \mu_{p,\alpha}-\nabla \mu|^{2}\dxdt-\int_0^\tau\int_{\mathbb{T}^{3}}\mathrm{div}(\phi_{\alpha}\mathbf{u}_{\alpha})\mu\,\dx\dt
					\\&\quad\nonumber  -\int_0^\tau\int_{\mathbb{T}^{3}}\mathrm{div}(\phi\mathbf{u})\mu_{p,\alpha}\,\dx\dt+\int_0^\tau\int_{\mathbb{T}^{3}}\mathrm{div}(\phi\mathbf{u}) \mu\,\dx\dt
					\\&\quad+\int_0^\tau\int_{\mathbb{T}^{3}}\partial_{t}\phi (F'(\phi)-F''(\phi)(\phi-\phi_{\alpha})-F'(\phi_{\alpha}))\,\dx\dt
					-\int_0^\tau\int_{\mathbb{T}^{3}} \alpha\partial _{t}\phi p_{\alpha}\dx\dt\nonumber
					\\&=\bigg[\frac{1}{2}\int_{\mathbb{T}^{3}} |\Lambda^{s}\phi_{\alpha}|^{2}\,\dx\bigg]_{t=0}^{t=\tau}+\bigg[\int_{\mathbb{T}^{3}}F(\phi_{\alpha})\,\dx\bigg]_{t=0}^{t=\tau}+\int_0^{\tau}\int\limits_{\mathbb{T}^{3}}|\nabla \mu_{p,\alpha}|^{2}\dxdt,
				\end{align}
				where we have used
				\begin{align*}
					\bigg[\int_{\mathbb{T}^{3}}F'(\phi)\phi-F(\phi)\,\dx\bigg]_{t=0}^{t=\tau}-\int_0^\tau\int_{\mathbb{T}^{3}} \partial_{t}\phi\phi F''(\phi)\,\dx\dt = 0.
				\end{align*}

				It follows from \eqref{mode54-122}, \eqref{mode54-15} and direct computations that
				\begin{align}
					\label{mode54-221}
					& \nonumber \bigg[\int_{\mathbb{T}^{3}}\frac{1}{2}\rho_{\alpha} |\mathbf{u}_{\alpha}-\mathbf{u}|^{2}\,\dx+\frac{1}{2}\int_{\mathbb{T}^{3}} |\Lambda^{s}\phi_{\alpha} - \Lambda^{s}\phi|^{2}\,\mathrm dx\bigg]_{t=0}^{t=\tau}
					\\&\quad\nonumber
					+\bigg[\int_{\mathbb{T}^{3}}\big(F(\phi_{\alpha})- F'(\phi)(\phi_{\alpha}-\phi)-F(\phi)\big)\,\dx\bigg]_{t=0}^{t=\tau}
					\\&\quad\nonumber+\int_0^{\tau}\int_{\mathbb{T}^{3}}|\nabla \mu_{p,\alpha}-\nabla \mu|^{2}\dxdt
					+ \int_0^\tau\int_{\mathbb{T}^{3}} (2\eta(\phi_{\alpha}) \mathbb{ D}(\mathbf{u}_{\alpha}-\mathbf{u}): \mathbb{ D}(\mathbf{u}_{\alpha}-\mathbf{u}))\,\dx\dt
					\\&\quad\nonumber-\int_0^{\tau}\int_{\mathbb{T}^{3}}  \frac{2}{3}\eta(\phi_{\alpha}) (\mathrm{div} \mathbf{u}_{\alpha})^2\, \mathrm{d} xdt
					\\&\leq \sum_{i=1}^{10}H_{i},
				\end{align}
				where
				\begin{align}
					\label{mode54-2231}
					\nonumber\sum_{i=1}^{10}H_{i}:&=
					\int_0^\tau\int_{\mathbb{T}^{3}}  \rho_{\alpha}(\mathbf{u}_{\alpha}-\mathbf{u})\cdot \nabla p\,\dx\dt
					-2\int_0^\tau\int_{\mathbb{T}^{3}} (\eta(\phi_{\alpha})-\eta(\phi)) \mathbb{ D}\mathbf{u}: \mathbb{ D}(\mathbf{u}_{\alpha}-\mathbf{u})\,\dx\dt
					\\&\nonumber\quad-  2\int_0^\tau\int_{\mathbb{T}^{3}}(\rho_{\alpha}-1)  (\mathbf{u}_{\alpha}-\mathbf{u})\cdot \mathrm{div}(\eta(\phi)\mathbb{ D}\mathbf{u})\,\dx\dt
					-\nonumber \alpha \int_0^\tau\int_{\mathbb{T}^{3}}\phi_{\alpha}\nabla p_{\alpha}\cdot \mathbf{u}\,\dx\dt
					\\&
					\quad\nonumber+\int_0^\tau\int_{\mathbb{T}^{3}}(\phi-\phi_{\alpha})\mathbf{u} \cdot(\nabla\mu-\nabla\mu_{p,\alpha})\,\dx\dt -\int_0^\tau\int_{\mathbb{T}^{3}}(\phi_{\alpha}-\phi)(\mathbf{u}_{\alpha}-\mathbf{u})\cdot\nabla\mu\,\dx\dt
					\\&\quad
					\nonumber+ \int_0^\tau\int_{\mathbb{T}^{3}}(\rho_{\alpha}-1)  (\mathbf{u}_{\alpha}-\mathbf{u})\cdot (\phi\nabla\mu)\,\dx\dt
					-\int_0^\tau\int_{\mathbb{T}^{3}}\rho_{\alpha} (\mathbf{u}_{\alpha}-\mathbf{u})\cdot \nabla \mathbf{u} \cdot (\mathbf{u}_{\alpha}-\mathbf{u})\,\dx\dt
					\\&\quad-\int_0^\tau\int_{\mathbb{T}^{3}}\partial_{t}\phi (F'(\phi)-F''(\phi)(\phi-\phi_{\alpha})-F'(\phi_{\alpha}))\,\dx\dt
					+\int_0^\tau\int_{\mathbb{T}^{3}} \alpha\partial _{t}\phi p_{\alpha}\,\dx\dt.
				\end{align}
				Recalling that $S(\phi_{\alpha},\mathbb{ D}\mathbf{u}_{\alpha})=2\eta(\phi_{\alpha}) \mathbb{ D}( \mathbf{u}_{\alpha})-\frac{2}{3}\eta(\phi_{\alpha}) (\mathrm{div}  \mathbf{u}_{\alpha}) \mathbf{I}$, there holds
				\begin{align*}
					& \int_0^\tau\int_{\mathbb{T}^{3}} 2\eta(\phi_{\alpha}) \mathbb{ D}(\mathbf{u}_{\alpha}-\mathbf{u}): \mathbb{ D}(\mathbf{u}_{\alpha}-\mathbf{u})\,\dx\dt
					-\int_0^{\tau}\int_{\mathbb{T}^{3}}  (\frac{2}{3}\eta(\phi_{\alpha}) (\mathrm{div} \mathbf{u}_{\alpha})^2)\, \mathrm{d} xdt
					\\&=\int_0^\tau\int_{\mathbb{T}^{3}} S(\phi_{\alpha},\mathbb{ D}(\mathbf{u}_{\alpha}-\mathbf{u})):
					\big( \mathbb{ D}( \mathbf{u}_{\alpha}-\mathbf{u})-\frac{1}{3} (\mathrm{div}  \mathbf{u}_{\alpha}) \mathbf{I}\big)\dx\dt
					\nonumber\\&=\int_0^\tau\int_{\mathbb{T}^{3}} 2\eta(\phi_{\alpha})|
					\mathbb{ D}( \mathbf{u}_{\alpha}-\mathbf{u})-\frac{1}{3} (\mathrm{div}  \mathbf{u}_{\alpha}) \mathbf{I}\big|^{2}\dx\dt\nonumber
					\\&\geq C \| \mathbf{u}_{\alpha}-\mathbf{u}\|_{L^2(0,\tau;H^{1}(\mathbb{T}^{3}))}^{2}-\int_0^{\tau}\norm{\bu-\bu_\alpha}_{L^2(\bbt^3)}^{2}dt,
				\end{align*}
				where we have used the generalized Korn's inequality \cite[Theorem 10.16]{FN}
				\begin{align*}
					\int_{\mathbb{T}^{3}}|
					\mathbb{ D}( \mathbf{u}_{\alpha}-\mathbf{u})-\frac{1}{3} (\mathrm{div}  \mathbf{u}_{\alpha}) \mathbf{I}\big|^{2}\dx+ \int_{\mathbb{T}^{3}} |\mathbf{u}_{\alpha}-\mathbf{u}|\,\dx
					\geq C\| \mathbf{u}_{\alpha}-\mathbf{u}\|_{H^{1}(\mathbb{T}^{3})}^{2}.
				\end{align*}

				Following a similar argument as \cite[(3.11)]{ALN2024}, one obtains
				\begin{align}
					\nonumber
					\bigg[\frac{\kappa}{2}\int_{\bbt^3} \abs{\phi - \phi_\alpha}^2 \,\dx\bigg]_{t=0}^{t=\tau} 
					& = - \kappa \int_0^\tau \int_{\bbt^3} (\nabla \mu_{p,\alpha} - \nabla \mu) \cdot (\nabla \phi_\alpha - \nabla \phi) \dxdt \\
					\label{eqs:phi-phi-alpha}
					& \quad - \kappa \int_0^\tau \int_{\bbt^3} ( \phi_\alpha - \phi) (\bu_\alpha - \bu) \cdot \nabla \phi \dxdt
					\eqqcolon H_{11} + H_{12}.
				\end{align}
				Due to Assumption \ref{ass:main}, we can decompose $F$ into a convex part $\Phi$ and a non-convex part $- \frac{\kappa}{2} \phi^2$. Namely, 
				\begin{align*}
					& \bigg[\int_{\mathbb{T}^{3}}\big(F(\phi_{\alpha})- F'(\phi)(\phi_{\alpha}-\phi)-F(\phi)\big)\,\dx\bigg]_{t=0}^{t=\tau} \\
					& = \bigg[\int_{\mathbb{T}^{3}}\big(\Phi(\phi_{\alpha})- \Phi'(\phi)(\phi_{\alpha}-\phi)-\Phi(\phi)\big)\,\dx\bigg]_{t=0}^{t=\tau}
					- \bigg[\frac{\kappa}{2}\int_{\bbt^3} \abs{\phi - \phi_\alpha}^2 \,\dx\bigg]_{t=0}^{t=\tau}.
				\end{align*}
				Now we rewrite \eqref{mode54-221} with \eqref{eqs:phi-phi-alpha} as follows
				\begin{align}
					\label{mode54-223}
					&\nonumber \bigg[ E(\phi_{\alpha},\mathbf{u}_{\alpha}|\phi,\mathbf{u})\bigg]_{t=0}^{t=\tau}
					\\&\quad\nonumber+\int_0^{\tau}\int_{\mathbb{T}^{3}}|\nabla \mu_{p,\alpha}-\nabla \mu|^{2}\dxdt
					+  C_1 \| \mathbf{u}_{\alpha}-\mathbf{u}\|_{L^2(0,\tau;H^{1}(\mathbb{T}^{3}))}^{2}
					\\&\leq\sum_{i=1}^{12}H_{i}+ C_2 \int_0^\tau \| \mathbf{u}_{\alpha}-\mathbf{u}\|_{L^{2}(\mathbb{T}^{3})}^{2} \,\dt,
				\end{align}
				where $C_1, C_2$ are two positive constants.
				
				\textbf{\underline{Step 2}: remainder estimates and Gr\"onwall's argument.}
				The main goal in the following is to control all the terms in \eqref{mode54-2231} such that they can either be part of the relative energy or be absorbed by the relative dissipation, which finally ensures a Gr\"onwall's argument.
				
				By the H\"{o}lder inequality and Lemma \ref{lem:limit 3}, we have
				\begin{align}
					\label{mode54-224}
					\nonumber H_{1}&=
					\int_0^\tau\int_{\mathbb{T}^{3}}  \rho_{\alpha}\mathbf{u}_{\alpha}\cdot \nabla p\,\dx\dt+\int_0^\tau\int_{\mathbb{T}^{3}}  (\rho_{\alpha}-1)\mathbf{u}\cdot \nabla p\,\dx\dt-\int_0^\tau\int_{\mathbb{T}^{3}}  \mathbf{u}\cdot \nabla p\,\dx\dt
					\\&\nonumber \lesssim \alpha
					+ \alpha\big(\|\phi_{\alpha}\|_{L^{\infty}(0,\tau;L^{\infty}(\mathbb{T}^{3}))}+1\big)\|\mathbf{u}\|_{L^{\infty}(0,\tau;L^{2}(\mathbb{T}^{3}))}\| \nabla p\|_{L^{1}(0,\tau;L^{2}(\mathbb{T}^{3}))}
					\\&\leq C(T',D)\alpha.
				\end{align}
				
				By the Young's inequality and Sobolev embedding, one has
				\begin{align}
					\label{mode54-2241}
					\nonumber H_{2}
					&\lesssim \|\mathbb{ D}\mathbf{u}\|_{L^{\infty}(0,\tau;L^{4}(\mathbb{T}^{3}))}\int_0^\tau\|\eta(\phi_{\alpha})-\eta(\phi)\|_{L^{4}(\mathbb{T}^{3})}
					\|\mathbb{ D}(\mathbf{u}_{\alpha}-\mathbf{u})\|_{L^{2}(\mathbb{T}^{3})}\dt
					\\&\leq C(T',D)\int_0^\tau\|\Lambda^s\phi-\Lambda^s\phi_{\alpha}\|_{L^{2}(\bbt^3)}^2\dt+\frac{C_1}{4} \| \mathbf{u}_{\alpha}-\mathbf{u}\|_{L^2(0,\tau;H^{1}(\mathbb{T}^{3}))}^{2}.
				\end{align}
				In a similar way we can show
				\begin{align}
					\label{mode54-225}
					\nonumber H_{3}
					&\lesssim\nonumber \alpha\|\mathbf{u}_{\alpha}-\mathbf{u}\|_{L^{\infty}(0,\tau;L^{2}(\mathbb{T}^{3}))}(\|\nabla\phi\|_{L^{\infty}(0,\tau;L^{3}(\mathbb{T}^{3}))}
					\|\mathbb{ D}\mathbf{u}\|_{L^{1}(0,\tau;L^{6}(\mathbb{T}^{3}))}
					\\&\nonumber\quad+
					\|\mathbf{u}\|_{L^{1}(0,\tau;H^{2}(\mathbb{T}^{3}))})
					\\&\leq C(T',D)\alpha.
				\end{align}

				Taking $\mathbf{g}=\mathbf{u}$ in Lemma \ref{lem:limit 2} directly leads to
				\begin{align}
					\label{mode54-226}
					H_{4}=\alpha\int_0^\tau\int_{\mathbb{T}^{3}}p_{\alpha}\nabla\phi_{\alpha} \cdot \mathbf{u}\,\dx\dt\leq C(T',D)\alpha.
				\end{align}
				
				In view of the Sobolev embedding and Young's inequality, we get
				\begin{align}
					\label{mode54-227}
					\nonumber H_{5}
					& \leq\int_0^\tau\|\phi-\phi_{\alpha}\|_{L^{2}(\mathbb{T}^{3})}\|\mathbf{u}\|_{L^{\infty}(\mathbb{T}^{3})}
					\|\nabla\mu-\nabla\mu_{p,\alpha}\|_{L^{2}(\mathbb{T}^{3})}\,\dt
					\\&\leq \nonumber \|\mathbf{u}\|_{L^{\infty}(0,T;L^{\infty}(\mathbb{T}^{3}))}\int_0^\tau\|\Lambda^s\phi-\Lambda^s\phi_{\alpha}\|_{L^{2}(\mathbb{T}^{3})}\|\nabla\mu-\nabla\mu_{p,\alpha}\|_{L^{2}(\mathbb{T}^{3})}\,\dt
					\\&\leq  C(T',D)\int_0^\tau\|\Lambda^s\phi-\Lambda^s\phi_{\alpha}\|_{L^{2}(\bbt^3)}^2 \,\dt+
					\frac{1}{4}\int_0^\tau\|\nabla\mu-\nabla\mu_{p,\alpha}\|_{L^{2}(\mathbb{T}^{3})}^{2}\,\dt
				\end{align}
				and
				\begin{align}
					\label{mode54-228}
					\nonumber H_{6}
					&\leq\int_0^\tau\|\phi-\phi_{\alpha}\|_{L^{6}(\mathbb{T}^{3})}\|\nabla\mu\|_{L^{3}(\mathbb{T}^{3})}\|\mathbf{u}_{\alpha}-\mathbf{u}\|_{L^{2}(\mathbb{T}^{3})}\,\dt
					\\&\leq\nonumber C\|\nabla\mu\|_{L^{\infty}(0,T;L^{3}(\mathbb{T}^{3}))}\int_0^\tau\|\Lambda^s\phi-\Lambda^s\phi_{\alpha}\|_{L^{2}(\bbt^3)}\|\mathbf{u}_{\alpha}-\mathbf{u}\|_{L^{2}(\mathbb{T}^{3})}\,\dt
					\\&\leq C(T',D)\bigg(\int_0^\tau\|\Lambda^s\phi-\Lambda^s\phi_{\alpha}\|_{L^{2}(\bbt^3)}^2\,\dt
					+\int_0^\tau\|\mathbf{u}_{\alpha}-\mathbf{u}\|_{L^{2}(\mathbb{T}^{3})}^{2}\,\dt\bigg).
				\end{align}

				Thanks to  the Young's inequality  we have
				\begin{align}
					\label{mode54-229}
					H_{7}
					\lesssim\alpha \|\phi\|_{L^{\infty}(Q_\tau)}\int_0^\tau\|\mathbf{u}_{\alpha}-\mathbf{u}\|_{L^{2}(\mathbb{T}^{3})}\|\nabla\mu\|_{L^{2}(\mathbb{T}^{3})}\,\dt\leq C(T',D)\alpha.
				\end{align}
				Moreover, one has
				\begin{align}
					\label{mode54-2210}
					\nonumber H_{8}
					&\lesssim C\int_0^\tau\|\rho_{\alpha}\|_{L^{\infty}(\mathbb{T}^{3})}\|\mathbf{u}_{\alpha}-\mathbf{u}\|_{L^{2}(\mathbb{T}^{3})}\|\mathbf{u}_{\alpha}-\mathbf{u}\|_{L^{6}(\mathbb{T}^{3})}\|\nabla\mathbf{u}\|_{L^{3}(\mathbb{T}^{3})}\,\dt
					\\&\leq \nonumber  C(T',D)\|\nabla\mathbf{u}\|_{L^{\infty}(0,T;L^{3}(\mathbb{T}^{3}))} \int_0^\tau\|\mathbf{u}_{\alpha}-\mathbf{u}\|_{L^{2}(\mathbb{T}^{3})}\|\mathbf{u}_{\alpha}-\mathbf{u}\|_{L^{6}(\mathbb{T}^{3})}\,\dt
					\\&\leq  C(T',D)\int_0^\tau\|\mathbf{u}_{\alpha}-\mathbf{u}\|_{L^{2}(\mathbb{T}^{3})}^{2}\,\dt
					+\frac{C_1}{4} \| \mathbf{u}_{\alpha}-\mathbf{u}\|_{L^2(0,\tau;H^{1}(\mathbb{T}^{3}))}^{2}
				\end{align}
				and
				\begin{align}
					\label{mode54-2211}
					H_{9}
					\leq C(T',D)\int_0^\tau\|\partial_{t}\phi\|_{L^{2}(\mathbb{T}^{3})}\|\phi-\phi_{\alpha}\|^{2}_{L^{4}(\mathbb{T}^{3})}\,\dt
					\leq  C(T',D) \int_0^\tau\|\Lambda^s\phi-\Lambda^s\phi_{\alpha}\|_{L^{2}(\bbt^3)}^2\,\dt.
				\end{align}
				Besides, taking $f=\partial _{t}\phi$ in  Lemma \ref{lem:limit 1}  implies
				\begin{align}
					\label{mode54-2212}
					&  H_{10}
					\leq  C(T',D)\alpha.
				\end{align}
				Finally, in view of H\"older's and Young's inequalities, one ends up with
				\begin{align}
					\label{mode54-H11}
					H_{11}
					&\leq \frac{1}{4}\int_0^{\tau}\int_{\mathbb{T}^{3}}|\nabla \mu_{p,\alpha}-\nabla \mu|^{2}\dxdt+C(T',D) \int_0^\tau\|\Lambda^s\phi-\Lambda^s\phi_{\alpha}\|_{L^{2}(\bbt^3)}^2\,\dt.
					\\
					\label{mode54-H12}
					H_{12}&\leq C(T',D)\left( \int_0^\tau\|\Lambda^s\phi-\Lambda^s\phi_{\alpha}\|_{L^{2}(\bbt^3)}^2\,\dt+\int_0^\tau\|\mathbf{u}_{\alpha}-\mathbf{u}\|_{L^{2}(\mathbb{T}^{3})}^{2}\,\dt\right).
				\end{align}
				
				Submitting  \eqref{mode54-224}--\eqref{mode54-H12} into \eqref{mode54-223}, we arrive at
				\begin{align}
					\label{mode54-2213}
					&\nonumber \bigg[ E(\phi_{\alpha},\mathbf{u}_{\alpha}|\phi,\mathbf{u})\bigg]_{t=0}^{t=\tau}
					\\&\nonumber\quad+\frac{1}{2}\int_0^{\tau}\int_{\mathbb{T}^{3}}|\nabla \mu_{p,\alpha}-\nabla \mu|^{2}\dxdt
					+ \frac{ C_1}{2} \| \mathbf{u}_{\alpha}-\mathbf{u}\|_{L^2(0,\tau;H^{1}(\mathbb{T}^{3}))}^{2}
					\\&\nonumber\leq C(T',D)\alpha+C(T',D)\int_0^\tau E(\phi_{\alpha},\mathbf{u}_{\alpha}|\phi,\mathbf{u})(t)\,\dt
					\\&\quad+(C(T',D)+C_2)\int_0^\tau E(\phi_{\alpha},\mathbf{u}_{\alpha}|\phi,\mathbf{u})(t)\,\dt.
				\end{align}
				
				Finally, we  apply the Gr\"onwall's inequality to \eqref{mode54-2213} and  derive 
				\begin{align}
					\nonumber
					&  E(\phi_{\alpha},\mathbf{u}_{\alpha}|\phi,\mathbf{u})(\tau)
					\\&\nonumber\quad+\frac{1}{2}\int_0^{\tau}\int_{\mathbb{T}^{3}}|\nabla \mu_{p,\alpha}-\nabla \mu|^{2}\dxdt
					+\frac{ C_1}{2} \| \mathbf{u}_{\alpha}-\mathbf{u}\|_{L^2(0,\tau;H^{1}(\mathbb{T}^{3}))}^{2}
					\\&
					\label{eqs:Gronwall-E-tilde}
					\leq C(T',D)\bigg(\alpha+ E(\phi_{\alpha},\mathbf{u}_{\alpha}|\phi,\mathbf{u})(0)\bigg).
				\end{align}
				This completes the proof of Theorem \ref{thm:limit}. \qed
				
				\appendix
				\section{Technical Tools}
				\subsection{Weak Neumann Laplace operator}
				This part is adapted from \cite{AH1}, one can also see \cite{FFHL2024}.
				
				Define  the orthogonal projection onto $L^{2}_{(0)}(\mathbb{T}^{3})=\{f\in L^{2}(\mathbb{T}^{3}):\bar{f}\triangleq\frac{1}{|\mathbb{T}^{3}|}\int_{\mathbb{T}^{3}}f(x)\,\mathrm dx=0\}$ by
				\begin{align*}
					P_{0}f\triangleq f-\bar{f}.
				\end{align*}
				Furthermore, we define  a Hilbert space $H_{(0)}^{1}(\mathbb{T}^{3})$ by
				\begin{align*}
					H_{(0)}^{1}(\mathbb{T}^{3})=H^{1}(\mathbb{T}^{3})\cap L^{2}_{(0)}(\mathbb{T}^{3})
				\end{align*}
				with the following inner product
				\begin{align*}
					(c,d)_{H_{(0)}^{1}(\mathbb{T}^{3})}=(\nabla c,\nabla d)_{L^{2}(\mathbb{T}^{3})},\ \  \forall c,d\in H_{(0)}^{1}(\mathbb{T}^{3}).
				\end{align*}

				The weak Neumann-Laplace operator $\Delta:H_{(0)}^{1}(\mathbb{T}^{3}))\rightarrow H_{(0)}^{-1}(\mathbb{T}^{3}))$ is defined by
				\begin{align}
					-\langle\Delta u,\varphi\rangle_{H_{(0)}^{-1}(\mathbb{T}^{3}), H_{(0)}^{1}(\mathbb{T}^{3})}=(\nabla u,\nabla \varphi)_{L^{2}(\mathbb{T}^{3})}, \quad\text{for $u,\varphi\in H_{(0)}^{1}(\mathbb{T}^{3})$},\label{def:wnl}
				\end{align}
				where $H_{(0)}^{-1}(\mathbb{T}^{3})=\big(H_{(0)}^{1}(\mathbb{T}^{3})\big)'$.
				By the Lemma of Lax-Milgram, for every $f\in H_{(0)}^{-1}(\mathbb{T}^{3})$  there is a unique $u\in H_{(0)}^{1}(\mathbb{T}^{3})$
				such that $\Delta u=f$ .

	Moreover, we embed $H_{(0)}^{1}(\mathbb{T}^{3})$ and $L^{2}_{(0)}(\mathbb{T}^{3})$ into $H_{(0)}^{-1}(\mathbb{T}^{3})$ in the standard way by defining
	\begin{align*}
		\langle c,\varphi\rangle_{H_{(0)}^{-1}(\mathbb{T}^{3}),H_{(0)}^{1}(\mathbb{T}^{3})}=\int\limits_{\mathbb{T}^{3}} c(x)\varphi(x)\,\mathrm dx, \quad\text{for $\varphi\in H_{(0)}^{1}(\mathbb{T}^{3}),c\in L^{2}_{(0)}(\mathbb{T}^{3})$},
	\end{align*}
	which gives us a useful inequality
	\begin{align}\label{HHH}
		\|f\|_{L^{2}(\mathbb{T}^{3})}^{2}\leq \|f\|_{H_{(0)}^{1}(\mathbb{T}^{3})} \|f\|_{H_{(0)}^{-1}(\mathbb{T}^{3})},\forall f\in H_{(0)}^{1}(\mathbb{T}^{3}).
	\end{align}
	Furthermore, assuming $\Delta u=f$ for $f\in H_{(0)}^{-1}(\mathbb{T}^{3})$, we have
	\begin{align*}
		\|u\|_{H_{(0)}^{1}(\mathbb{T}^{3})}\leq  \|f\|_{H_{(0)}^{-1}(\mathbb{T}^{3})}.
	\end{align*}

Finally, we define $\mathrm{div}:L^{2}(\mathbb{T}^{3})\rightarrow H_{(0)}^{-1}(\mathbb{T}^{3})$ as
\begin{align}\label{def:div}
	\langle \mathrm{div} f,\varphi \rangle_{H_{(0)}^{-1}(\mathbb{T}^{3}),H_{(0)}^{1}(\mathbb{T}^{3})}=-(f,\nabla\varphi)_{L^{2}(\mathbb{T}^{3})}, \quad\text{for $\varphi\in H_{(0)}^{1}(\mathbb{T}^{3})$},
\end{align}
which and \eqref{def:wnl} imply that $\Delta u=\mathrm{div}\nabla u$ for $u\in  H_{(0)}^{1}(\mathbb{T}^{3})$.

\subsection{Some auxiliary lemmas}
In this part, we show some auxiliary results.
\begin{lemma}\label{lem1}\cite {AH}
	Let $X,Y$ be two Banach spaces satisfying that $Y\hookrightarrow X$ and $X'\hookrightarrow Y'$ densely. For $0<T<\infty$, we have that $L^{\infty}(0,T;Y)\cap BUC([0,T];X)\hookrightarrow BC_{w}([0,T];Y)$.
\end{lemma}



Similar to \cite[Lemma 4.2]{AH}, we can also define the very weak solution of Laplace operator in $\mathbb{T}^{3}$.
\begin{lemma}\label{lem3}
	Let $1<p<\infty$ and $\frac{1}{p}+\frac{1}{p'}=1$. Then for every $f\in (W_{(0)}^{2,p'}(\mathbb{T}^{3}))'$, we can find a unique solution $u\in L^{p}(\mathbb{T}^{3})$ such that
	\begin{align*}
		&(u,\Delta\varphi)=\inner{f}{\varphi}_{(W_{(0)}^{2,p'}(\mathbb{T}^{3}))',W_{(0)}^{2,p'}(\mathbb{T}^{3})}
	\end{align*}
	for all $\varphi\in W_{(0)}^{2,p'}(\mathbb{T}^{3}):=\{f\in \varphi\in W^{2,p'}(\mathbb{T}^{3}):\bar{f}=0\}$. Moreover, there exists a positive constant  $C_{p}$ such that
	\begin{align*}
		& \|u\|_{L^{p}(\mathbb{T}^{3})}\leq C_{p}\|f\|_{(W_{(0)}^{2,p'}(\mathbb{T}^{3}))'}.
	\end{align*}
\end{lemma}

Next, we show following important results about $\Lambda^{2s}$ that will used to show the existence of approximate solutions.
\begin{lemma}\label{lem5}
	The pseudo-differential operator $\Lambda^{2s}:H^{s}\rightarrow H^{-s}$ is of type M (see \cite{SRE}). For $\lambda>0$, we can derive that $\lambda+\Lambda^{2s}:H^{s}\rightarrow H^{-s}$ is bijective with strongly continuous inverse.
\end{lemma}
\begin{proof}
	We can easily show  for $u,v \in H^{s}(\mathbb{T}^{3})$
	\begin{align*}
		\langle \Lambda^{2s}u-\Lambda^{2s}v,u-v \rangle_{H^{-s},H^{s}}=\|\Lambda^{s}(u-v)\|_{L^{2}}^2\geq0
	\end{align*}
	and
	\begin{align*}
		&\langle \Lambda^{2s}(u+(h+t)v),v \rangle_{H^{-s},H^{s}}-\langle \Lambda^{2s}(u+tv),v \rangle_{H^{-s},H^{s}}
		\leq h\| v\|_{H^{s}(\mathbb{T}^{3})}^{2}  , \text{ for } h>0.
	\end{align*}
	Hence, we can prove that  $\Lambda^{2s}$ is of type M by  \cite[Lemma 2.1 in Chapter II]{SRE}.

	In order to prove that $\lambda+\Lambda^{2s}$ is bijective. We first verify that
	\begin{align*}
		\lambda\langle u,u\rangle_{H^{-s},H^{s}}+\langle \Lambda^{2s}u,u \rangle_{H^{-s},H^{s}}=\lambda(u,u)_{L^{2}}+(\Lambda^{s}u,\Lambda^{s}u)_{L^{2}}\gtrsim \| u\|_{H^{s}(\mathbb{T}^{3})}^{2},
	\end{align*}
	hence we can get that $\lambda+\Lambda^{2s}$ is coercive. By \cite[Lemma 2.1 and Corollary 2.2 in Chapter II]{SRE}, we then deduce that $\lambda+\Lambda^{2s}$ is surjective. Besides, we can prove that
	\begin{align*}
		\langle \Lambda^{2s}u-\Lambda^{2s}v,u-v \rangle_{H^{-s},H^{s}}=0 \ \text{if only if $u=v$.}
	\end{align*}
	Combining the above results, the proof of the first part of this lemma is completed.

	Now we begin to show strongly continuous inverse of $\lambda+\Lambda^{2s}$. We assume that
	$(\lambda+\Lambda^{2s})u_{n}=f_{n},(\lambda+\Lambda^{2s})u=f$ and $f_{n}\rightarrow f$ in $H^{-s}(\mathbb{T}^{3})$, then we have
	\begin{align*}
		\| u_{n}-u\|_{H^{s}(\mathbb{T}^{3})}^{2}&\lesssim\lambda\langle u_{n}-u,u_{n}-u\rangle_{H^{-s},H^{s}}+\langle \Lambda^{2s}u_{n}-\Lambda^{2s}u,u_{n}-u \rangle_{H^{-s},H^{s}}
		\\&=\langle f_{n}-f,u_{n}-u \rangle_{H^{-s},H^{s}}\rightarrow0(n\rightarrow\infty),
	\end{align*}
	i.e., $u_{n}\rightarrow u$ in $H^{s}(\mathbb{T}^{3})$.

	Hence, we get the desired results.
\end{proof}

The following lemma ensures the energy inequality of the limit passage.
\begin{lemma}\label{lem4}\cite{AH}
	Let $E:[0,T)\rightarrow[0,\infty)$ be a lower semi-continuous function and $D:(0,T)\rightarrow [0,\infty)$ be a integrable function, $0<T\leq\infty$. Then
	\begin{align*}
		&E(0)\varphi(0)+\int_0^{T}E(t)\varphi '(t)dt\geq \int_0^{T}D(t)\varphi(t)dt
	\end{align*}
	holds for all $\varphi\in W^{1,1}(0,T)$ with $\varphi(T)=0$ if and only if
	\begin{align*}
		&E(t)+\int_s^{t}D(\tau)d\tau\leq E(s)
	\end{align*}
	for almost all $0\leq s\leq t<T$ including $s=0$.
\end{lemma}
The following lemma shows that for given $R,\theta>0$, we can obtain $-1-\theta<\phi<1+\theta$ under the condition $E_{\text{free}}(\phi)\leq R$ with a suitable extension of $F$, see also \cite{AH}.
\begin{lemma}\label{lem2}
	Let $R,\theta>0$, $s>\frac{3}{2}$. Then there is an extension $F\in C^{3}(\mathbb{R})$ fulfilling $F(\phi)\geq0$, $F''(\phi)\geq -M> -\infty$ such that such that 
	\begin{align*}
		-1-\theta<\phi<1+\theta
	\end{align*}
	if $ \frac{1}{|\bbt^3|} \int_{\bbt^3} \phi \,\dx=:\overline{\phi} \in(-\infty,+\infty)$  and 
	\begin{align}
		& \int_{\mathbb{T}^{3}}\Big( F(\phi)+\frac{1}{2}|\Lambda^{s} \phi|^2\Big)\, \mathrm dx\leq R.\label{condition}
	\end{align}
\end{lemma}
\begin{proof}
	Firstly, we can note  that
	\begin{align*} \|\phi\|^2_{L^{2}(\mathbb{T}^{3})}&=|\hat{\phi}(0)|^2+\sum_{k\in \mathbb{Z}^3,|k|\geq 1}|\hat{\phi}(k)|^2
		\\&=\absm{\bbt^3}^2|\overline{\phi}|^{2}+\sum_{k\in \mathbb{Z}^3,|k|\geq 1}|\hat{\phi}(k)|^2
		\\&\leq \absm{\bbt^3}^2|\overline{\phi}|^{2}+\sum_{k\in \mathbb{Z}^3,|k|\geq 1}|k|^{2s}|\hat{\phi}(k)|^2
		\\&= \absm{\bbt^3}^2|\overline{\phi}|^{2}+\int_{\mathbb{T}^{3}}|\Lambda^{s} \phi|^2dx
	\end{align*}
	which combined with \eqref{condition} implies
	\begin{align*}
		\|\phi\|^2_{L^{2}(\mathbb{T}^{3})}&\leq\absm{\bbt^3}^2|\overline{\phi}|^{2}+2R.
	\end{align*}
	Therefore $\phi\in H^{s}(\mathbb{T}^{3})$ and $\|\phi\|^2_{H^{s}(\mathbb{T}^{3})}\leq\absm{\bbt^3}^2|\overline{\phi}|^{2}+4R$. Mimicking the argument in \cite[Lemma 2.3]{AH} with the fact $H^{s}(\mathbb{T}^{3})\hookrightarrow C^{\alpha}(\mathbb{T}^{3})$ with $\alpha=s-\frac{3}{2}$ due to $s>\frac{3}{2}$, we then complete the proof of this lemma.
\end{proof}

\section*{Acknowledgments}

This work was partially done when Y. Liu visited School of Mathematics and Statistics of Anhui Normal University. The hospitality is gratefully appreciated.

M. Fei is partially supported by NSF of China under Grant No. 11871075, 11971357, and NSF of Anhui Province of China under Grant No. 2308085J10. Y. Liu is partially supported by the startup funding from Nanjing Normal University, the NSF of Jiangsu Province  of China under Grant No.~BK20240572, the NSF of Jiangsu Higher Education Institutions of China under Grant No.~24KJB110020, and the China Postdoctoral Science Foundation under Grant No.~2025M773078. The support is gratefully acknowledged.

\section*{Compliance with Ethical Standards}
\subsection*{Date avability}
Data sharing not applicable to this article as no datasets were generated during the current study.
\subsection*{Conflict of interest}
The authors declare that there are no conflicts of interest.


\begin{thebibliography}{99}
	\addcontentsline{toc}{section}{References}
	
	\bibitem{AbelsARMA2009} H. Abels, \textit{On a diffuse interface model for two-phase flows of viscous, incompressible fluids with matched densities}, Arch. Ration. Mech. Anal. 194(2009), no.~2, 463--506.
	
	\bibitem {AH}H. Abels, \textit{Existence of Weak Solutions for a Diffuse Interface Model for Viscous, Incompressible Fluids with General Densities}, Commun. Math. Phys. 289(2009), 45-73.
	
	\bibitem {AH1} H. Abels, \textit{Strong well-posedness of a diffuse interface model for a viscous, quasi-incompressible two-phase flow},  SIAM J.
	Math. Anal. 44(2012), 316-340.
	
	\bibitem {ABG} H. Abels, S. Bosia and M. Grasselli, \textit{Cahn-Hilliard equation with nonlocal singular free energies}, Ann.Mat. Pura Appl. 194(2015), 1071-1106.
	
	\bibitem {ADG} H. Abels, D. Depner and H. Garcke, \textit{Existence of weak solutions for a diffuse interface model for two-phase
		flows of incompressible fluids with different densities}, J. Math. Fluid Mech. 15(2013), 453-480.
	
	
	
	\bibitem{AF2008} H. Abels and E. Feireisl, \textit{On a diffuse interface model for a two-phase flow of compressible viscous fluids}, Indiana Univ. Math. J. 57(2008), no.~2, 659-698.
	
	\bibitem {AGG2024}  H. Abels, H. Garcke and A. Giorgini, \textit{Global regularity and asymptotic stabilization for the incompressible Navier-Stokes-Cahn-Hilliard model with unmatched densities},  Math. Ann. 389(2024), 1267-1321.
	
	\bibitem {AGG2012}  H. Abels, H. Garcke and G. Gr\"{u}n, \textit{Thermodynamically consistent, frame indifferent diffuse interface models for incompressible
		two-phase flows with different densities},  Math. Models Methods Appl. Sci. 22(2012), 1150013.
	
	\bibitem {AGW2025}  H. Abels, H. Garcke and J. Wittmann, \textit{Diffuse interface models for two-Phase flows with phase transition: modeling and existence of weak solutions}, arXiv:2505.05383.
	
	\bibitem {ALN2024}H. Abels, Y. Liu and \v{S}. Ne\v{c}asov\'{a}, \textit{Low Mach number limit of a diffuse interface model for
		two-phase flows of compressible viscous fluids}, GAMM-Mitt. 47(2024), 1-15.
	
	
	\bibitem {AH2}  H. Amann, \textit{Linear and Quasilinear Parabolic Problems}, Volume 1: Abstract
	Linear Theory. Birkh\"{a}user, Basel-Boston -Berlin, 1995.
	
	\bibitem {AM}M. Ainsworth, Z. P. Mao, \textit{ Well-posedness of the Cahn-Hilliard equation with fractional free energy and its Fourier Galerkin approximation},
	Chaos, Solitons and Fractals. 102(2017), 264-273.
	
	\bibitem {ASS}G. Akagi, G. Schimperna and A. Segatti, \textit{ Fractional Cahn-Hilliard, Allen-Cahn and porousmedium equations},
	J. Differential Equations 261(2016), 2935-2985.
	
	\bibitem{ADGC} G. L. Aki, W. Dreyer, J. Giesselmann, and C. Kraus, \textit{ A quasi-incompressible diffuse interface model
		with phase transition}, Math. Models Methods Appl. Sci. 24(2014), no. 5, 827-861.
	
	\bibitem {BH}P. W. Bates, J. Han, \textit{ The Dirichlet boundary problem for a nonlocal Cahn-Hilliard equation},
	J. Math. Anal. Appl. 311(2005), 289-312.
	
	\bibitem {Boyer2002} F. Boyer, \textit{A theoretical and numerical model for the study of incompressible mixture
		flows}, Comput. Fluids. 31(2002), 41-68.
	
	\bibitem{CH1958} J. W. Cahn and J .E. Hilliard, Free energy of a nonuniform system I. Interfacial free energy, J. Chem. Phys. 28(1958), 258-267.
	%
	
	
	
	\bibitem {DSS2007}  H. Ding, P. D. M. Spelt and C. Shu, \textit{Diffuse interface model for incompressible two-phase flows with large density ratios},  J. Comput. Phys. 226(2007), 2078-2095.
	
	\bibitem {ELC}L. C. Evans, \textit{Partial differential equations}, volume 19 of Graduate Studies in Mathematics. American Mathematical Society, Providence, RI, 2010, second edition.
	
	\bibitem{EPPS2024} C. Elbar, B. Perthame, A. Poiatti and J. Skrzeczkowski, \textit{Nonlocal Cahn-Hilliard equation with degenerate mobility: incompressible limit and convergence to stationary states}, Arch. Ration. Mech. Anal. 248(2024), no.~3, Paper No. 41, 38 pp.
	
	\bibitem {FS1}S. Frigeri, \textit{Global existence of weak solutions for a nonlocal model for two-phase flows of incompressible fluids with unmatched densities}, Math. Models Methods. Appl. Sci. 26(2016), 1955-1993.
	
	\bibitem {FS}S. Frigeri, \textit{On a nonlocal Cahn-Hilliard/Navier-Stokes system with degenerate mobility and singular potential for incompressible fluids with
		different densities}, Annales de l'Institut Henri Poincar\'{e}-Analyse non lin\'{e}aire. 38(2020), 647-687.
	
	
	\bibitem {FG1}S. Frigeri, M. Grasselli, \textit{Nonlocal Cahn-Hilliard-Navier-Stokes systems with singular potentials}, Dyn. Partial Differ. Equ. 9(2012), 273-304.
	
	\bibitem {FGG}S. Frigeri, C.G. Gal and M. Grasselli, \textit{On nonlocal Cahn-Hilliard-Navier-Stokes systems in two dimensions},  J. Nonlinear Sci.
	26(2016), 847-893.
	
	\bibitem {FGK}S. Frigeri, M. Grasselli and P. Krej\v{c}\'{l}, \textit{Strong solutions for two-dimensional nonlocal Cahn-Hilliard-Navier-Stokes systems}, J. Differential. Equations 255 (2013), 2587-2614.
	
	
	
	
	
	
	
	\bibitem{FFHL2024} M. Fei, X. Fei, D. Han and Y. Liu, \textit{Local-in-time existence of strong solutions to a quasi-incompressible Cahn-Hilliard-Navier-Stokes system}, arXiv: 2411.09455.
	
	\bibitem {FJN}E. Feireisl, B. J. Jin, and A. Novot\'{n}y, \textit{Relative Entropies, Suitable Weak Solutions, and Weak-Strong Uniqueness
		for the Compressible Navier-Stokes System}, J. Math. Fluid. Mech. 14(2012), 717-730.
	
	\bibitem {FLM}E. Feireisl, Y. Lu and J. M\'alek, \textit{On PDE analysis of flows of quasi-incompressible fluids}, Z. Angew. Math. Mech. 96(2016), 491-508.
	
	\bibitem {FLN}E. Feireisl, Y. Lu and A. Novot\'{n}y, \textit{Weak-strong uniqueness for the compressible Navier-Stokes equations with a hard-sphere pressure law}, Sci. China Math. 61(2018), 2003-2016.
	
	\bibitem {FN}E. Feireisl, A. Novot\'{n}y, \textit{Singular Limits in Thermodynamics of Viscous Fluids}, Adv. Math. Fluid Mech. Birkh\"{a}user Verlag, Basel, 2009.
	
	
	\bibitem {FPP} E. Feireisl, M. Petcu and D. Prak, \textit{Relative energy approach to a diffuse interface model of a compressible two-phase flow}, Math. Methods Appl. Sci. 42(2019), 1465-1479.
	
	
	
	\bibitem {GGG} C. G. Gal, A. Giorgini and M. Grasselli, \textit{ The nonlocal Cahn-Hilliard equation with singular potential: well-posedness, regularity and strict separation property}, J. Differential. Equations 263(2017), 5253-5297.
	
	
	\bibitem {GMW} C. G. Gal, M. Grasselli and H. Wu, \textit{Global weak solutions to a diffuse interface model for incompressible two-phase flows with moving contact lines and different densities}, Arch. Ration. Mech. Anal. 234(2019), 1-56.
	
	\bibitem{Giorgini2021} A. Giorgini, \textit{Well-posedness of the two-dimensional Abels-Garcke-Grün model for two-phase flows with unmatched densities}, Calc. Var. Partial Differential Equations 60(2021), no. 3, Paper No. 100, 40 pp.
	
	\bibitem{Giorgini2024} A. Giorgini, \textit{On the separation property and the global attractor for the nonlocal Cahn-Hilliard equation in three dimensions}, J. Evol. Equ. 24(2024), no. 2, Paper No. 21, 16 pp.
	
	\bibitem {GMT}A. Giorgini, A. Miranville and R. Temam, \textit{Uniqueness and regularity for the Navier-Stokes-Cahn-Hilliard system}, SIAM J. Math. Anal. 51(2019), 2535-2574.
	
	
	
	\bibitem {GPV1996}  M. E. Gurtin, D. Polignone and J. Vi\~{n}als, \textit{Two-phase binary fluids and immiscible fluids described by an order parameter},
	Math. Models Methods Appl. Sci. 6(1996), 815-831.
	
	\bibitem {GCLLJ} Z. L. Guo,  Q. Chen, P. Lin, C. Liu and J. Lowengrub, \textit{ Second order
		approximation for a quasi-incompressible Navier-Stokes Cahn-Hilliard system of two-phase flows with variable density},
	J. Comput. Phys. 448(2022), Paper No. 110727, 17 pp.
	
	\bibitem {HH} P. C. Hohenberg, B. I. Halperin, \textit{ Theory of dynamic critical phenomena},
	Rev. Mod. Phys. 49(1977), 435.
	
	\bibitem{HKP2024} C. Hurm, P. Knopf and A. Poiatti, \textit{Nonlocal-to-local convergence rates for strong solutions to a Navier-Stokes-Cahn-Hilliard system with singular potential}, Comm. Partial Differential Equations 49(2024), no.~9, 832-871.
	
	\bibitem {LT1998} J. Lowengrub, L. Truskinovsky, \textit{Quasi-incompressible Cahn-Hilliard fluids and topological transitions}, Proc. R. Soc. Lond.
	Ser. A, Math. Phys. Eng. Sci. 454(1998), 2617-2654.
	
	\bibitem {MR} S. Melchionna, E. Rocca, \textit{On a nonlocal Cahn-Hilliard equation with a reaction term}, Adv. Math. Sci. Appl. 24(2014), 461-497.
	
	\bibitem {NL} L. Nirenberg, \textit{Topics in nonlinear functional analysis},  Courant Lecture Notes in Mathematics,vol. 6, Courant Institute of Mathematical Sciences, New York University, New York, 2001.
	
	\bibitem{Poiatti2025} A. Poiatti, \textit{The 3D strict separation property for the nonlocal Cahn-Hilliard equation with singular potential}, Anal. PDE 18(2025), no.~1, 109-139.
	
	\bibitem {SRE} R.~E.~Showalter, \textit{Monotone Operators in Banach Space and Nonlinear Partial Differential Equations}, Volume 49 of Mathematical Surveys and Monographs. Providence, RI: Amer. Math. Soc., 1997.
	
	\bibitem {SY} Y. Sawano \textit{Theory of Besov Spaces}, Dev. Math,. vol. 56. Springer, Singapore, 2018.
	
	\bibitem {SYW}J. Shen, X. F. Yang and Q. Wang, \textit{Mass and volume conservation in phase field models for binary fluids},
	Commun. Comput. Phys. 13(2013), 1045-1065.	
	
	\bibitem {SRSvBvdZ2018} M. Shokrpour Roudbari, G. \c{S}im\c{s}ek, E. H. van Brummelen and  K. G. van der Zee, \textit{ Diffuse-interface two-phase flow models with different densities: a new quasi-incompressible form and a linear energy-stable method},
	Math. Models Methods Appl. Sci. 28(2018), 733-770.
	
	\bibitem {MKID} M. F. P. ten Eikelder, K. G. van der Zee, I. Akkerman and D. Schillinger, \textit{A unified framework for Navier-Stokes Cahn-Hilliard models with non-matching densities},
	Math. Models Methods Appl. Sci., 33(2023), 175-221.
	
	\bibitem {TPT} T. P. Tsai, \textit{Lectures on Navier-Stokes Equations},  Graduate Studies in Mathematics, vol. 192. American Mathematical Society, Providence, 2018.
	
	\bibitem {TM} M. Taylor, \textit{Partial Differential Equations: I basic theory}, Applied Mathematical Sciences. Springer, New York, 2010.
\end{thebibliography}
\end{document}